\documentclass[10pt]{amsart}
\usepackage{entete}

\begin{document}
\maketitle
\begin{abstract}
%We study the possible property of an integral perverse motivic t-structure, looking at an analog in the $\ell$-adic setting. 
We show that the perverse t-structure induces a t-structure on the category $\mc{D}^A(S,\Z_\ell)$ of Artin $\ell$-adic complexes when $S$ is an excellent scheme of dimension less than $2$ and provide a counter-example in dimension $3$. The heart $\mathrm{Perv}^A(S,\Z_\ell)$ of this t-structure can be described explicitly in terms of representations in the case of $1$-dimensional schemes.

When $S$ is of finite type over a finite field, we also construct a perverse homotopy t-structure over $\mc{D}^A(S,\Q_\ell)$ and show that it is the best possible approximation of the perverse t-structure. We describe the simple objects of its heart $\mathrm{Perv}^A(S,\Q_\ell)^\#$ and show that the weightless truncation functor $\omega^0$ is t-exact. We also show that the weightless intersection complex $EC_S=\omega^0 IC_S$ is a simple Artin homotopy perverse sheaf. If $S$ is a surface, it is also a perverse sheaf but it need not be simple in the category of perverse sheaves.
\end{abstract}
\tableofcontents
\section*{Introduction}
In this paper, we aim to define abelian categories of Artin perverse sheaves. Those categories should be categories of perverse sheaves which come from finite schemes over the base scheme.
Letting $S$ be a noetherian finite dimensional scheme, we study to this end the stable $\infty$-category of \emph{Artin $\ell$-adic complexes} $\mc{D}^A(S,\Z_\ell)$ which is defined as the thick stable $\infty$-subcategory of the stable $\infty$-category $\mc{D}^b_c(S,\Z_\ell)$ of $\ell$-adic complexes generated by the complexes of the form $f_*\Z_{\ell,X}$ for $f\colon X\rar S$ finite, where $\Z_{\ell,X}$ denotes the constant sheaf with value $\Z_\ell$ over $X$. To define categories of perverse sheaves, we will define perverse t-structures on those categories.

The motivation for this problem comes from the conjectural motivic t-structure. We have at our disposal several stable $\infty$-categories of constructible étale motives $\mc{DM}_{\et,c}(-)$ endowed with Grothendieck's six functors formalism and with $\ell$-adic realization functors \cite{thesevoe,orange,ayo07,ayo14,tcmm,em}. Let $S$ be a noetherian finite-dimensional scheme. One of the major open problems regarding the stable $\infty$-category $\mc{DM}_{\et,c}(S)$ is the existence of a motivic t-structure, whose heart would be the abelian category of mixed motivic sheaves that satisfies Beilinson's conjectures \cite{Jannsen}.

More precisely, working by analogy with the stable $\infty$-category $\mc{D}^b_c(S,\Z_\ell)$, there are two possible versions of the motivic t-structure: the \emph{perverse motivic t-structure} and the \emph{ordinary motivic t-structure}. If Beilinson's conservativity conjecture holds, \textit{i.e.} if the $\ell$-adic realization functor $$\rho_\ell\colon \DM_{\et}(S,\Q)\rar \mc{D}^b_c(S,\Q_\ell)$$ is conservative, then the perverse (resp. ordinary) motivic t-structure on $\DM_{\et,c}(S)$ is characterized by the joint t-exactness of the $\ell$-adic realization functors for every prime number $\ell$ when the targets are endowed with the perverse (resp. ordinary) t-structures.

As Beilinson proved in \cite{notebeilinson} that the category $\mc{D}^b_c(S,\overline{\Q}_\ell)$ is the bounded derived category of the abelian category of perverse sheaves over $S$, the category $\mc{DM}_{\et,c}(S)$ is by analogy conjectured to be the derived category of the conjectural abelian category of \emph{perverse motives} which would be the heart of the conjectural motivic t-structure. In \cite{bondarko15}, Bondarko has shown that the perverse motivic t-structure on the stable $\infty$-category $\mc{DM}_{\et,c}(S)$ can be recovered from the ordinary motivic t-structures on the stable $\infty$-categories $\mc{DM}_{\et,c}(K)$ (assuming they exist) when $K$ runs through the set of residue fields of $S$.

Proving that the motivic t-structure exists seems completely out of reach at the moment. However, if $n=0,1$, the ordinary motivic t-structure has been constructed in \cite{orgo,ayo11,abv,bvk,plh,plh2,vaish} on the subcategory $\mc{DM}^n_{\et,c}(S,\Q)$ of \emph{constructible $n$-motives}, which is the thick stable $\infty$-subcategory generated by the cohomological motives of proper $S$-schemes of dimension less than $n$. We state their main result:

\begin{theorem*}(Voevodsky, Orgogozo, Ayoub, Barbieri-Viale, Kahn, Pepin Lehalleur, Vaish) Let $S$ be a noetherian excellent finite dimensional scheme admitting resolution of singularities by alterations. Let $\ell$ be a prime number invertible on $S$. Let $n=0,1$.

Then, the stable $\infty$-category $\mc{DM}^n_{\et,c}(S,\Q)$ can be endowed with a non-degenerate t-structure such that the $\ell$-adic realization functor:
$$\mc{DM}^n_{\et,c}(S,\Q)\rar \mc{D}^b_c(S,\Q_\ell)$$
is t-exact when the stable $\infty$-category $\mc{D}^b_c(S,\Q_\ell)$ is endowed with its ordinary t-structure.
\end{theorem*}

In the case of the dimensional subcategories $\mc{DM}^n_{\et,c}(S,\Q)$, Bondarko's approach, which defines the perverse motivic t-structure from the ordinary motivic t-structure, does not apply. It indeed requires to use Grothendieck's six functors formalism and some of the six functors do not behave well with respect to the dimensional subcategories. Therefore, the existence of the perverse motivic t-structure on the stable $\infty$-category $\mc{DM}^n_{\et,c}(S,\Q)$ does not follow formally from the existence of the ordinary motivic t-structure on the stable $\infty$-categories $\mc{DM}^n_{\et,c}(K,\Q)$ when $K$ runs through the set of residue fields of $S$.

Another natural problem is to determine whether those results can be extended to the case of étale motives with integral coefficients: Voevodsky has shown in \cite[4.3.8]{orange} that it is not possible to construct a reasonable motivic t-structure on the stable $\infty$-category of \emph{Nisnevich motives} $\mc{DM}_{\mathrm{Nis},c}(S,\Z)$; 
however, nothing is known about the t-structure on the stable $\infty$-category of étale motives.
%% However, the obstruction described by Voevodsky vanishes étale-locally \cite{??} and thus, one can hope to construct a motivic t-structure on the étale model $\mc{DM}_{\et}(S,\Z)$.

Hence, the above theorem leaves two open questions:
\begin{enumerate}\item Is there a t-structure on the stable $\infty$-category $\mc{DM}^n_{\et,c}(S,\Z)$ such that the $\ell$-adic realization functor:
$$\mc{DM}^n_{\et,c}(S,\Z)\rar \mc{D}^b_c(S,\Z_\ell)$$
is t-exact when $\mc{D}^b_c(S,\Z_\ell)$ is endowed with its \emph{ordinary} t-structure ?
\item Let $R=\Z$ or $\Q$. Can the stable $\infty$-category $\mc{DM}^n_{\et,c}(S,R)$ be endowed with a t-structure such that the $\ell$-adic realization functor:
$$\mc{DM}^n_{\et,c}(S,R)\rar \mc{D}^b_c(S,R_\ell)$$
is t-exact when $\mc{D}^b_c(S,R_\ell)$ is endowed with its \emph{perverse} t-structure ?
\end{enumerate}

One of the goals of this paper is to determine what answer to this question could reasonably be expected in the case when $n=0$, looking at the analogous situation in the setting of $\ell$-adic complexes. Namely, the stable $\infty$-category $\mc{D}^A(S,\Z_\ell)$ is the $\ell$-adic analog of the stable $\infty$-category of $0$-motives (sometimes called Artin motives). The analog of the above problem is then to determine whether the ordinary and the perverse t-structures on the stable $\infty$-category of constructible $\ell$-adic complexes induce a t-structure on the stable $\infty$-category of Artin $\ell$-adic complex. The advantage of this setting is that the ordinary and the perverse t-structure exist while they are still conjectural in the motivic world.

To investigate this problem and to state some of our results, we need to introduce the stable $\infty$-subcategory $\mc{D}^{smA}(S,\Z_\ell)$ of $\mc{D}^A(S,\Z_\ell)$ of \emph{smooth Artin $\ell$-adic complexes} which is the stable $\infty$-subcategories of $\mc{D}^b_c(S,\Z_\ell)$ generated by the $f_*\Z_{\ell,X}$, for $f\colon X\rar S$ finite and étale. This category plays in this paper the role that the category of lisse $\ell$-adic complexes plays in the theory of perverse sheaves. It is closely linked to a subclass of representations of the pro-étale fundamental group of $S$.

To explain this link, we use two ingredients. On the first hand, we need Bhatt and Scholze's construction of the category of constructible $\ell$-adic complexes: they defined it in \cite{bhatt-scholze} as the category of constructible objects of the derived category $\mc{D}(\Sh(S_{\mathrm{pro\acute{e}t}},\Z_\ell))$ of sheaves of $\Z_{\ell,S}$-modules on the pro-étale site of $S$. In the spirit of \cite[7]{bhatt-scholze}, if $\xi$ is a geometric point of $S$ and if $S$ is connected, the abelian category of continuous representations of the pro-étale fundamental group $\pi_1^{\mathrm{pro\acute{e}t}}(S,\xi)$ is equivalent to the abelian category of \emph{lisse $\ell$-adic sheaves}.

On the second hand, letting $\pi$ be a topological group, we introduce the abelian category of \emph{representations of Artin origin} of $\pi$ as the category of iterated extensions of Artin representations of $\pi$ in the category of continuous representations. We then say that a lisse $\ell$-adic sheaf is \emph{smooth Artin} if it corresponds to a representation of Artin origin of $\pi_1^{\mathrm{pro\acute{e}t}}(S,\xi)$. We can extend this definition to non-connected schemes by asking the sheaf to be smooth Artin when restricted to every connected component. We can now state the link between smooth Artin $\ell$-adic sheaves and representations more precisely:

\begin{proposition*}(\Cref{t-structure ordinaire smooth}) Let $S$ be a noetherian finite dimensional scheme and let $\ell$ be a prime number which is invertible on $S$. Then, the stable $\infty$-category $\mc{D}^{smA}(S,\Z_\ell)$ is the category of those bounded $\ell$-adic complexes whose ordinary cohomology sheaves are smooth Artin. In particular, the ordinary t-structure induces a t-structure on $\mc{D}^{smA}(S,\Z_\ell)$ whose heart is the category of smooth Artin $\ell$-adic sheaves.
\end{proposition*}

To answer our question regarding the ordinary t-structure, we introduce the notion of Artin $\ell$-adic sheaves by mimicking the definition
of constructible sheaves: an $\ell$-adic sheaf is \emph{Artin} if there is a stratification of $S$ such that its pullback to any stratum is smooth Artin.

\begin{proposition*}(\Cref{t-structure ordinaire}) Let $S$ be a noetherian finite dimensional scheme and let $\ell$ be a prime number which is invertible on $S$. Then, the category $\mc{D}^{A}(S,\Z_\ell)$ is the category of those $\ell$-adic complexes whose ordinary cohomology sheaves are Artin. In particular, the ordinary t-structure induces a t-structure on $\mc{D}^{A}(S,\Z_\ell)$ whose heart is the category of Artin $\ell$-adic sheaves.
\end{proposition*}
Our main result concerns the perverse t-structure:
\begin{theorem*}(\Cref{main theorem} and \Cref{main theorem 2}) Let $S$ be a noetherian finite dimensional excellent scheme and let $\ell$ be a prime number which is invertible on $S$. Assume that $S$ is of dimension at most $2$, or that $S$ is of dimension $3$ and that the residue fields of the closed points of $S$ are separably closed or real closed. Then, the perverse t-structure induces a t-structure on $\mc{D}^A(S,\Z_\ell)$.
\end{theorem*}
We let $\mathrm{Perv}^A(S,\Z_\ell)$ be the heart of the induced perverse t-structure on $\mc{D}^A(S,\Z_\ell)$ when it is defined.

The above result seems rather optimal: if $k$ is a finite field and $\ell$ is a prime number invertible on $k$, we prove (see \Cref{exemple dim 3}) that the perverse t-structure of $\mc{D}^b_c(\mb{A}^3_k,\Z_\ell)$ does not induce a t-structure on $\mc{D}^A(\mb{A}^3_k,\Z_\ell)$.

To circumvent this negative result, we introduce another notion of Artin perverse sheaves inspired by the approach of \cite{plh,vaish}, \cite[2.2.4]{ayo07} and \cite{bondarko-deglise}. More precisely, we define a \emph{perverse homotopy t-structure} on the category $\mc{D}^A_{\In}(S,\Z_\ell)$ of \emph{Ind-Artin $\ell$-adic complexes} formally by specifying the set of t-negative objects. It has the following properties (see \Cref{perv vs hperv}).\begin{itemize}
    \item If the perverse homotopy t-structure on the stable $\infty$-category $\mc{D}^A_{\In}(S,\Z_\ell)$ induces a t-structure on the subcategory $\mc{D}^A(S,\Z_\ell)$, then the induced t-structure is final among those such that the inclusion functor $$\mc{D}^A(S,\Z_\ell)\rar \mc{D}^b_c(S,\Z_\ell)$$ is right t-exact when the right hand side is endowed with its perverse t-structure. 
    \item  When the perverse t-structure induces a t-structure on $\mc{D}^A(S,\Z_\ell)$, then so does the perverse homotopy t-structure and both induced t-structures coincide.
\end{itemize}

The properties of this t-structure are closely related to the properties of the right adjoint $\omega^0$ of the inclusion functor $$\iota\colon\mc{D}^A_{\In}(S,\Z_\ell)\rar \mc{D}^{\coh}_{\In}(S,\Z_\ell)$$ of Ind-Artin $\ell$-adic complexes into Ind-cohomological $\ell$-adic complexes. The motivic analog of this functor was first introduced in \cite{az} to study the reductive Borel-Serre compactification of a symmetric variety.

When $S$ is of finite type over $\mb{F}_p$ and working with coefficients in $\Q_\ell$, the functor $\omega^0$ is more tractable. We have at our disposal the stable $\infty$-category of mixed $\ell$-adic sheaves $\mc{D}^b_m(S,\Q_\ell)$ of \cite[5]{bbd} and Deligne's theory of weights (see \cite{WeilII}). Using the truncation functors defined by Morel in \cite{thesemorel}, Nair and Vaish defined in \cite{nv} a \emph{weightless truncation functor} $$\omega^0 \colon \mc{D}^b_m(S,\Q_\ell)\rar \mc{D}^b_m(S,\Q_\ell)$$
(denoted by $w_{\leqslant \id}$ in \cite{nv}) which by \cite{vaish2} is the $\ell$-adic analog of Ayoub and Zucker's functor.

We show the following results:

\begin{proposition*}(\Cref{w<id vs omega0,perverse homotopy t-structure induces a t-structure} and \Cref{corollaire de w<id vs omega0,perv vs pervsharp}) Let $S$ be a scheme of finite type over $\mb{F}_p$. Let $\ell \neq p$ be a prime number.
\begin{enumerate} \item The two functors denoted above by $\omega^0$ induce equivalent functors $$\mc{D}^{\coh}(S,\Q_\ell)\rar \mc{D}^A(S,\Q_\ell).$$
In particular, both functors send cohomological $\ell$-adic complexes to Artin $\ell$-adic complexes.

\item The perverse homotopy t-structure induces a t-structure on $\mc{D}^A(S,\Q_\ell)$. We denote its heart by $\mathrm{Perv}^A(S,\Q_\ell)^\#$ and call it the \emph{abelian category of Artin homotopy perverse sheaves}. 
 
\item The abelian categories $\mathrm{Perv}^A(S,\Q_\ell)$ and $\mathrm{Perv}^A(S,\Q_\ell)^{\#}$ coincide when the perverse t-structure restricts to $\mc{D}^A(S,\Q_\ell)$.

%\item The functor $\omega^0$ is t-exact when $\mc{D}^{\coh}(S,\Q_\ell)$ is endowed with the perverse t-structure. In particular, denoting $\mathrm{Perv}^{coh}(S,\Q_\ell)$ the category of cohomological perverse sheaves, $(\iota,\omega^0)$ induces an adjunction $$\iota:\mathrm{Perv}^A(S,\Q_\ell)^\#\leftrightarrows \mathrm{Perv}^{\coh}(S,\Q_\ell):\omega^0$$ such that $\omega^0$ is exact.
\end{enumerate}
\end{proposition*}

To study the abelian category $\mathrm{Perv}^A(S,\Q_\ell)^{\#}$, we would like to show that the functor $\omega^0$ is t-exact. 
However, this t-exactness statement does not make sense as it is: it is unknown whether the perverse truncation functors preserve cohomological complexes. To get around this issue, we define the stable $\infty$-category $\mc{D}^{w\leqslant 0}(S,\Q_\ell)$ of \emph{weightless $\ell$-adic complexes} as the essential image of the weightless truncation functor. The main insight leading to this definition is that Artin $\ell$-adic complexes are the cohomological complexes which are also weightless. 
This stable $\infty$-category $\mc{D}^{w\leqslant 0}(S,\Q_\ell)$ is then endowed with a t-structure such that the inclusion of $\mc{D}^A(S,\Q_\ell)$ equipped with its perverse homotopy t-structure is t-exact (see \Cref{gluing weightless}(3)). We can now provide a meaningful t-exactness statement:

\begin{theorem*}(\Cref{omega^0 t-exact}) Let $S$ be a scheme of finite type over $\mb{F}_p$ and let $\ell\neq p$ be a prime number.
Let $$\iota\colon\mc{D}^{w\leqslant 0}(S,\Q_\ell)\rar \mc{D}^b_m(S,\Q_\ell)$$ be the inclusion. When the left hand side is endowed with the weightless perverse t-structure and the right hand side is endowed with the perverse t-structure, the adjunction $(\iota,\omega^0)$ is a t-adjunction and the functor $\omega^0$ is t-exact.
\end{theorem*}

A consequence of this result is that the abelian category $\mathrm{Perv}^A(S,\Q_\ell)^\#$ is similar to the abelian category of perverse sheaves and contains the weightless intersection complex of \cite{nv}: 
\begin{proposition*}(\Cref{simple,ECX} and \Cref{exemple ECX}) Let $S$ be a scheme of finite type over $\mb{F}_p$ and let $\ell\neq p$ be a prime number. 
\begin{enumerate} \item All the objects of $\mathrm{Perv}^A(S,\Q_\ell)^\#$ are of finite length. Its simple objects can be described as the $\omega^0j_{!*}(L[\dim(V)])$ where  $j\colon V\hookrightarrow S$ runs through the set of inclusions of smooth subschemes of $S$ and $L$ runs through the set of simple smooth Artin $\ell$-adic sheaves over $V$.
\item Let $IC_S$ be the intersection complex of $S$. The weightless intersection complex $EC_S=\omega^0 IC_S$ of \cite{nv} is a simple object of the abelian category $\mathrm{Perv}^A(S,\Q_\ell)^\#$.
\item If $S$ is a surface, then, the complex $EC_S$ is also a perverse sheaf, but it need not be simple in the category of perverse sheaves.
\end{enumerate}
\end{proposition*}

Finally, we describe the category $\mathrm{Perv}^A(C,\Z_\ell)$ when  $C$ is an excellent scheme of dimension $1$. Let $\Gamma$ be the set of generic points of $C$. Let $\nu\colon\widetilde{C}\rar C$ be the normalization of $C$. In this setting, we introduce a residue map $\partial_y$ for any closed point $y$ of $\widetilde{C}$ in \Cref{lemma not regular}. The category $\mathrm{Perv}^A(C,\Z_\ell)$ is equivalent to the category $\mathrm{P}^A(C,\Z_\ell)$ (see \Cref{P^A(S)}) which is defined as follows: 
\begin{itemize}\item An object of $\mathrm{P}^A(C,\Z_\ell)$ is a quadruple $(U,(M_\eta)_{\eta\in \Gamma},(M_x)_{x\in C \setminus U},(f_x)_{x\in C\setminus U})$ where 
\begin{enumerate}
    
    \item $U$ is a nil-regular open subset of $C$ such that the immersion $U\rar C$ is affine,
    \item for all $\eta\in \Gamma$, $M_\eta$ is a representation of Artin origin of $\pi_1^{\et}(U_\eta,\overline{\eta})$ with coefficients in $\Z_\ell$,
    \item for all $x\in C\setminus U$, $M_x$ is a complex in $\mc{D}^b_{\Z_\ell-\mathrm{perf}}(\Sh((BG_{k(x)})_{\mathrm{pro\acute{e}t}},\Z_\ell))$ (see \Cref{cas des corps}) placed in degrees $[0,1]$,
    \item for all $x\in C\setminus U$, $f_x$ is an element of $$\Hom_{{\mathrm{D}}  (\Rep(G_{k(x)},\Z_\ell))}\left(M_x,\bigoplus\limits_{\nu(y)=x}\mathrm{Ind}_{G_{k(y)}}^{G_{k(x)}}\left( \partial_y \left[\phi_y^*(M_{\eta(y)})\right]\right)\right)$$

such that $\Hl^0(f_x)$ is injective.
\end{enumerate}

\item An element of $\Hom_{\mathrm{P}^A(C,\Z_\ell)}\left((U,(M_\eta),(M_x),(f_x)),(V,(N_\eta),(N_x),(g_x))\right)$ is a couple of the form $((\Phi_\eta)_{\eta \in \Gamma},(\Phi_x)_{x \in (C\setminus U)\cap (C\setminus V)})$ where $\Phi_\eta\colon M_\eta\rar N_\eta$ is a map of representations of $G_{K_\eta}$, where $\Phi_x\colon M_x \rar N_x$ is a map in $\mc{D}(\Sh((BG_{k(x)})_{\mathrm{pro\acute{e}t}},\Z_\ell))$ and where the diagram

$$\begin{tikzcd} M_x \ar[r,"f_x"]\ar[d,"\Phi_x"]& \bigoplus\limits_{\nu(y)=x}\mathrm{Ind}_{G_{k(y)}}^{G_{k(x)}}\left(  \partial_y\left[\phi_y^*(M_{\eta(y)})\right]\right) \ar[d,"\bigoplus\limits_{\nu(y)=x}\mathrm{Ind}_{G_{k(y)}}^{G_{k(x)}}\left( \partial_y\left(\phi_y^*(\Phi_{\eta(y)})\right)\right)"]\\
N_x \ar[r,"g_x"]& \bigoplus\limits_{\nu(y)=x}\mathrm{Ind}_{G_{k(y)}}^{G_{k(x)}}\left( \partial_y\left[\phi_y^*(N_{\eta(y)})\right]\right)
\end{tikzcd}$$
is commutative.
\end{itemize} 

In the next paper, we study the motivic case and prove that we have similar results. One of the main interests of the results of this paper is that the motivic constructions will realize into our constructions.

\textbf{Acknowledgements}

This paper is part of my PhD thesis, done under the supervision of Frédéric Déglise. I would like to express my deepest gratitude to him for his constant support, for his patience and for his numerous suggestions and improvements. 

My sincere thanks also go to Joseph Ayoub, Jakob Scholbach and the anonymous referee for their very detailed reports that allowed me to greatly improve the quality of this paper.
I would also like to thank warmly Marcin Lara, Sophie Morel, Riccardo Pengo, Simon Pepin Lehalleur and Jörg Wildeshaus for their help, their kind remarks and their questions which have also enabled me to improve this work.

Finally, I would like to thank Olivier Benoist, Robin Carlier, Mattia Cavicchi, Adrien Dubouloz, Wiesława Nizioł, Fabrice Orgogozo, Timo Richarz, Wolfgang Soergel, Markus Spitzweck, Olivier Taïbi, Swann Tubach and Olivier Wittenberg for their interest in my work and for some helpful questions and advices. 

\newpage
\section{Preliminaries}
All schemes are assumed to be \textbf{noetherian} and of \textbf{finite dimension}; furthermore all smooth (and étale) morphisms and all quasi-finite morphisms are also implicitly assumed to be separated and of finite type.

In this text we will freely use the language of $(\infty,1)$-categories of \cite{htt,ha}. The term \emph{category} will by default mean $(\infty,1)$-category, and the term \emph{$2$-category} will mean $(\infty,2)$-category. When we refer to derived categories, we refer to the $\infty$-categorical version. %There is an exception to this rule: we use the term \emph{abelian category}, to refer to what according to our conventions, we should call \emph{abelian 1-category}.

There are several definitions of the derived category of constructible $\ell$-adic sheaves. The first definition was given in \cite{bbd} over schemes over a field. This definition was then extended to all noetherian finite dimensional schemes in \cite{torsten} \cite[7.2]{em}, \cite{bhatt-scholze} and \cite[XIII.4]{travauxgabber}. In this paper, we will use the model of \cite{bhatt-scholze}. We will use the theorems of \cite[4]{bbd} over excellent schemes; in particular, to use the affine Lefschetz theorem \cite[4.1.1]{bbd} and its consequences, one has to replace \cite[XIV.3.1]{sga4} in the proof of \cite[4.1.1]{bbd} with \cite[XV.1.1.2]{travauxgabber}.

We adopt the cohomological convention for t-structures (\textit{i.e} the convention of \cite[1.3.1]{bbd} and the opposite of \cite[1.2.1.1]{ha}): a t-structure on a stable category $\mc{D}$ is a pair $(\mc{D}^{\leqslant 0},\mc{D}^{\geqslant 0})$ of strictly full subcategories of $\mc{D}$ having the following properties: %$\mc{D}^{\leqslant n}=\mc{D}^{\leqslant 0}[-n]$ and $\mc{D}^{\geqslant n}=\mc{D}^{\geqslant 0}[-n]$, we have
\begin{itemize}
	\item For any object $M$ of $\mc{D}^{\leqslant 0}$ and any object $N$ of $\mc{D}^{\geqslant 0}$, $\pi_0 \Map(M,N[-1])=0$.
	\item We have inclusions $\mc{D}^{\leqslant 0} \subseteq \mc{D}^{\leqslant 0}[-1]$ and $\mc{D}^{\geqslant 0}[-1] \subseteq \mc{D}^{\geqslant 0}$.
	\item For any object $M$ of $\mc{D}$, there exists an exact triangle $M'\rar M \rar M''$ where $M'$ is an object of $\mc{D}^{\leqslant 0}$ and $M''$ is an object of $\mc{D}^{\geqslant 0}[-1]$.
\end{itemize}

A stable category endowed with a t-structure is called a t-category. If $\mc{D}$ is a t-category, we denote by $\mc{D}^\heart=\mc{D}^{\geqslant 0} \cap \mc{D}^{\leqslant 0}$ the heart of the t-structure which is an abelian category and by $\mc{D}^b$ the full subcategory of bounded objects which is a stable t-category.

Let $i\colon Z\rar X$ be a closed immersion and $j\colon U\rar X$ be the complementary open immersion. We call localization triangles the exact triangles of functors on $\ell$-adic sheaves over $X$: 
\begin{equation}\label{localization}j_!j^*\rar \id \rar i_*i^*.
\end{equation}
\begin{equation}\label{colocalization}i_!i^!\rar \id \rar j_*j^*.
\end{equation}

Let $S$ be a scheme. A \emph{stratification} of $S$ is a partition $\mc{S}$ of $S$ into non-empty equidimensional locally closed subschemes called \emph{strata} such that the topological closure of any stratum is a union strata.

If $S$ is a scheme and $\xi$ is a geometric point of $S$, we denote by $\pi_1^{\et}(S,\xi)$ the étale fundamental group of $S$ with base point $\xi$ defined in \cite{sga1} and $\pi_1^{\mathrm{pro\acute{e}t}}(S,\xi)$ the pro-étale fundamental group of $S$ with base point $\xi$ defined in \cite[7]{bhatt-scholze}.

If $k$ is a field, we denote by $G_k$ its absolute Galois group.

If $p$ is a prime number, we denote by $\mb{F}_p$ the field with $p$ elements, by $\mb{Z}_p$ the ring of $p$-adic integers, by $\Q_p$ the field of $p$-adic numbers and by $\overline{\Q}_p$  its algebraic closure.

If $\phi\colon H\rar G$ is a map of groups, we denote by $\phi^*$ the forgetful functor which maps representations of $G$ to representations of $H$. 

If $H$ is a finite index subgroup of a group $G$, we denote by $\In^G_H$ the induced representation  functor which maps representations of $H$ to representations of $G$. 

\subsection{Notions of abelian and stable subcategories generated by a set of objects}
	
	Recall the following definitions (see \cite[02MN]{stacks}).
	
	\begin{definition} Let $\mathrm{A}$ be an abelian category.
		\begin{enumerate}\item A \emph{Serre subcategory} of $\mathrm{A}$ is a nonempty full subcategory $\mathrm{B}$ of $\mathrm{A}$ such that given an exact sequence $$M' \rar M \rar M''$$ such that $M'$ and $M''$ are objects of $\mathrm{B}$, then $M$ is an object of $\mathrm{B}$.
			\item A \emph{weak Serre subcategory} of $\mathrm{A}$ is a nonempty full subcategory $\mathrm{B}$ of $\mathrm{A}$ such that given an exact sequence $$M_1 \rar M_2 \rar M_3 \rar M_4 \rar M_5$$ such that for all $i \in \{1,2,4,5\}$, $M_i$ is an object of $\mathrm{B}$, then $M_3$ is an object of $\mathrm{B}$.
			\item Let $\mc{E}$ be a set of objects of $\mathrm{A}$, the \emph{Serre (resp. weak Serre) subcategory of $\mathrm{A}$ generated by $\mc{E}$} is the smallest Serre (resp. weak Serre) subcategory of $\mathrm{A}$ whose set of objects contains $\mc{E}$. 
		\end{enumerate}
	\end{definition}

Note that a Serre subcategory of an abelian category $\mathrm{A}$ is also a weak Serre subcategory of $\mathrm{A}$ and that if $\mathrm{B}$ is a weak Serre subcategory of $\mathrm{A}$, then $\mathrm{B}$ is abelian and the inclusion functor $\mathrm{B}\rar \mathrm{A}$ is exact and fully faithful.

Recall that a strictly full subcategory is a full subcategory whose set of objects is closed under isomorphisms.

\begin{proposition}\label{A*} Let $\mathrm{B}$ be a strictly full abelian subcategory of an abelian category $\mathrm{A}$. Assume that $\mathrm{B}$ is closed under subobject and quotient in $\mathrm{A}$.
	
Let $\mathrm{B}^*$ be the full subcategory of $\mathrm{A}$ spanned by the objects obtained as successive extensions of the objects $\mathrm{B}$. Then \begin{enumerate}
	\item The category $\mathrm{B}^*$ is the smallest weak Serre subcategory of $\mathrm{A}$ containing $\mathrm{B}$.
	\item The category $\mathrm{B}^*$ is a Serre subcategory of $\mathrm{A}$.
\end{enumerate}
	
\end{proposition}
\begin{proof}
Any weak Serre subcategory containing $\mathrm{B}$ is closed under extensions and therefore contains $\mathrm{B}^*$. Hence it suffices to show that $\mathrm{B}^*$ is a Serre subcategory of $\mathrm{A}$.

The subcategory $\mathrm{B}^*$ is strictly full, contains the zero object and any extension of objects of $\mathrm{B}^*$ is an object of $\mathrm{B}^*$. Therefore using \cite[02MP]{stacks}, it suffices to show that $\mathrm{B}^*$ is closed under subobject and quotient.
	
We first define subcategories of $\mathrm{A}$ by induction. Let $\mathrm{B}_0=\mathrm{B}$ and if $n\geqslant 0$ is an integer, let $\mathrm{B}_{n+1}$ be the full subcategory of $\mathrm{A}$ whose objects are the objects $M$ of $\mathrm{A}$ such that there exists an exact sequence $$0\rar M' \rar M \rar M'' \rar 0$$ where $M'$ and $M''$ are objects of $\mathrm{B}_n$. 
	
By definition, $\mathrm{B}^*$ is the full subcategory of $\mathrm{A}$ whose objects are objects of $\mathrm{B}_n$ for some non-negative integer $n$. Therefore it suffices to show that for all $n\geqslant 0$, the subcategory $\mathrm{B}_n$ is closed under subobject and quotient.

The case where $n=0$ is the hypothesis on $\mathrm{B}$. Assume that any subobject and any quotient of any object of $\mathrm{B}_n$ lies in $\mathrm{B}_n$ and let $M$ be an object of $\mathrm{B}_{n+1}$. Let $N$ be a subobject of $M$, we show that $N$ and $M/N$ are objects of $\mathrm{B}_{n+1}$. 

We have an exact sequence $$0\rar M' \rar M \overset{\pi}{\rar} M'' \rar 0$$ where $M'$ and $M''$ are objects of $\mathrm{B}_n$. This yields exact sequences:

$$0\rar \ker(\pi|_N) \rar N \overset{\pi|_N}{\rar} \pi(N) \rar 0,$$
$$0 \rar M'/\ker(\pi|_N) \rar M/N \rar M''/\pi(N) \rar 0.$$

Since $\pi(N)$ is a subobject of $M''$, it is an object of $\mathrm{B}_n$ and since $\ker(\pi|_N)$ is a subobject of $M'$, it is also an object of $\mathrm{B}_n$. Thus $N$ is an object of $\mathrm{B}_{n+1}$. Similarly, $M/N$ is an object of $\mathrm{B}_{n+1}$.
\end{proof}
	
\begin{definition}\label{loc cat} Let $\mc{C}$ be a stable category 
\begin{enumerate}\item A \emph{thick} subcategory of $\mc{C}$ is a full subcategory $\mc{D}$ of $\mc{C}$ which is closed under finite limits, finite colimits and retracts.
\item A \emph{localizing} subcategory of $\mc{C}$ is a full subcategory $\mc{D}$ of $\mc{C}$ which is closed under finite limits and arbitrary colimits.
\item Let $\mc{E}$ be a set of objects. We call \emph{thick} (resp. \emph{localizing}) \emph{subcategory generated by} $\mc{E}$ the smallest thick (resp. localizing) subcategory of $\mc{C}$ whose set of objects contains $\mc{E}$.
	\end{enumerate}
\end{definition}

\subsection{Induced t-structures}
	Recall the following definition of \cite[1.3.19]{bbd}:
\begin{definition}
Let $\mc{D}$ be a t-category and $\mc{D}'$ be a full stable subcategory of $\mc{D}$. If  $(\mc{D}^{\leqslant 0} \cap \mc{D}', \mc{D}^{\geqslant 0} \cap \mc{D}')$ defines a t-structure on $\mc{D}'$, we say that $\mc{D'}$ is a \emph{sub-t-category} of $\mc{D}$, that the t-structure of $\mc{D}$ \emph{induces a t-structure} on $\mc{D}'$ and call the latter the \emph{induced t-structure}.
\end{definition}
First notice the following properties on sub-t-categories:
\begin{proposition}\label{premieres prop du coeur}
	Let $\mc{D}$ be a t-category and $\mc{D}_0$ be a strictly full sub-t-category of $\mc{D}^b$. Then,
	\begin{enumerate}
		\item The subcategory $\mc{D}_0$ is thick.
		\item The category $\mc{D}^\heart_0$ is a weak Serre subcategory of $\mc{D}^\heart$.
	\end{enumerate}
\end{proposition}
\begin{proof}
We first prove (1).	Since $\mc{D}_0$ is a strictly full stable subcategory of $\mc{D}$, it is closed under finite limits and colimits.
	
Let $M$ be an object of $\mc{D}_0$. Assume that $M=M'\bigoplus M''$ in $\mc{D}$. We prove that $M'$ is an object of $\mc{D}_0$.

Since $M$ is bounded, so is $M'$. Hence, by dévissage, it suffices to show that $\Hl^k(M')$ lies in $\mc{D}_0$ for any integer $k$. Therefore, we can assume that $M$ lies in $\mc{D}^\heart$.
	
	Let $p\colon M\rar M$ be the projector corresponding to $M'$. Let $C$ be the cone of $p$. Then, $M'=\ker(p)=\Hl^0(C)$ lies in $\mc{D}_0^\heart$ which proves (1).
	
We now prove (2). Since the inclusion functor $\mc{D}_0^\heart\rar \mc{D}^\heart$ is fully faithful and exact, by \cite[0754]{stacks}, it suffices to show that $\mc{D}_0^\heart$ is closed under extensions. But if $$0\rar M' \rar M \rar M''\rar 0$$ is an exact sequence in $\mc{D}^\heart$ with $M'$ and $M''$ objects of $\mc{D}^\heart_0$, the triangle $$M' \rar M \rar M''$$ is exact in $\mc{D}$. Since $\mc{D}_0$ is a stable subcategory of $\mc{D}$, this proves that $M$ is an object of $\mc{D}_0$ and thus of $\mc{D}_0^\heart$. This finishes the proof.
\end{proof}

The following lemma will allow us to restrict t-structures to thick subcategories.

\begin{lemma}\label{keur} Let $\mc{D}$ be a t-category. Let $\mc{D}'$ be a thick subcategory of $\mc{D}^b$. Assume that $\mc{D}' \cap \mc{D}^\heart$ contains a set of generators of $\mc{D}'$ as a thick subcategory of $\mc{D}$. Then, the following conditions are equivalent.
\begin{enumerate}
	\item The t-structure of $\mc{D}$ induces a t-structure on $\mc{D}'$.
	\item If $f\colon M\rar N$ is a map in $\mc{D}' \cap \mc{D}^\heart$, $\ker(f)$ is an object of $\mc{D}'$.
	\item If $f\colon M\rar N$ is a map in $\mc{D}' \cap \mc{D}^\heart$, $\coker(f)$ is an object of $\mc{D}'$.
\end{enumerate}	
\end{lemma}
\begin{proof}
	Condition (1) implies conditions (2) and (3). If $f\colon M\rar N$ is a map in $\mc{D}' \cap \mc{D}^\heart$ and $C$ is the cone of $f$, we have an exact triangle:
	$$\ker(f)\rar C \rar \coker(f)[-1].$$
	Thus, conditions (2) and (3) are equivalent. 
	
	Assume now that conditions (2) and (3) hold. First notice that $\mc{D'} \cap \mc{D}^\heart$ is closed under extensions since $\mc{D}'$ is. Therefore, the subcategory $\mc{D'} \cap \mc{D}^\heart$ of $\mc{D}^\heart$ is weak Serre using \cite[0754]{stacks}.
	
	We consider the full subcategory $\mc{C}$ of $\mc{D}'$ made of those objects $M$ such that for all integer $n$, the object $\Hl^n(M)$ lies in $\mc{D}'$. The subcategory $\mc{C}$ is closed under suspensions and direct factors. We now show that it is closed under extensions. 
 
 Let $$M'\rar M \rar M''$$ be an exact triangle in $\mc{D'}$ such that $M'$ and $M''$ are objects of $\mc{C}$. Then, if $n$ is an integer, we get an exact sequence in $\mc{D}^\heart$:
	$$\Hl^{n-1}(M'') \rar \Hl^n(M') \rar \Hl^n(M)\rar \Hl^n(M'') \rar \Hl^{n+1}(M').$$
	
	Since  $\mc{D'} \cap \mc{D}^\heart$ is a weak Serre subcategory of $\mc{D}^\heart $ and since $M'$ and $M''$ are objects of $\mc{C}$, the object $\Hl^n(M)$ lies in $\mc{D}'$.
	
	Therefore, the subcategory $\mc{C}$ is thick and contains a set of generators of $\mc{D'}$, whence, we have $\mc{C}=\mc{D}'$. This shows that the subcategory $\mc{D}'$ of $\mc{D}^b$ is stable under the truncation functor $\tau_{\geqslant 0}$, thus by \cite[1.3.19]{bbd}, the t-structure of $\mc{D}$ induces a t-structure on $\mc{D}'$.
\end{proof}

In fact, \Cref{keur} implies that weak Serre subcategories of the heart of a t-category on the one hand and strictly full bounded sub-t-categories of the same t-category on the other hand are in one-to-one correspondence. More precisely, we define the following particular sub-t-categories:

\begin{definition}\label{D^b_A} Let $\mc{D}$ be a t-category. Let $\mathrm{A}$ be a weak Serre subcategory of $\mc{D}^{\heart}$. We will denote by $\mc{D}^b_\mathrm{A}$ the full subcategory of $\mc{D}$ spanned by those bounded objects $C$ of $\mc{D}$ such that for all integer $n$, the object $\Hl^n(C)$ is in $\mathrm{A}$.

\end{definition}

\begin{remark}
Be careful that in general, the category $\mc{D}^b_\mathrm{A}$ need not coincide with the category $\mc{D}^b(\mathrm{A})$. However, by \cite[7.60]{cisinski-bunke}, there is a canonical t-exact functor: 
$$\mc{D}^b(\mathrm{A})\rar \mc{D}^b_\mathrm{A}.$$
\end{remark}

We then have the following proposition:
\begin{proposition}\label{D^b_A description} Let $\mc{D}$ be a t-category. Then, 
	\begin{enumerate}
		\item Let $\mathrm{A}$ be a weak Serre subcategory of $\mc{D}^\heart$. Then, the category $\mc{D}^b_{\mathrm{A}}$ is a strictly full sub-t-category of $\mc{D}^b$. Moreover, we have $(\mc{D}^b_\mathrm{A})^\heart=\mathrm{A}$ as subcategories of $\mc{D}^\heart$.
		\item Let $\mc{D}_0$ be a strictly full sub-t-category of $\mc{D}^b$. Then, the category $\mc{D}_0^\heart$ is a weak Serre subcategory of $\mc{D}^\heart$ and $\mc{D}^b_{\mc{D}_0^\heart}=\mc{D}_0$ as subcategories of $\mc{D}$.
	\end{enumerate}
\end{proposition}
\begin{proof}
	The first assertion follows from \Cref{keur} and from the fact that $\mathrm{A}$ generates $\mc{D}^b_{\mathrm{A}}$ as a thick subcategory of $\mc{D}^b$.
	
The first part of the second assertion follows from \Cref{premieres prop du coeur}. 

Finally, the cohomology objects of any object of $\mc{D}_0$ lie in $\mc{D}_0^{\heart}$. Thus, any object of $\mc{D}_0$ lies in $\mc{D}^b_{\mc{D}_0^\heart}$. Moreover, if $M$ is an object of $\mc{D}^b_{\mc{D}_0^\heart}$, to show that it lies in $\mc{D}_0$, it suffices by dévissage to show that for all $n$, the object $\Hl^n(M)$ is an object of $\mc{D}_0$. This is true by definition.
\end{proof}

\subsection{Finite permutation resolutions}
In this paragraph, we recall a result of Balmer and Gallauer (see \cite{balmer-gallauer}). If $G$ is a finite group, $R$ is a ring and $H$ is a subgroup of $G$, the $R$-module $R[G/H]$ is endowed with a left $R[G]$-module structure letting $G$ act on $G/H$ by left multiplication.

\begin{proposition}\label{resolution de permutation} Let $G$ be a finite group. Let $R$ be a regular ring and let $M$ be a representation of $G$ with coefficients in $R$. Then, the complex $M \oplus M[1]$ of $R[G]$-modules is quasi-isomorphic as a complex of $R[G]$-modules to a complex $$\cdots \rar 0 \rar P_n \rar P_{n-1} \rar \cdots \rar P_0 \rar 0 \rar \cdots$$ where each $P_i$ is placed in degree $-i$ and is isomorphic to $R[G/H_1]\oplus \cdots \oplus R[G/H_m]$ for some subgroups $H_1,\ldots,H_m$ of $G$.
\end{proposition}
\begin{proof} We say that a bounded complex $P$ of $R[G]$-modules is a \emph{permutation complex} if the $P_i$ are isomorphic to $R[G/H_1]\oplus \cdots \oplus R[G/H_m]$ for some subgroups $H_1,\ldots,H_m$ of $G$.

Let $\mc{P}(R,G)$ be the subcategory of $\mc{D}^b(R[G])$ made of those complexes $N$ which admit a quasi-isomorphism $P\rar N$, where $P$ is permutation complex. Then, \cite[4.6]{balmer-gallauer} states that the idempotent completion of $\mc{P}(R,G)$ is the category $\mc{D}^b(R[G])$ itself. 

Moreover, using \cite[2.13]{balmer-gallauer}, the idempotent completion of $\mc{P}(R,G)$ is the subcategory made of those complexes $N$ such that $N\oplus N[1]$ lies in $\mc{P}(R,G)$.  Thus, the complex $M \oplus M[1]$ admits a quasi-isomorphism $P\rar M \oplus M[1]$, where $P$ is permutation complex.

Finally, noting that $M \oplus M[1]$ lies in nonpositive degree, we can modify $P$ so that the $P_i$ vanish for $i>0$ by \cite[2.4]{balmer-gallauer} which yields the result.
\end{proof}

\subsection{Artin representations and representations of Artin origin}

Let $\ell$ be a prime number. The topology on $\Z_\ell$ induces a topology on every $\Z_\ell$-module of finite type. 

In this paragraph and in the following paragraphs, we only consider the case of $\Z_\ell$-coefficients for simplicity; however, all the definitions can be formulated with coefficients in $\Q_\ell$ (or even a finite extension of $\Q_\ell$, its ring of integers,  $\overline{\Q}_\ell$ or the subring of elements of $\overline{\Q}_\ell$ which are integral over $\Z_\ell$) and all the propositions still hold.

Recall the following definitions:
\begin{definition} Let $\pi$ be a topological group and let $\ell$ be a prime number. 
	\begin{enumerate}
		\item A \emph{continuous representation} of $\pi$ with coefficients in $\Z_\ell$ is a $\Z_\ell$-module of finite type endowed with a continuous action of $\pi$. We will denote by $\mathrm{Rep}(\pi,\Z_\ell)$ the abelian category of continuous representations of $\pi$ with coefficients in $\Z_\ell$.
		\item An \emph{Artin representation} of $\pi$  is a continuous representation of $\pi$ which factors through a finite quotient of $\pi$. We will denote by $\mathrm{Rep}^A(\pi,\Z_\ell)$ the abelian category of Artin representation of $\pi$ with coefficients in $\Z_\ell$.
	\end{enumerate}
\end{definition}

The category $\Rep^A(\pi,\Z_\ell)$ is a strictly full abelian subcategory of $\Rep(\pi,\Z_\ell)$; in general, it is not a weak Serre subcategory. Indeed, consider an additive character $\chi\colon \pi\rar \Z_\ell$ that does not factor through a finite quotient (for example if $\pi=\hat{\Z}$, write $\hat{\Z}=\prod \limits_{p \text{ prime}} \Z_p$ and take the projection to $\Z_\ell$). Then, we have a morphism $\pi \rar GL_2(\Z_\ell)$ given by $$\begin{pmatrix} 1 & \chi \\ 0 & 1 \end{pmatrix}$$
which defines a continuous representation $M$ of $\pi$ on $\Z_\ell^2$. The exact sequence:
$$0\rar \Z_\ell \rar M \rar \Z_\ell \rar 0$$ is an exact sequence of representations if $\pi$ acts trivially on both copies of $\Z_\ell$. Hence, this gives an extension of Artin representations which is not Artin.

This leads us to the following definition:

\begin{definition} Let $\pi$ be a topological group and let $\ell$ be a prime number. A continuous representation is \emph{of Artin origin} if it is obtained as a successive extension of Artin representations in $\Rep(\pi,\Z_\ell)$. We denote by $\Rep^A(\pi,\Z_\ell)^*$ the category of representations of Artin origin.
\end{definition}

\Cref{A*} implies the following proposition.

\begin{proposition}\label{Artin origin} Let $\pi$ be a topological group and let $\ell$ be a prime number. Then,
	\begin{enumerate}
		\item The smallest weak Serre subcategory of  $\Rep(\pi,\Z_\ell)$ containing $\Rep^A(\pi,\Z_\ell)$ is the category $\Rep^A(\pi,\Z_\ell)^*$.
		\item The category $\Rep^A(\pi,\Z_\ell)^*$ is a Serre subcategory of $\Rep(\pi,\Z_\ell)$.
		\end{enumerate}
\end{proposition}

The following proposition shows that the Artin origin condition can be checked with coefficients in $\Q_\ell$.

\begin{proposition}\label{passage coeffs Ql rep} Let $\pi$ be a topological group. Let $M$ be a continuous representation with coefficients in $\Z_\ell$. Assume that $M\otimes_{\Z_\ell} \Q_\ell$ is a representation of Artin origin with coefficients in $\Q_\ell$. Then, $M$ itself is of Artin origin.
\end{proposition}
\begin{proof}
%	Assume first that $M\otimes_{\Z_\ell} \Q_\ell$ is an Artin representation. 
Let $\mathrm{A}_0=\Rep^A(\pi,\Q_\ell)$ and if $n$ is a non-negative integer, let $\mathrm{A}_{n+1}$ be the full subcategory of $\Rep(\pi,\Q_\ell)$ whose objects are the extensions of the objects of $\mathrm{A}_n$ in $\Rep(\pi,\Q_\ell)$. 
There is a non-negative integer $n$ such that $M\otimes_{\Z_\ell} \Q_\ell$ is an object of $\mathrm{A}_n$. 

Let $n$ be a non-negative integer. We prove by induction on $n$ that if $M$ is a continuous representation with coefficients in $\Z_\ell$ such that $M\otimes_{\Z_\ell} \Q_\ell$ lies in $\mathrm{A}_n$, then $M$ is of Artin origin. 

Notice first that we can assume $M$ to be torsion free. 
Indeed, let $M_{\mathrm{tors}}$ be the sub-$\Z_\ell$-module of torsion elements of $M$. 
Then, the representation on $M$ induces a sub-representation on $M_{\mathrm{tors}}$. 
Moreover, we have an exact sequence of representations:
$$0\rar M_{\mathrm{tors}} \rar M \rar M/M_{\mathrm{tors}}\rar 0.$$
	
Now, the abelian group $M_{\mathrm{tors}}$ is finite. 
Therefore, the representation on $M_{\mathrm{tors}}$ is an Artin representation. 
Moreover, we have $$M/M_{\mathrm{tors}}\otimes_{\Z_\ell} \Q_\ell =M \otimes_{\Z_\ell} \Q_\ell.$$ 
Thus, we can assume that $M$ is torsion free.

Now, if, $n=0$, the representation $M\otimes_{\Z_\ell} \Q_\ell$ is an Artin representation. Since $M$ is torsion free, the group homomorphism 
$$\mathrm{Aut}_{\Z_\ell}(M) \rar \mathrm{GL}_{\Q_\ell}(M\otimes_{\Z_\ell} \Q_\ell)$$
is injective. 
Therefore, the morphism $\pi\rar \mathrm{Aut}_{\Z_\ell}(M)$ which gives rise to the $\Z_\ell[\pi]$-module structure of $M$ factors through a finite quotient and thus $M$ is an Artin representation. 

Assume now that $n>0$. We get an exact sequence:
$$0\rar N' \overset{i}{\rar} M\otimes_{\Z_\ell} \Q_\ell \overset{p}{\rar} N'' \rar 0$$
such that $N'$ and $N''$ are objects of $\mathrm{A}_{n-1}$. 

Let $M'$ (resp. $M''$) be the inverse image (resp. the image) of $M$ along the map $i$ (resp. through the map $p$).

There is a $\Q_\ell$-basis of $N'$ whose image lies in $M$. 
Therefore, $$N'=M' \otimes_{\Z_\ell} \Q_\ell.$$ 

Furthermore, there is a $\Q_\ell$-basis of $N''$ which lies in the image of $M$. 
Thus, $$N''=M''\otimes_{\Z_\ell} \Q_\ell.$$ 

Hence, by induction, $M'$ and $M''$ are of Artin origin.

The map $i$ and $p$ induce maps $i_0\colon M'\rar M$ and $p_0\colon  M\rar M''$ such that $p_0 \circ i_0=0$. Letting $K$ be the kernel of $p_0$, the map $i_0$ factors through $K$. 

Moreover, the rank of $K$ as a $\Z_\ell$-module is $\dim_{\Q_\ell}(M)-\dim_{\Q_{\ell}}(M'')$ which is also the rank of $M'$. This yields two exact sequences:
$$0\rar K \rar M \rar M'' \rar 0$$
$$0 \rar M' \rar K \rar F \rar 0$$
where $F$ is of rank $0$.

Since $F$ is finite as an abelian group, it is an Artin representation. Thus, $K$ is of Artin origin and finally $M$ is of Artin origin.
\end{proof}

Artin representations are contained in a more explicit Serre subcategory of $\Rep(\pi,\Z_\ell)$ which we will call \emph{strongly of weight $0$}. To prove that a representation is not of Artin origin, it is often convenient to prove that it is not strongly of weight $0$.

\begin{definition} Let $\pi$ be a topological group and $\ell$ be a prime number. Let $M$ be a continuous representation of $\pi$ with coefficients in $\Z_\ell$. We say that $M$ is \emph{strongly of weight $0$} if for all $g\in \pi$, the eigenvalues of the matrix defined by the action of $g$ on $M\otimes_{\Z_\ell} \overline{\Q_\ell}$ are roots of unity. We will denote by $\Rep^{w=0}(\pi,\Z_\ell)$ the abelian category of representations that are strongly of weight $0$.
\end{definition}

\begin{remark}\label{weight 0} In the theory of weights developed by Deligne in \cite{WeilII}, letting $k$ be a finite field and letting $\pi=G_k=\hat{\Z}$ be its absolute Galois group, an $\ell$-adic representation of $G_k$ is of weight $0$ if the action of the topological generator $F$ of $G_k$ are algebraic numbers such that all their conjugates are of absolute value $1$ as complex numbers (\textit{i.e.} are roots of unity by a theorem of Kronecker). Hence, representations that are strongly of weight $0$ of $G_k$ are contained in representations of weight $0$ of $G_k$. 
%
%Finally, note that by a theorem of Kronecker, an algebraic integer such that all its conjugates are of absolute value $1$ is a root of unity.
\end{remark}

\begin{proposition}\label{strongly of weight 0} Let $\pi$ be a topological group and $\ell$ be a prime number. The abelian category $\Rep^{w=0}(\pi,\Z_\ell)$ is a Serre subcategory of $\Rep(\pi,\Z_\ell)$ and contains the category representations of Artin origin. 
\end{proposition}
\begin{proof}First, any sub-representation and any quotient of a representation that is strongly of weight $0$ is strongly of weight $0$. Now, if $$0\rar M' \rar M \rar M'' \rar 0$$ is an exact sequence of continuous representations where $M'$ and $M''$ are strongly of weight $0$, the sequence $$0\rar M'\otimes_{\Z_\ell} \overline{\Q_\ell} \rar M \otimes_{\Z_\ell} \overline{\Q_\ell} \rar M'' \otimes_{\Z_\ell} \overline{\Q_\ell}\rar 0$$
	is still exact and if $g$ is an element of $\pi$, its action on $M$ is the action of a matrix of the form
	$A=\begin{bmatrix} A'&* \\ 0 & A''\end{bmatrix}$ where $A'$ (resp. $A''$) is the matrix of the action of $g$ on $M'$ (resp. $M''$). The set of eigenvalues of $A$ is the union of the set of eigenvalues of $A'$ and $A''$ and those are roots of unity. Hence, the category $\Rep^{w=0}(\pi,\Z_\ell)$ is a Serre subcategory of $\Rep(\pi,\Z_\ell)$.
	
	Now, it suffices to show that Artin representations are strongly of weight $0$. Let $M$ be an Artin representation of $\pi$. Then, the representation $M\otimes_{\Z_\ell} \overline{\Q}_\ell$ is an Artin representation of $\pi$ with coefficients in $\overline{\Q}_\ell$. Since it factors through a finite quotient $G$ of $\pi$, if $g\in \pi$, the element $g^{|G|}$ of $\pi$ acts trivially on $M\otimes_{\Z_\ell} \overline{\Q}_\ell$. 
\end{proof}

\subsection{Categories of \texorpdfstring{$\ell$}{\textell}-adic sheaves and complexes}
In this paragraph, we recall the definitions of \cite{bhatt-scholze,hrs} of various $\ell$-adic categories. If $S$ is a scheme, denote by $S_{\mathrm{pro\acute{e}t}}$ the pro-étale site of $S$ introduced in \cite[4]{bhatt-scholze}. 
%If $*$ the set with one element, we denote $*_{\mathrm{pro\acute{e}t}}$ the category of profinite sets with covers given by finite families of jointly surjective maps.

%There is a map of sites $$p_X:X_{\mathrm{pro\acute{e}t}}\rar *_{\mathrm{pro\acute{e}t}}$$ given by the limit preserving functor $p_X^{-1}S=\lim S_i \mapsto \lim X^{S_i}$.

Any totally disconnected topological space $M$ induces a sheaf on $S_{\mathrm{pro\acute{e}t}}$: $$U \mapsto \mathcal{C}^0(U,M)=\mathcal{C}^0(\pi_0(U),M).$$ we denote this sheaf by $M_S$ and call it the constant sheaf of value $M$. If $M=\Lambda$ is a totally disconnected topological ring, then the sheaf $\Lambda_S$ is a sheaf of rings on $S_{\mathrm{pro\acute{e}t}}$.

\begin{definition} Let $S$ be a scheme (noetherian by conventions) and $\ell$ be a prime number invertible on $S$. We define:
	\begin{enumerate}
		\item The abelian category of \emph{$\ell$-adic sheaves} $\Sh(S,\Z_\ell)$ as the abelian category of sheaves of $\Z_{\ell,S}$-modules on $S_{\mathrm{pro\acute{e}t}}$.
		\item The abelian category of \emph{lisse $\ell$-adic sheaves} $\mathrm{Loc}_S(\Z_\ell)$ as the category of sheaves of $\Z_{\ell,S}$-modules $L$ such that there exists a pro-étale cover $\{S_i\rar S\}$ such that for all $i$, there exists a $\Z_\ell$-module of finite type $M$ such that $L|_{S_i}$ is isomorphic to $M_{S_i}$.
		\item The abelian category of \emph{constructible $\ell$-adic sheaves} $\Sh_c(S,\Z_\ell)$ as the category of sheaves of $\Z_{\ell,S}$-modules $L$ such that there exists a stratification of $S$ such that for any stratum $T$, the sheaf $L|_T$ is lisse.
		\item The \emph{derived category of $\ell$-adic sheaves} $\mc{D}(S,\Z_\ell)$ as the derived category of $\Sh(S,\Z_\ell)$.
		\item The \emph{stable category of lisse $\ell$-adic sheaves} $\mc{D}_{\mathrm{lisse}}(S,\Z_\ell)$ as category of dualizable objects of $\mc{D}(S,\Z_\ell)$.
		\item The \emph{derived category of constructible $\ell$-adic sheaves} $\mc{D}^b_c(S,\Z_\ell)$ as the category of complexes of sheaves of $\Z_{\ell,S}$-modules $C$ such that there exists a stratification of $S$ such that for any stratum $T$, the sheaf $C|_T$ is lisse.
	\end{enumerate}
\end{definition}
\begin{remark} Beware that in \cite{bhatt-scholze}, they denote by $\mathrm{Loc}_S(\Z_\ell)$ the category of locally free sheaves of $\Z_{\ell,S}$-modules of finite rank. We want the category of lisse sheaves to be abelian, hence we chose to include all sheaves that are locally of finite type. This is also the convention of \cite{sga4.5}.	
\end{remark}

These categories are related to each other by the following proposition.
\begin{proposition}(\cite[3.28]{hrs}) Let $S$ be a scheme and $\ell$ be a prime number invertible on $S$. The canonical t-structure of $\mc{D}(S,\Z_\ell)$ induces a t-structure on $\mc{D}^b_c(S,\Z_\ell)$ (resp. $\mc{D}_{\mathrm{lisse}}(S,\Z_\ell)$) whose heart is $\Sh_c(S,\Z_\ell)$ (resp. $\mathrm{Loc_S(\Z_\ell)}$).
\end{proposition}

The above proposition shows in particular that $\Sh_c(S,\Z_\ell)$ and $\mathrm{Loc_S(\Z_\ell)}$ are weak Serre subcategories of $\Sh(S,\Z_\ell)$. Furthermore, we have the following characterization of lisse sheaves:

\begin{proposition}\label{description des faisceaux l-adiques lisses} Let $S$ be a connected scheme and $\ell$ be a prime number invertible on $S$. Let $\xi$ be a geometric point of $S$. The fiber functor associated to $\xi$ induces an equivalence of abelian monoidal categories $$\xi^*\colon \mathrm{Loc}_S(\Z_\ell)\rar \Rep(\pi_1^{\mathrm{pro\acute{e}t}}(S,\xi),\Z_\ell)$$
where $\pi_1^{\mathrm{pro\acute{e}t}}(S,\xi)$ is the pro-étale fundamental group of $S$ with base point $\xi$ of \cite[7.4.2]{bhatt-scholze}
\end{proposition}
\begin{remark} The proposition above with coefficients in $\Z_\ell$ still holds with the étale fundamental group by \cite[Chapitre 2, 2.4]{sga4.5}; we need the pro-étale fundamental group only in the case when the ring of coefficients is an algebraic extension of $\Q_\ell$. In this setting, the proposition does not follow from the proof below. However, it is a consequence of \cite[7.4.7,7.4.10]{bhatt-scholze}. 
\end{remark}
\begin{proof} Let $\mathrm{Loc}_S$ be the infinite Galois category (see \cite[7.2.1]{bhatt-scholze}) of locally constant sheaves of sets on $S_{\mathrm{pro\acute{e}t}}$. By definition of the pro-étale fundamental group, the fiber functor $$\xi^*\colon \mathrm{Loc}_S \rar \pi_1^{\mathrm{pro\acute{e}t}}(S,\xi)-\mathrm{Set}$$ is an equivalence.

The constant sheaf $(\Z/\ell^n \Z)_{S}$  is sent to the set $\Z/\ell^n\Z$ endowed with the trivial action of $\pi_1^{\mathrm{pro\acute{e}t}}(S,\xi)$. Thus, $\xi^*$ induces an equivalence $$\mathrm{Loc}_S(\Z/\ell^n\Z)\rar \Rep(\pi_1^{\mathrm{pro\acute{e}t}}(S,\xi),\Z/\ell^n\Z).$$

The proposition follows by passing to the limit in the 2-category of categories.
\end{proof}

Recall that Bhatt and Scholze defined using the theory of infinite Galois categories of \cite[7.3]{bhatt-scholze} a map $$\pi_1^{\mathrm{pro\acute{e}t}}(S,\xi)\rar \pi_1^{\et}(S,\xi)$$ that exhibits $\pi_1^{\et}(S,\xi)$ as the profinite completion of $\pi_1^{\mathrm{pro\acute{e}t}}(S,\xi)$ by \cite[7.4.3]{bhatt-scholze}, and which is a group isomorphism when $S$ is geometrically unibranch. We now claim that the formation of $\pi_1^{\mathrm{pro\acute{e}t}}(S,\xi)$ can be upgraded into a functor with respect to $(S,\xi)$ which is compatible with the map above.

\begin{lemma}\label{6 functors lisse 1} Let $f\colon T \rar S$ be a map of schemes. Let $\xi'\colon \Spec(\Omega)\rar T$ be a geometric point of $T$ and let $\xi=f\circ \xi'$. Then, 
	\begin{enumerate} 
		\item The map $f$ induces a continuous map of groups $$\pi_1^{\mathrm{pro\acute{e}t}}(f)\colon \pi_1^{\mathrm{pro\acute{e}t}}(T,\xi')\rar \pi_1^{\mathrm{pro\acute{e}t}}(S,\xi)$$ such that $\pi_1^{\mathrm{pro\acute{e}t}}$ is a functor and the map $\pi_1^{\mathrm{pro\acute{e}t}}\rar \pi_1^{\et}$ is a natural transformation. 
	    \item Assume that $f$ is finite and étale, then, the map $\pi_1^{\mathrm{pro\acute{e}t}}(f)$ is injective and its image is a finite index subgroup.
        \item Let $\ell$ be a prime number which is invertible on $S$. Then, the square
		$$\begin{tikzcd}\mathrm{Loc}_S(\Z_\ell)\ar[r,"f^*"]\ar[d,"\xi^*"]&  \mathrm{Loc}_T(\Z_\ell)\ar[d]\ar[d,"(\xi')^*"] \\
		\Rep(\pi_1^{\mathrm{pro\acute{e}t}}(S,\xi),\Z_\ell) \ar[r,"\pi_1^{\mathrm{pro\acute{e}t}}(f)^*"]&
		\Rep(\pi_1^{\mathrm{pro\acute{e}t}}(T,\xi'),\Z_\ell)
		\end{tikzcd}$$
		is commutative.
\end{enumerate}
\end{lemma}
\begin{proof}
	If $X$ is a scheme, let $\mathrm{Loc}_X$ (resp. $\mathrm{Loc}_{X}^{\et}$) be the infinite Galois category (see \cite[7.2.1]{bhatt-scholze}) of locally constant sheaves of sets on $X_{\mathrm{pro\acute{e}t}}$ (resp. $X_{\et}$). 
	
	Now, if $f\colon T\rar S$ is a map of schemes, let ${\mathrm{Loc}}_{f}$ (resp.${\mathrm{Loc}}^{\et}_f$) be the map $f^*\colon \mathrm{Loc}_S \rar \mathrm{Loc}_T$. This defines contravariant functors $\mathrm{Loc}$ and $\mathrm{Loc}^{\et}$ from the category of schemes to the category of infinite Galois category. Moreover, the map of Bhatt and Scholze between the pro-étale and the étale fundamental group is induced by the natural map $\mathrm{Loc}^{\et}\rar \mathrm{Loc}$. Thus, the first assertion follows from the structure theorem of infinite Galois categories \cite[7.2.5]{bhatt-scholze}.
	
	Now, by definition of the map $\pi_1^{\mathrm{pro\acute{e}t}}(f)$, the square:
	$$\begin{tikzcd}\mathrm{Loc}_S\ar[r,"f^*"]\ar[d,"\xi^*"]&  \mathrm{Loc}_T\ar[d]\ar[d,"(\xi')^*"] \\
		\pi_1^{\mathrm{pro\acute{e}t}}(S,\xi)-\text{Set} \ar[r,"\pi_1^{\mathrm{pro\acute{e}t}}(f)^*"]&
		\pi_1^{\mathrm{pro\acute{e}t}}(T,\xi')-\text{Set}
	\end{tikzcd}$$
is commutative which implies the third assertion.

%To prove (3), we can assume that $S$ and $T$ are connected. First notice that the action of $\pi_1^{\mathrm{pro\acute{e}t}}(S,\xi)$ on $f^{-1}(\xi)$ is transitive. Indeed, the action of $\pi_1^{\mathrm{pro\acute{e}t}}(S,\xi)$ factors through a finite quotient and therefore it comes from the usual action of $\pi_1^{\mathrm{ét}}(S,\xi)$ which is transitive by \cite[V.7]{sga1}.

We now prove the second assertion. We have a functor $f_*\colon \mathrm{Loc}_T\rar \mathrm{Loc}_S$ which is both left and right adjoint to $f^*$ and sends a locally constant sheaf $L$ to the sheaf $X \mapsto L(X\times_S T)$. Moreover, if $L$ is a locally constant sheaf on $S$, the map $L\rar f_* f^* L$ is injective. Therefore, using \cite[2.37]{lara}, the map $\pi_1^{\mathrm{pro\acute{e}t}}(f)$ is a topological embedding.

Now to prove that the quotient is finite, we can assume that $T$ is connected. Via the Yoneda embedding, the $S$-scheme $T$ corresponds to a locally constant proétale sheaf of sets $M_T$ which itself corresponds to a $\pi^{\mathrm{pro\acute{e}t}}(S,\xi)$-set $E:=\xi^*(M_T)$ which is the fiber over the point $\xi$. As $T$ is connected, the $\pi^{\mathrm{pro\acute{e}t}}(S,\xi)$-set $E$ is transitive. Furthermore, the geometric point $\xi'$ corresponds to a point $t$ of $E$. As the action of $\pi^{\mathrm{pro\acute{e}t}}(S,\xi)$ on $E$ is transitive, we have $$E \simeq \pi^{\mathrm{pro\acute{e}t}}(S,\xi)/\mathrm{Stab}(t)$$ and therefore $U:=\mathrm{Stab}(t)$ is a finite index subgroup of $\pi^{\mathrm{pro\acute{e}t}}(S,\xi)$. We now claim that the image of $\pi_1^{\mathrm{pro\acute{e}t}}(T,\xi')$ through $\pi_1^{\mathrm{pro\acute{e}t}}(f)$ is $U$ which will finish the proof.

First, the image of  in $\pi_1^{\mathrm{pro\acute{e}t}}(S,\xi)$, when acting on $E$, stabilizes $t$. Indeed, the $\pi_1^{\mathrm{pro\acute{e}t}}(T,\xi')$-set $E$ corresponds through $(\xi')^*$ to the locally constant sheaf $f^*(M_T)$ which corresponds to the proétale cover $T\times_S T$. The point $t$ then corresponds to a point of $T\times_S T$ which belongs to the image of the diagonal maps $T\rar T \times_S T$ of $T$. As $T$ corresponds to the trivial one point $\pi_1^{\mathrm{pro\acute{e}t}}(T,\xi')$-set, the action of $\pi_1^{\mathrm{pro\acute{e}t}}(T,\xi')$ on $t$ is trivial.

Finally, the map $\pi_1^{\mathrm{pro\acute{e}t}}(T,\xi')\rar \mathrm{Stab}(t)$ has a section. Indeed, the elements of $\mathrm{Stab}(t)$ can be seen as automorphisms of the fiber functor $\xi^*$ on the category of proétale covers of $S$ which respect the fiber over $\xi'$. The action of those elements can be restricted to automorphisms of the fiber functor $\xi'$ which yields the desired section. %\pi^proet(Y,\by) = Aut_{Cov_Y}(ev_\by). So let’s take a Z -> Y in Cov_Y. We want to define an action of Stab_s on ev_\by (functorially in Z). But Stab_s \subset G acts on ev_\bx(Z -> Y -> X) and by the definiton of s, it respects the fiber over \by. So one then uses it to cook up an action on ev_\by(Z -> Y).
%So we got a surjectivity on U = Stab_s.
\end{proof}
\begin{remark}
    The proof of (2) of the above result was communicated to the author by M. Lara. His proof also yields that $$\xi^*(M_T) \simeq \pi^{\mathrm{pro\acute{e}t}}(S,\xi)/\mathrm{Stab}(t)$$ for any proétale cover $T$ of $S$.
\end{remark}

\begin{corollary}\label{6 functors lisse 2} Let $f\colon T \rar S$ be a finite étale map and let $\ell$ be a prime number which is invertible on $S$. Let $\xi'$ be a geometric point of $T$ and let $\xi=f\circ \xi'$. Then, the functor $f_*\colon \Sh(T,\Z_\ell)\rar \Sh(S,\Z_\ell)$ induces a map $\mathrm{Loc}_T(\Z_\ell)\rar \mathrm{Loc}_S(\Z_\ell)$ and the square:
		$$\begin{tikzcd}\mathrm{Loc}_T(\Z_\ell)\ar[rr,"f_*"]\ar[d,"(\xi')^*"]&& \mathrm{Loc}_S(\Z_\ell)\ar[d]\ar[d,"\xi^*"] \\
			\Rep(\pi_1^{\mathrm{pro\acute{e}t}}(T,\xi'),\Z_\ell) \ar[rr,"\mathrm{Ind}_{\pi'}^{\pi}"]&&
			\Rep(\pi_1^{\mathrm{pro\acute{e}t}}(S,\xi),\Z_\ell)
		\end{tikzcd}$$
	where $\pi'=\pi_1^{\mathrm{pro\acute{e}t}}(T,\xi')$ and $\pi=\pi_1^{\mathrm{pro\acute{e}t}}(S,\xi)$ is commutative.
\end{corollary}
\begin{proof}
	The functor $f_*$ is right adjoint to the functor $f^*$. On the other hand, the functor $\mathrm{Ind}_{\pi'}^{\pi}$ is right adjoint to the functor $\pi_1^{\mathrm{pro\acute{e}t}}(f)^*$ which yields the result.
 \end{proof}
	
Finally, recall that the fibered category $\mc{D}^b_c(-,\Z_\ell)$ over the category of noetherian excellent schemes is endowed with the 6 functors formalism (see  \cite[6.7]{bhatt-scholze}). 

\begin{remark}\label{cas des corps}
    If $k$ is a field, the topos $\widetilde{k}_{\mathrm{pro\acute{e}t}}$ is equivalent to the topos $\widetilde{(BG_k)}_{\mathrm{pro\acute{e}t}}$ where $(BG_k)_{\mathrm{pro\acute{e}t}}$ is the site of profinite continuous $G_k$-sets with covers given by continous surjections (see \cite[4.1.10]{bhatt-scholze}). 
    Therefore, the abelian category $\Sh(k_{\mathrm{pro\acute{e}t}},\Z_\ell)$ identifies with the abelian category $\Sh((BG_k)_{\mathrm{pro\acute{e}t}},\Z_\ell)$ of modules over the profinite $G_k$-set $\Z_\ell$ endowed with trivial action (in the case of coefficients in $\Q_\ell$, the topological ring $\Q_\ell$ with its trivial action still defines a sheaf of rings on $(BG_k)_{\mathrm{pro\acute{e}t}}$ using \cite[4.3.2]{bhatt-scholze}).

    Finally, the stable category $\mc{D}^b_c(k,\Z_\ell)$ identifies with the stable subcategory of \emph{$\Z_\ell$-perfect complexes} $\mc{D}^b_{\Z_\ell-\mathrm{perf}}(\Sh((BG_k)_{\mathrm{pro\acute{e}t}},\Z_\ell))$ of $\mc{D}^b(\Sh((BG_k)_{\mathrm{pro\acute{e}t}},\Z_\ell))$ which is the full subcategory made of those complexes whose restriction to $\mc{D}^b(\Z_\ell)$ is perfect (\textit{i.e.} whose cohomology sheaves are finite when seen as $\Z_\ell$-modules).

    The heart of $\mc{D}^b_c(k,\Z_\ell)$  then identifies with $\Rep(G_k,\Z_\ell)$.
\end{remark}

\subsection{Reminders on cdh-descent and Mayer-Vietoris for closed immersions}
\begin{definition}\label{cdh-square} We say following \cite[2]{voe10} and \cite[2.1.11]{tcmm} that a cartesian square of schemes: $$\begin{tikzcd}E \ar[r]\ar[d] & F \ar[d,"i"]\\
		Y \ar[r,"f"] & X
	\end{tikzcd}$$
	is \emph{cdh-distinguished} if $f$ is proper and $i$ is a closed immersion such that, letting $U$ be its open complement in $X$, the map $f^{-1}(U)\rar U$ induced by $f$ is an isomorphism.
\end{definition}

The definition of \cite[7.2]{em} of $\ell$-adic complexes (which is equivalent to the definition of \cite{bhatt-scholze} by \cite[7.2.21]{em}) yields the following:
\begin{proposition}(cdh-descent) \label{cdh-descent} Let $\ell$ be a prime number. A cdh-distinguished square of $\Z[1/\ell]$-schemes: $$\begin{tikzcd}E \ar[r,"p"]\ar[d,"i_E"] & F \ar[d,"i"]\\
		Y \ar[r,"f"] & X
	\end{tikzcd}$$
	induces an exact triangle:
	
	$$\Z_{\ell,X}\rar f_*\Z_{\ell,Y}\oplus i_* \Z_{\ell,F}\rar i_* p_* \Z_{\ell,E}.$$
\end{proposition}

The following formulas are consequences of \cite[5.1.7]{ddo}. They are closely linked to the Rapoport-Zink complex \cite[2.5]{rz} and are generalizations of the Mayer-Vietoris exact triangle for closed immersions \cite[2.2.31]{ayo07}.

\begin{proposition}\label{Mayer-Vietoris} Let $i\colon E\rar Y$ be a closed immersion and let $\ell$ be a prime number invertible on $Y$. Assume that $E=\bigcup\limits_{i \in I} E_i$ with $I$ a finite set and $(E_i)_{i \in I}$ a family of closed subschemes of $E$. If $J\subseteq I$, write $E_J=\bigcap\limits_{i \in J} E_i$ and write $i_J\colon E_J\rar Y$ the canonical closed immersion. Then, the canonical maps $(i_J)_* i_J^! \Z_{\ell,Y}\rar i_*i^!\Z_{\ell,Y}$ exhibit  $i_*i^!\Z_{\ell,Y}$ as the colimit
	
$$\colim\limits_{p \in (\Delta^{\mathrm{inj}})^{op}} \left(\bigoplus\limits_{\substack{J \subseteq I \\ |J|=p+1}} (i_J)_* i_J^! \Z_{\ell,Y}\right).$$ where $\Delta^{\mathrm{inj}}$ is the category of finite ordered sets with morphisms the injective maps.
\end{proposition}

\begin{proposition}\label{Mayer-Vietoris2} Let $i\colon E\rar Y$ be a closed immersion and let $\ell$ be a prime number invertible on $Y$. Assume that $E=\bigcup\limits_{i \in I} E_i$ with $I$ a finite set and $(E_i)_{i \in I}$ a family of closed subschemes of $E$. If $J\subseteq I$, write $E_J=\bigcap\limits_{i \in J} E_i$ and write $i_J\colon E_J\rar Y$ the canonical closed immersion. Then, the canonical maps $i_*i^*\Z_{\ell,Y}\rar (i_J)_* i_J^* \Z_{\ell,Y}$ exhibit  $i_*i^*\Z_{\ell,Y}$ as the limit
	
$$\lim\limits_{p \in (\Delta^{\mathrm{inj}})^{op}} \left(\bigoplus\limits_{\substack{J \subseteq I \\ |J|=p+1}} (i_J)_* i_J^* \Z_{\ell,Y}\right).$$ where $\Delta^{\mathrm{inj}}$ is the category of finite ordered sets with morphisms the injective maps.
\end{proposition}

\subsection{Perverse sheaves} In this paragraph, we recall the definition of the perverse t-structure over base scheme endowed with a dimension function following \cite{notesgabber}. First recall the definition of dimension functions (see for instance \cite[1.1.1]{bondarko-deglise}):
\begin{definition}
    Let $S$ be a scheme. A \emph{dimension function} on $S$ is a function $\delta\colon S\rar \Z$ such that for any immediate specialization $y$ of a point $x$ (\textit{i.e} a specialization such that $\codim_{\overline{\{ x\}}} (y)=1$), we have $\delta(x)=\delta(y)+1$.
\end{definition}

Two dimension functions on a scheme differ by a Zariski-locally constant function. 
Moreover, if $S$ is universally catenary and integral, the formula $$\delta(x)=\dim(X)-\codim_X(x)$$ gives a dimension function on $S$.
Finally, any $S$-scheme of finite type $f\colon X\rar S$ inherits a canonical dimension function by setting 
$$\delta(x)=\delta(f(x))+\mathrm{tr.deg}\hspace{1mm} k(x)/k(f(x)).$$

The auto-dual perversity $p_{1/2}\colon Z\mapsto -\delta(Z)$ induces two t-structures $[p_{1/2}]$ and $[p_{1/2}^+]$ on $\mc{D}^b_c(S,\Z_\ell)$. In the thesis, we only consider the t-structure $[p_{1/2}]$. We now recall its definition over a general base scheme. 

\begin{definition} (Compare with \cite[2]{notesgabber}) 
Let $S$ be a scheme endowed with a dimension function and $\ell$ be a prime number invertible on $S$. 
For every point $x\in S$, let $i_x\colon \{ x\}\rar S$ be the inclusion.
We define
\begin{itemize}
\item the full subcategory ${}^{\mathrm{p}} \mc{D}^{\leqslant 0}(S,\Z_\ell)$ as the subcategory of those complexes $K$ such that for any point $x\in S$, we have $i_x^*K\leqslant p_{1/2}(\overline{\{ x\}})=-\delta(x)$ with respect to the ordinary t-structure of $\mc{D}(k(x),\Z_\ell)$.
\item the full subcategory ${}^{\mathrm{p}} \mc{D}^{\geqslant 0}(S,\Z_\ell)$ as the subcategory of those complexes $K$ such that for any point $x\in S$, we have $i_x^!K\geqslant p_{1/2}(\overline{\{ x\}})=-\delta(x)$ with respect to the ordinary t-structure of $\mc{D}(k(x),\Z_\ell)$.
\end{itemize}
\end{definition}

In the case when the ring of coefficients is $\Z/\ell\Z$, \cite[6]{notesgabber} implies that this defines a t-structure on $\mc{D}(S,\Z/\ell\Z)$ and \cite[8.2]{notesgabber} implies that $\mc{D}^b_c(S,\Z/\ell\Z)$ is a sub-t-category of $\mc{D}(S,\Z/\ell\Z)$. The proofs of \cite{notesgabber} actually apply with coefficients in $\Z_\ell$ which defines a t-structure on $\mc{D}^b_c(S,\Z_\ell)$. 

Alternatively, when the scheme $S$ is assumed to be excellent, this is a t-structure for the same reasons as in \cite[2.2.9-2.2.17]{bbd} (see \cite[2.2]{morel2} for a more precise explanation; the proof uses the absolute purity property \cite[XVI.3.1]{travauxgabber}). The following result is an important step of the latter proof and will be used in the rest of the paper. %This alternative definition yields the following:

\begin{proposition}\label{t-structure perverse lisse} Let $S$ be a regular connected scheme endowed with a dimension function $\delta$ and $\ell$ be a prime number invertible on $S$. The perverse t-structure of $\mc{D}^b_c(S,\Z_\ell)$ induces a t-structure on $\mc{D}_{\mathrm{lisse}}(S,\Z_\ell)$ which coincides with the ordinary t-structure shifted by $\delta(S)$. In particular, the heart of this t-structure is $\mathrm{Loc}_S(\Z_\ell)[\delta(S)]$.
\end{proposition}
\begin{proof} Since the category $\mc{D}_{\mathrm{lisse}}(S,\Z_\ell)$ is bounded with respect to the shifted ordinary t-structure, by dévissage it suffices to show that the objects of $\mathrm{Loc}_S(\Z_\ell)$ are in degree $\delta(S)$ with respect to the perverse t-structure. 

Let $L$ be a locally constant sheaf of $\Z_\ell$-modules on $S$ and let $x$ be a point of $S$. The sheaf $i_x^*L$ is in degree $0\leqslant \delta(S)-\delta(x)$. Furthermore, since the scheme $S$ is regular and connected, by the absolute purity property \cite[XVI.3.1]{travauxgabber}, we have $$i_x^!L=i_x^*L(-c)[-2c]$$ where $c=\codim_S(x)=\delta(S)-\delta(x)$. Therefore, the sheaf $i_x^!L$ is in degree $2c\geqslant \delta(S)-\delta(x)$. Hence, the sheaf $L$ is in degree $\delta(S)$ with respect to the perverse t-structure which finishes the proof. 
\end{proof}

If $n$ is an integer, we will denote by $\Hlp^n$ the $n$-th cohomology functor with respect to the perverse t-structure. 
\begin{definition} Let $S$ be a scheme endowed with a dimension function and $\ell$ be a prime number that is invertible on $S$. 
The \emph{category $\mathrm{Perv}(S,\Z_\ell)$ of perverse sheaves} is the heart of the perverse t-structure on $\mc{D}^b_c(S,\Z_\ell)$.
\end{definition}

Finally, note that \Cref{t-structure perverse lisse,premieres prop du coeur} imply that $\mathrm{Loc}_S(\Z_\ell)[\delta(S)]$ is a weak Serre subcategory of $\mathrm{Perv}(S,\Z_\ell)$ when the scheme $S$ is regular and connected.

\section{Artin \texorpdfstring{$\ell$}{\textell}-adic complexes}
In this section, we define abelian categories of Artin $\ell$-adic sheaves and stable categories of Artin $\ell$-adic complexes. We show that the ordinary t-structure on the derived category of $\ell$-adic sheaves induces a t-structure on the category of Artin $\ell$-adic complexes such that the heart of this t-structure is the category of Artin $\ell$-adic sheaves.
\subsection{Artin \texorpdfstring{$\ell$}{\textell}-adic sheaves and complexes}

\begin{definition} Let $S$ be a scheme and $\ell$ a prime number that is invertible on $S$. 
	\begin{enumerate}\item We define the abelian category $\Sh^{smA}(S,\Z_\ell)$  of \emph{smooth Artin $\ell$-adic sheaves} over $S$ as the Serre subcategory of $\mathrm{Loc}_S(\Z_\ell)$ made of those objects $M$ such that for any geometric point $\xi$ of $S$, the representation $\xi^*M$ of $\pi_1^{\mathrm{pro\acute{e}t}}(S,\xi)$ is of Artin origin.
	\item We define the abelian category $\Sh^A(S,\Z_\ell)$  of \emph{Artin $\ell$-adic sheaves} over $S$ as the category of constructible $\ell$-adic sheaves such that there exists a stratification of $S$ such that for any stratum $T$, the $\ell$-adic sheaf $L|_T$ is smooth Artin.
	\end{enumerate}
\end{definition}
\begin{remark} Let $S$ be a connected scheme and $\ell$ a prime number invertible on $S$. Let $\xi$ be a geometric point of $S$. Then, the functor $\xi^*$ induces an equivalence from $\Sh^{smA}(S,\Z_\ell)$ to $\Rep^A(\pi_1^{\mathrm{pro\acute{e}t}}(S,\xi),\Z_\ell)^*$ by \Cref{description des faisceaux l-adiques lisses}.
\end{remark}

\begin{definition}\label{def l-adiques} Let $S$ be a scheme and $\ell$ a prime number which is invertible on $S$.
	\begin{enumerate}\item The stable category $\mc{D}^A(S,\Z_\ell)$ of \emph{Artin $\ell$-adic complexes} is the thick stable subcategory of the stable category $\mc{D}^b_c(S,\Z_\ell)$ generated by the complexes of the form $f_*\Z_{\ell,X}$ with $f\colon X\rar S$ finite.
		\item The stable category $\mc{D}^{smA}(S,\Z_\ell)$ of \emph{smooth Artin $\ell$-adic complexes} over $S$ with coefficients in $\Z_\ell$ is the thick stable subcategory of the stable category $\mc{D}^b_c(S,\Z_\ell)$ generated by the complexes of the form $f_*\Z_{\ell,X}$ with $f\colon X\rar S$ étale and finite.
		\item The stable category $\mc{D}^A_{\In}(S,\Z_\ell)$ (resp. $\mc{D}^{smA}_{\In}(S,\Z_\ell)$) is the localizing subcategory of the stable category $\mc{D}(S,\Z_\ell)$ with the generators as in (1) (resp. (2)).
	\end{enumerate}
\end{definition}

\begin{proposition}\label{presentability l-adique} Let $S$ be a scheme and $\ell$ a prime number which is invertible on $S$. Then, the stable category $\mc{D}^A_{\In}(S,\Z_\ell)$ is presentable.    
\end{proposition}
\begin{proof} This follows from \cite[1.4.4.11(2)]{ha}.
\end{proof}

\subsection{Description of smooth Artin \texorpdfstring{$\ell$}{\textell}-adic complexes}

\begin{proposition}\label{t-structure ordinaire smooth} Let $S$ be a scheme and $\ell$ a prime number which is invertible on $S$. The ordinary t-structure of $\mc{D}(S,\Z_\ell)$ induces a t-structure on $\mc{D}^{smA}(S,\Z_\ell)$. The heart of this t-structure is equivalent to $\Sh^{smA}(S,\Z_\ell)$. In particular, the subcategory $\mc{D}^{smA}(S,\Z_\ell)$ is the subcategory made of those bounded complexes $C$ of $\ell$-adic sheaves such that for all integer $n$, the sheaf $\Hl^n(C)$ is smooth Artin.
\end{proposition}
\begin{proof} We can assume that $S$ is connected.
	By \Cref{keur} and \Cref{D^b_A description}, it suffices to show that $$\mc{D}^{smA}(S,\Z_\ell) \cap \Sh(S,\Z_\ell)=\Sh^{smA}(S,\Z_\ell).$$
	
	Let $f\colon X\rar S$ be étale and finite. Recall that $f_*$ is t-exact with respect to the ordinary t-structure of $\mc{D}(S,\Z_\ell)$. Therefore, using \Cref{6 functors  lisse 2}, the object $f_*\Z_{\ell,X}$  lies in $\Sh^{smA}(S,\Z_\ell)$.
	
	Thus, by \Cref{D^b_A description,premieres prop du coeur,Artin origin}, the category $\mc{D}^{smA}(S,\Z_\ell)$ is included in $\mc{D}^b(S,\Z_\ell)_{\Sh^{smA}(S,\Z_\ell)}$ (see \Cref{D^b_A}). Thus, we have $$\mc{D}^{smA}(S,\Z_\ell) \cap \Sh(S,\Z_\ell)\subseteq \Sh^{smA}(S,\Z_\ell).$$
	
	Take now $M$ a smooth Artin $\ell$-adic sheaf. It corresponds to a representation of Artin origin of $\pi_1^{\mathrm{pro\acute{e}t}}(S,\xi)$ that we still denote by $M$. 
	
	Assume first that $M$ is an Artin representation. Let $G$ be a finite quotient of $\pi_1^{\mathrm{pro\acute{e}t}}(S,\xi)$ through which $M$ factors. Then, by \Cref{resolution de permutation} the complex $M \oplus M[1]$ of $\Z_\ell[G]$-modules is quasi-isomorphic as a complex of $\Z_\ell[G]$-modules to a complex $$P=\cdots \rar 0 \rar P_n \rar P_{n-1} \rar \cdots \rar P_0 \rar 0 \rar \cdots$$ where each $P_i$ is placed in degree $-i$ and is isomorphic to $\Z_\ell[G/H_1]\oplus \cdots \oplus \Z_\ell[G/H_m]$ for some subgroups $H_1,\ldots,H_m$ of $G$.
	
	The exact functor $$\Rep(G,\Z_\ell)\rar \Rep(\pi_1^{\mathrm{pro\acute{e}t}}(S,\xi),\Z_\ell) \rar \Sh(S,\Z_\ell)$$ induces an exact functor $$\mc{D}^b(\Rep(G,\Z_\ell))\rar \mc{D}(S,\Z_\ell).$$
	
	If $H$ is a subgroup of $G$, the finite set $G/H$ is endowed with an action of $\pi_1^{\mathrm{pro\acute{e}t}}(S,\xi)$ and therefore corresponds to a finite locally constant sheaf on $S_{\et}$, and therefore to a finite étale $S$-scheme $f\colon X\rar S$. The image of $\Z_\ell[G/H]$ through the functor above is then $f_*\Z_{\ell,X}$.
	
	Thus, the object $M\oplus M[1]$, and therefore the object $M$ itself, is in the thick category generated by the $f_*\Z_{\ell,X}$ with $f\colon X\rar S$ finite and étale, which is $\mc{D}^{smA}(S,\Z_\ell)$.
	
	Now, since $\mc{D}^{smA}(S,\Z_\ell)$ is closed under extensions, any representation of Artin origin of $\pi_1^{\mathrm{pro\acute{e}t}}(S,\xi)$
	lies in $\mc{D}^{smA}(S,\Z_\ell)$ which finishes the proof.
\end{proof}

\begin{corollary}\label{t-structure ordinaire smooth fields} If $k$ is a field and $\ell$ is a prime number that is distinct from the characteristic of $k$, then, the category $\mc{D}^A(k,\Z_\ell)$ is equivalent to the category of the complexes $C$ in $\mc{D}^b(\Sh((BG_k)_{\mathrm{pro\acute{e}t}},\Z_\ell))$ (see \Cref{cas des corps}) such that the sheaf $\Hl^n(C)$ identifies with a representation of Artin origin for all $n$.
\end{corollary}
\subsection{The trace of the six functors on Artin \texorpdfstring{$\ell$}{\textell}-adic complexes}
Artin $\ell$-adic complexes are stable under some of the six functors of constructible $\ell$-adic complexes. The results of this paragraph could also have been formulated with the categories $\mc{D}^A_{\In}(-,\Z_\ell)$, replacing thick subcategories with localizing subcategories when needed. We will first show the following lemma which was proved by Ayoub and Zucker in the motivic setting and in the case of schemes of finite type over a field of characteristic $0$ and was generalized to arbitrary schemes by Pepin Lehalleur.

\begin{proposition}\label{generators} Let $S$ be a scheme and $\ell$ a prime number which is invertible on $S$. The category $\mc{D}^A(S,\Z_\ell)$ is the thick subcategory of $\mc{D}^b_c(S,\Z_\ell)$ generated by any of the following families of objects:
	
	\begin{enumerate}	\item $f_* \Z_{\ell,X}$ for $f\colon X\rar S$ finite.
		\item $f_! \Z_{\ell,X}$ for $f\colon X\rar S$ étale.
		\item $f_! \Z_{\ell,X}$ for $f\colon X\rar S$ quasi-finite.
	\end{enumerate}
\end{proposition}
\begin{proof} We follow the proof of \cite[1.28]{plh}.
	
Let $j\colon U\rar S$ be an open immersion. Let $i\colon F\rar S$ be the reduced complementary closed immersion. By localization \eqref{localization}, we have an exact triangle $$j_!\Z_{\ell,U}\rar \Z_{\ell,S}\rar i_*\Z_{\ell,F}.$$
	
Thus, $j_!\Z_{\ell,U}$ is an Artin $\ell$-adic complex. 
	
Now, let $f\colon X\rar S$ be quasi-finite. Using Zariski's Main Theorem \cite[18.12.13]{ega4} and the case of open immersions, we get that $f_!\Z_{\ell,X}$ is an Artin $\ell$-adic complex.
	
Let $\mc{C}(S)$ be the thick subcategory of $\mc{D}^b_c(S,\Z_\ell)$ generated by the $f_!\Z_{\ell,X}$ with $f\colon X\rar S$ étale. To prove the proposition, it suffices to show that if $g\colon Y\rar S$ is finite, $g_*\Z_{\ell,Y}$ is in $\mc{C}(S)$. 

Notice that the fibered category $\mc{C}(-)$ is closed under the functor $p^*$ for any morphism $p$ and the functor $q_!$ for any quasi-finite morphism $q$.

Furthermore, when $g$ is a universal homeomorphism, we have $g_*\Z_{\ell,Y}=\Z_{\ell,S}$ by topological invariance of the small étale site (see \cite[04DY]{stacks}).
	
We now prove this claim by noetherian induction on $S$. We can assume that $Y$ is reduced using the topological invariance of the small étale site once again, we have $$g_*\Z_{\ell,Y}=g'_*\Z_{\ell,Y_{\mathrm{red}}}$$ where $Y_{\mathrm{red}}$ is the reduction of $Y$ and $g'\colon Y_{\mathrm{red}}\rar S$ is the canonical morphism. 

Moreover, there is an open immersion $j\colon U\rar S$ such that the pullback map $g_U\colon Y\times_S U\rar U$ is the composition of a finite étale map and a purely inseparable morphism. Thus, $j^*g_*\Z_{\ell,Y}=(g_U)_*\Z_{\ell,Y\times_S U}$ lies in $\mc{C}(U)$. 

Furthermore, if $i\colon Z\rar S$ is the reduced complementary immersion, the complex $i^*g_*\Z_{\ell,Y}$ is in $\mc{C}(Z)$ by noetherian induction and proper base change. Finally, by localization, it suffices to prove the lemma below.
\end{proof}

\begin{lemma} Keep the same notations, then, the functor $i_*$ sends $\mc{C}(Z)$ to $\mc{C}(S)$.
\end{lemma}
\begin{proof} We follow the proof of \cite[1.18(ii), 1.14(iii)]{plh}. It suffices to show that if $f\colon X\rar Z$ is étale, $i_*f_!\Z_{\ell,X}$ lies in $\mc{C}(S)$. As in the proof of \cite[1.14]{plh} we use Mayer-Vietoris exact triangles to reduce to the case where $S=\Spec(A)$, $Z=\Spec(A/I)$ and $X=\Spec(B)$ where $A$ is a ring, $I$ is an ideal of $A$ and $B$ is a standard étale $A/I$-algebra, \textit{i.e.} $B$ is of the form $$B=(A/I)[x_1,\ldots,x_n]/(P_1,\ldots,P_n)$$ and the determinant of the matrix $(\frac{\partial P_i}{\partial x_j})_{1\leqslant i,j\leqslant n}$ is invertible in $A/I$. 
	
We can then lift the equations $P_i$ to $\widetilde{P}_i\in A[x_1,\ldots,x_n]$. 
The open subset $V$ of $S$ over which the map $$\widetilde{f}\colon X'=\Spec(A[x_1,\ldots,x_n]/(\widetilde{P}_1,\ldots,\widetilde{P}_n)) \rar \Spec(A)=S$$ is étale contains $Z$ and $\widetilde{f}$ extends $f$. Since the functor $(V\rar S)_!$ sends $\mc{C}(V)$ to $\mc{C}(S)$, we can replace $S$ with $V$ and therefore assume that $\widetilde{f}$ is étale.

We then have a localization triangle:
$$j_! j^*\widetilde{f}_!\Z_{\ell,X'} \rar \widetilde{f}_!\Z_{\ell,X'} \rar i_* f_!\Z_{\ell,X}.$$

Therefore, $i_* f_!\Z_{\ell,X}$ lies in $\mc{C}(S)$ as wanted.
\end{proof}

\begin{corollary}\label{6 foncteurs} Let $\ell$ be a prime number. The six functors from $\mc{D}(-,\Z_\ell)$ induce by restriction the following functors on $\mc{D}^A(-,\Z_\ell)$:
	\begin{enumerate}\item $f^*$ for any morphism $f$,
		\item $f_!$ for any quasi-finite morphism $f$,
		\item A monoidal structure $\otimes$.
	\end{enumerate}
\end{corollary} 

\begin{remark}Notice that on the categories $\mc{D}^A_{\In}(-,\Z_\ell)$, we have formal right adjoints to the three functors that we have described.
\end{remark}

In the case of smooth Artin $\ell$-adic complexes, we also have functors $\otimes$, $f^*$ for any morphism $f$ and $g_*=g_!$ for any finite étale morphism $g$. Since $f^*$ and $g_*$ are t-exact, they induce functors on the hearts. Furthermore, the monoidal structure on $\mc{D}^{smA}(S,\Z_\ell)$ induces a monoidal structure on its heart: map two objects $M$ and $N$ to $\Hl^0(M\otimes N)$. Using \Cref{6 functors lisse 1,6 functors lisse 2,t-structure ordinaire smooth}, we identify those functors with classical constructions on categories of representations:
\begin{proposition}\label{6 foncteurs rpz} Let $f\colon T\rar S$ be a map between connected schemes and let $\ell$ be a prime number which is invertible on $S$. Let $\xi'$ be a geometric point of $T$ and let $\xi=f\circ \xi'$. Then,
	\begin{enumerate}
		\item The functor $f^*\colon \mc{D}^{smA}(S,\Z_\ell)\rar \mc{D}^{smA}(T,\Z_\ell)$ induces the forgetful functor $$\pi_1^{\mathrm{pro\acute{e}t}}(f)^*\colon \Rep^A(\pi_1^{\mathrm{pro\acute{e}t}}(S,\xi),\Z_\ell)^*\rar\Rep^A(\pi_1^{\mathrm{pro\acute{e}t}}(T,\xi'),\Z_\ell)^*$$ by restriction.
		\item Furthermore, if $f$ is finite and étale, the map induced by $f_*$ by restriction is the induced representation functor $\mathrm{Ind}_{\pi'}^{\pi}$ where $\pi=\pi_1^{\mathrm{pro\acute{e}t}}(S,\xi)$ and $\pi'=\pi_1^{\mathrm{pro\acute{e}t}}(T,\xi')$.
		\item The monoidal structure of $\mc{D}^{smA}(S,\Z_\ell)$ induces by restriction a monoidal structure on $\Rep^A(\pi_1^{\mathrm{pro\acute{e}t}}(S),\Z_\ell)^*$ which is the usual monoidal structure.
	\end{enumerate}
\end{proposition}

\subsection{Artin \texorpdfstring{$\ell$}{\textell}-adic complexes as stratified complexes of representations}
In this paragraph, we show that for any Artin $\ell$-adic complex, there is a stratification of the base scheme such that the restriction of the complex to any stratum is smooth Artin. We then deduce that the ordinary t-structure of $\mc{D}(S,\Z_\ell)$ induces a t-structure on $\mc{D}^A(S,\Z_\ell)$ whose heart is $\Sh^A(S,\Z_\ell)$.

\begin{proposition}\label{stratification} Let $S$ be a scheme, $\ell$ be a prime number that is invertible on $S$ and $M$ be an $\ell$-adic complex. Then, the complex $M$ is an Artin $\ell$-adic complex if and only if there is a stratification of $S$ such that for any stratum $T$, the complex $M|_T$ is a smooth Artin $\ell$-adic complex.
\end{proposition}

\begin{proof} To prove the "only if" part of the proposition, proceed by induction on the number of strata and use \Cref{6 foncteurs}. We now prove the converse  starting with a lemma:
	\begin{lemma}\label{ouvert dense lisse} Let $S$ be a scheme, $\ell$ be a prime number that is invertible on $S$ and $M$ be an Artin $\ell$-adic complex. Then, there is a dense open subscheme $U$ of $S$ such that $M|_U$ is smooth Artin.
	\end{lemma} 
	\begin{proof} Let $\mc{C}$ be the full subcategory of $\mc{D}^A(S,\Z_\ell)$ made of those objects $M$ such that there is a dense open subscheme $U$ of $S$ such that $M|_U$ is smooth Artin. Then, the subcategory $\mc{C}$ is closed under finite (co)limits and retracts and is therefore thick.
		
		Furthermore, if $f\colon W\rar S$ is étale, there is a dense open immersion $j\colon U\rar S$ such that the pull-back morphism $g\colon V\rar U$ is finite étale. 
		
		Therefore, the complex $$(f_!\Z_{\ell,W})|_U=j^*f_!\Z_{\ell,W}=g_!\Z_{\ell,V}$$ is smooth Artin. Thus, the complex $f_!\Z_{\ell,W}$ lies in $\mc{C}$. Hence, using the description of generators of \Cref{generators}, the category $\mc{C}$ is equivalent to $\mc{D}^A(S,\Z_\ell)$.
	\end{proof}
	
	To finish the proof, we proceed by noetherian induction on $S$. Let $M$ be an Artin $\ell$-adic sheaf over $S$. 
	
	The lemma gives a dense open subscheme $U$ of $S$ such that $M|_U$ is smooth Artin. Using the induction hypothesis, there is a stratification $\mc{S}_F$ of $F=S\setminus U$ such that the restriction of $M|_F$ to any stratum of $\mc{S}_F$ is smooth Artin.
\end{proof}

\begin{corollary}\label{t-structure ordinaire}
	Let $S$ be a scheme, $\ell$ be a prime number that is invertible on $S$.  The canonical t-structure of $\mc{D}(S,\Z_\ell)$ induces a t-structure on $\mc{D}^{A}(S,\Z_\ell)$. The heart of this t-structure is equivalent to $\Sh^{A}(S,\Z_\ell)$. In particular, the subcategory $\mc{D}^{A}(S,\Z_\ell)$ is the subcategory of $\mc{D}^b_c(S,\Z_\ell)$ made of those complexes $C$ of $\ell$-adic sheaves such that for all integer $n$, the sheaf $\Hl^n(C)$ is Artin.
\end{corollary}
\begin{proof} Let $M$ be an Artin $\ell$-adic complex. Let $\tau_{\geqslant 0}$ be the truncation functor with respect to the ordinary t-structure on $\mc{D}(S,\Z_\ell)$. To show that the canonical t-structure of $\mc{D}(S,\Z_\ell)$ induces a t-structure on $\mc{D}^{A}(S,\Z_\ell)$, it suffices to show that $\tau_{\geqslant 0}M$ is Artin. 
	
Let $\mc{S}$ be a stratification of $S$ such that for all $T\in \mc{S}$, the complex $M|_T$ is smooth Artin. Since the pullback to $T$ functor is t-exact, we have $(\tau_{\geqslant 0}M)|_T=\tau_{\geqslant 0}(M|_T)$. The latter is smooth Artin by \Cref{t-structure ordinaire smooth}. Therefore, the complex $\tau_{\geqslant 0} M$ is Artin. 
	
Furthermore, using \Cref{t-structure ordinaire smooth,stratification}, we get that any Artin $\ell$-adic sheaf is an Artin $\ell$-adic complex, in other words, we have $$\Sh^{A}(S,\Z_\ell)\subseteq \mc{D}^{A}(S,\Z_\ell).$$ 

Finally, if $M$ is an object of $\mc{D}^{A}(S,\Z_\ell)^\heart$, there exists a stratification $\mc{S}$ of $S$ such that for all $T\in \mc{S}$, the complex $M|_T$ is smooth Artin and thus, it is a smooth Artin $\ell$-adic sheaf.
\end{proof}

\begin{corollary}\label{passage coeffs Ql}	Let $S$ be a scheme, $\ell$ be a prime number that is invertible on $S$. Let $M$ be a constructible $\ell$-adic complex with coefficients in $\Z_\ell$. Then, the $\ell$-adic complex $M$ is Artin if and only if $M\otimes_{\Z_\ell}\Q_\ell$ is Artin as an $\ell$-adic complex with coefficients in $\Q_\ell$.
\end{corollary}
\begin{proof} Since $M$ is constructible, there is a stratification $\mc{S}$ of $S$ such that for all $T\in \mc{S}$, the complex $M|_T$ is lisse.  Thus, for any integer $n$, the sheaf $\Hl^n(M|_T)$ is lisse. 
	
Refining the stratification if needed, we can assume the strata to be connected. Let $\xi_T$ be a geometric point of $T$. Then, for any integer $n$, the sheaf $\Hl^n(M|_T)$ corresponds through the equivalence of \Cref{description des faisceaux l-adiques lisses} to an object $E^n(M,T)$ of $\Rep(\pi_1^{\mathrm{pro\acute{e}t}}(T,\xi_T),\Z_\ell)$.

Since the functor $-\otimes_{\Z_\ell} \Q_\ell$ is t-exact, we have $$\Hl^n(M|_T)\otimes_{\Z_\ell} \Q_\ell=\Hl^n((M \otimes_{\Z_\ell} \Q_\ell)|_T)$$ for any integer $n$.

Since $M \otimes_{\Z_\ell} \Q_\ell$ is Artin, refining the stratification further, we can assume the $\ell$-adic sheaves $\Hl^n((M \otimes_{\Z_\ell} \Q_\ell)|_T)$ to be smooth Artin for any integer $n$ \Cref{t-structure ordinaire}. 

Therefore, the representation $E^n(M,T)\otimes_{\Z_\ell} \Q_\ell$ is of Artin origin for any integer $n$. Thus, by \Cref{passage coeffs Ql rep}, the representation $E^n(M,T)$ belongs for any integer $n$ to the abelian category $\Rep^A(\pi_1^{\mathrm{pro\acute{e}t}}(T,\xi_T),\Z_\ell)^*$. Therefore, by \Cref{t-structure ordinaire}, the complex $M$ is Artin.
\end{proof}

\section{The perverse t-structure}

\textbf{From now on, all schemes are assumed to be excellent and endowed with a dimension function} $\delta$ (see for instance \cite[1.1.1]{bondarko-deglise}).
One of the main goals of this paper is to understand when the perverse t-structure can be restricted to Artin $\ell$-adic complexes and to study Artin perverse sheaves when that is the case.
\subsection{The smooth Artin case}
Using \Cref{t-structure perverse lisse,t-structure ordinaire smooth}, we have:
\begin{proposition}\label{t-structure perverse smooth} Let $S$ be a regular connected scheme and $\ell$ be a prime number which is invertible on $S$. Then, the perverse t-structure of $\mc{D}^b_c(S,\Z_\ell)$ induces a t-structure on $\mc{D}^{smA}(S,\Z_\ell)$. This t-structure coincides with the ordinary t-structure of $\mc{D}^{smA}(S,\Z_\ell)$ (see \Cref{t-structure ordinaire smooth}) shifted by $\delta(S)$. In particular, the heart of this t-structure is $\Sh^{smA}(S,\Z_\ell)[\delta(S)]$ which is equivalent to $\Rep^A(\pi_1^{\mathrm{pro\acute{e}t}}(S,\xi),\Z_\ell)^*$ if $\xi$ is a geometric point of $S$.
\end{proposition}

In particular the case of fields translates to:

\begin{corollary}\label{t-structure perverse smooth fields} Let $k$ be a field and $\ell$ a prime number which is distinct from the characteristic of $k$. Then, the perverse t-structure of $\mc{D}^b_c(k,\Z_\ell)$ induces a t-structure on $\mc{D}^{A}(k,\Z_\ell)$. This t-structure coincides with the ordinary t-structure of \Cref{t-structure ordinaire smooth fields} shifted by $\delta(k)$. In particular, the heart of this t-structure is equivalent to $\Rep^A(G_k,\Z_\ell)^*$.
\end{corollary}

In the case where the base scheme is not connected but is still regular, we have:

\begin{corollary}\label{t-structure perverse smooth non connexe} Let $S$ be a regular scheme and $\ell$ be a prime number which is invertible on $S$. 
Then, the perverse t-structure of $\mc{D}^b_c(S,\Z_\ell)$ induces a t-structure on $\mc{D}^{smA}(S,\Z_\ell)$. If $T$ is a connected component of $S$, denote by $i_T\colon T\rar S$ the clopen immersion. The heart of this t-structure is then the full subcategory made of those objects of the form 
$$\bigoplus \limits_{T \in \pi_0(S)} (i_T)_* L_T[\delta(T)]$$ where $L_T$ is an object of $\Sh^{smA}(T,\Z_\ell)$.
\end{corollary}
\begin{remark} Keeping the notations of \Cref{t-structure perverse smooth non connexe}, the heart of the perverse t-structure on $\mc{D}^{smA}(S,\Z_\ell)$ is equivalent (as an abelian category) to $\Sh^{smA}(S,\Z_\ell)$.
\end{remark}
\begin{definition}\label{smooth Artin perverse sheaves} Let $S$ be a regular scheme. We define the category of \emph{smooth Artin perverse sheaves} denoted by $\mathrm{Perv}^{smA}(S,\Z_\ell)$ to be the heart of the perverse t-structure on $\mc{D}^{smA}(S,\Z_\ell)$.
\end{definition}

\subsection{General properties}
We have the following restriction property:
\begin{proposition}\label{gluing 0} Let $S$ be a scheme and $\ell$ be a prime number which is invertible on $S$. Let $T$ be a locally closed subscheme of $S$. Then, if the perverse t-structure on $\mc{D}^b_c(S,\Z_\ell)$ induces a t-structure on $\mc{D}^A(S,\Z_\ell)$, then, the perverse t-structure on $\mc{D}^b_c(T,\Z_\ell)$ induces a t-structure on $\mc{D}^A(T,\Z_\ell)$ 
%$i\colon F\rar S$ be a closed immersion and $j\colon U\rar S$ be the open complementary immersion. If the perverse t-structure on $\mc{D}^b_c(S,\Z_\ell)$ induces a t-structure on $\mc{D}^A(S,\Z_\ell)$, then the perverse t-structure on $\mc{D}^b_c(U,\Z_\ell)$ induces a t-structure on $\mc{D}^A(U,\Z_\ell)$ and the perverse t-structure on $\mc{D}^b_c(F,\Z_\ell)$ induces a t-structure on $\mc{D}^A(F,\Z_\ell)$. 
%In that case, the t-structure on $\mc{D}^A(S,\Z_\ell)$ is obtained by gluing the t-structures on $\mc{D}^A(U,\Z_\ell)$ and $\mc{D}^A(F,\Z_\ell)$.
\end{proposition}
\begin{proof}
    We can assume that $T$ is either an open subset or a closed subset of $S$. 

    Assume first that $T$ is open and denote by $j\colon T\rar S$ the canonical immersion. Let $M$ be an Artin $\ell$-adic complex over $T$. It suffices to show that ${}^p \tau_{\geqslant 0} M$ is an Artin $\ell$-adic complex. But since the functor $j^*$ is t-exact with respect to the perverse t-structure, we have $${}^p \tau_{\geqslant 0} M=j^*({}^p \tau_{\geqslant 0} j_!M).$$ 
    
    Therefore, using the stability properties of Artin $\ell$-adic complexes of \Cref{6 foncteurs}, the $\ell$-adic complex ${}^p \tau_{\geqslant 0} M$ is Artin.

    Assume now that $T$ is closed and denote by $i\colon T\rar S$ the canonical immersion. Let $M$ be an Artin $\ell$-adic complex over $T$. Since the functor $i_*$ is t-exact with respect to the perverse t-structure, we have $${}^p \tau_{\geqslant 0} M=i^*({}^p \tau_{\geqslant 0} i_*M).$$ 
    
    Therefore, using the stability properties of Artin $\ell$-adic complexes of \Cref{6 foncteurs}, the $\ell$-adic complex ${}^p \tau_{\geqslant 0} M$ is Artin.
\end{proof}

\begin{definition}\label{Artin perverse sheaves} Let $S$ be a scheme and $\ell$ be a prime number which is invertible on $S$. Assume that the perverse t-structure on $\mc{D}^b_c(S,\Z_\ell)$ induces a t-structure on $\mc{D}^A(S,\Z_\ell)$. Then, we define the category $\mathrm{Perv}^A(S,\Z_\ell)$ of \emph{Artin perverse sheaves} to be the heart of this t-structure.
\end{definition}

The following properties are analogous to the classical properties of perverse sheaves:

\begin{proposition}\label{properties of perverse sheaves} Let $\ell$ be a prime number. Assume that all schemes below are defined over $\Z[1/\ell]$ and that the perverse t-structures on their categories of constructible $\ell$-adic sheaves induce t-structures on their categories of Artin $\ell$-adic sheaves.
\begin{enumerate}
\item Let $f\colon T\rar S$ be a quasi-finite morphism of schemes. 
Then, the functor 
$$\Hlp^0f_!\colon \mathrm{Perv}^A(T,\Z_\ell)\rar \mathrm{Perv}^A(S,\Z_\ell)$$
is right exact.

Furthermore, if $f$ is affine, $\Hlp^0f_!=f_!$ and this functor is exact.
\item Let $g\colon T\rar S$ be a morphism of schemes. Assume that $\dim(g)\leqslant d$. 
Then, the functor $$\Hlp^dg^*\colon \mathrm{Perv}^A(S,\Z_\ell)\rar  \mathrm{Perv}^A(T,\Z_\ell)$$
is right exact.

Furthermore, if $f$ is étale, $\Hlp^0f^*=f^*$ and this functor is exact.
\item Consider a cartesian square of schemes
$$\begin{tikzcd}Y\ar[r,"g"] \ar[d,"q"]&  X \ar[d,"p"]\\
T\ar[r,"f"]& S
\end{tikzcd}$$
such that $p$ is quasi-finite and $\dim(f)\leqslant d$. Then we have a canonical equivalence:
$$\Hlp^d f^* (\Hlp^0 p_!) \rar \Hlp^0 q_! (\Hlp^d g^*).$$
\item Let $S$ be a scheme. Let $i\colon F \rar C$ be a closed immersion and $j\colon U\rar C$ be the open complement. Let $M$ be a perverse $\ell$-adic sheaf on $S$. Then, we have an exact sequence of perverse Artin $\ell$-adic sheaves:
$$0\rar i_* \Hlp^{-1} i^* M \rar \Hlp^0j_! j^* M \rar M \rar i_* \Hlp^0 i^* M \rar 0.$$

When $j$ is affine, we have an exact sequence of perverse Artin $\ell$-adic sheaves:
$$0\rar i_* \Hlp^{-1} i^* M \rar j_! j^* M \rar M \rar i_* \Hlp^0 i^* M \rar 0.$$
\end{enumerate}
\end{proposition}
\begin{proof} (1) and (2) follow from the analogous properties of perverse sheaves described in \cite[4.1.2, 4.2.4]{bbd}, from the stability properties of Artin $\ell$-adic complexes of \Cref{6 foncteurs} and from general properties of t-structures described in \cite[1.3.17]{bbd}. Note that \cite[4.1.2]{bbd} still applies in this general setting (and not only when the base scheme is of finite type over a field) since it is a consequence of the absolute purity property of \cite[XIV.3.1]{sga4} which has been generalized by Gabber in \cite[XV.1.1.2]{travauxgabber} in the case of quasi-excellent base schemes. 

Property (3) follows from base change, from properties (1) and (2) and from  \cite[1.3.17]{bbd}. Finally, property (4) follows from \cite[1.4.19]{bbd} which is the analogous property on perverse sheaves and from (1) and (2).
\end{proof}

\subsection{Construction of Artin perverse sheaves over schemes of dimension \texorpdfstring{$2$}{2} or less}

To prove the main theorem of this paper, we first need a few lemmas. 
\begin{lemma}\label{easy} Let $k$ be a field, let $\ell$ be a prime number distinct from the characteristic of $k$ and let $p\colon E\rar \Spec(k)$ be a proper morphism. Let $$E\rar \pi_0(E/k) \overset{q}{\rar} \Spec(k)$$ be the Stein factorization of p.

Then, $R^0p_*\Z_{\ell,X}=q_* \Z_{\ell,\pi_0(E/k)}$.
\end{lemma}
\begin{proof}For any scheme $X$ and any ring of coefficients in $R$, $$\Hl^0_{\et}(X,R)=R^{\pi_0(X)}.$$

The lemma follows from this result applied to $R=\Z/\ell^n\Z$ for all $n$.
\end{proof}
One of our main tools will be the following lemma, which is closely linked to the Rapoport-Zink spectral sequence \cite[2.8]{rz}.

\begin{lemma}\label{spectral sequence} Let $\ell$ be a prime number. Consider a diagram of $\Z[1/\ell]$-schemes:

$$\begin{tikzcd}E\ar[r,"p"]\ar[d,"i"] &F\\
Y &
\end{tikzcd}$$
where $p$ is a proper map and $i$ is a closed immersion of a simple normal crossing divisor $E$ into a regular scheme $Y$. 

Write $E=\bigcup\limits_{i\in I} E_i$ where $I$ is a finite set and for all $J\subseteq I$, the closed subscheme $E_J\colon =\bigcap\limits_{i \in J} E_i$ of $Y$ is a regular subscheme of codimension $|J|$.

For all $J\subseteq I$, let $p_J\colon E_J\rar F$ be the natural map. Then, there is a spectral sequence such that:

$$\mathrm{E}^{p,q}_1=\begin{cases}\bigoplus\limits_{\substack{J \subseteq I \\ |J|=p+1}}\Hlp^{q-2}\left((p_J)_*\Z_{\ell,E_J}\right)(-p-1) & \text{ if }p \geqslant 0 \\
0 &\text{ if }p<0\end{cases}\Longrightarrow \Hlp^{p+q}(p_*i^!\Z_{\ell,Y}).$$
\end{lemma}
\begin{proof}Denote $i_J\colon E_J\rar Y$ the inclusion. \Cref{Mayer-Vietoris} asserts that $i_*i^!\Z_{\ell,Y}$ is the colimit:

$$\colim\limits_{p \in \Delta^{inj}} \left(\bigoplus\limits_{\substack{J \subseteq I \\ |J|=p+1}} (i_J)_* i_J^! \Z_{\ell,Y}\right).$$ where $\Delta^{inj}$ is the category of finite ordered sets with morphisms the injective maps.

We can truncate (naively) this diagram in each degree to write $i_*i^!\Z_{\ell,Y}$ as the colimit of a sequential diagram $(M_p)_{p\geqslant 0}$ such that the cofiber of the map $M_{p-1}\rar M_{p}$ is $$M(p-1,p):=\bigoplus\limits_{\substack{J \subseteq I \\ |J|=p+1}} (i_J)_* i_J^! \Z_{\ell,Y}[p].$$

Pulling back to $E$, pushing forward to $F$ and using the absolute purity \cite[XV.1.1.2]{travauxgabber}, we can write $p_*i^!\Z_{\ell,Y}$ as the colimit of a sequential diagram $(N_p)_{p\geqslant 0}$ such that the cofiber of the map $N_{p-1}\rar N_{p}$ is $$N(p-1,p):=\bigoplus\limits_{\substack{J \subseteq I \\ |J|=p+1}} (p_J)_* \Z_{\ell,E_J}(-p-1)[-p-2].$$

But by \cite[1.2.2.14]{ha} such a sequential diagram gives rise to a spectral sequence such that $$\mathrm{E}^{p,q}_1=\bigoplus\limits_{\substack{J \subseteq I \\ |J|=p+1}}\Hlp^{p+q}(N(p-1,p))\Longrightarrow \Hlp^{p+q}(p_*i^!\Z_{\ell,Y}).$$

Finally, $$\Hlp^{p+q}(N(p-1,p))=\Hlp^{q-2}\left((p_J)_*\Z_{\ell,E_J}\right)(-p-1)$$ which finishes the proof.

\end{proof}

The following consequence of this lemma will be useful:

\begin{corollary}\label{calcul penible} Keep the notations of the lemma. Take the convention that $\delta(F)=\dim(F)$. Then, 

\begin{enumerate}\item $\Hlp^k(p_*i^!\Z_{\ell,Y})$ vanishes for $k=0,1$.
\item Write $I_0=\{i \in I\mid \delta(p_i(E_i))=0\}$. If $i\in I_0$, let $$E_i\rar Z_i\overset{q_i}{\rar} F.$$ be the Stein factorization of $p_i$, so that the map $q_i$ is finite. Then, 
\begin{enumerate}
    \item The image of $q_i$ is a finite subset of closed points of codimension $\delta(F)$ of $F$ if $i\in I_0$.
    \item We have $\Hlp^2(p_*i^!\Z_{\ell,Y})=\bigoplus\limits_{i \in I_0} (q_i)_* \Z_{\ell,Z_i}(-1).$
\end{enumerate}
\end{enumerate}
\end{corollary}
\begin{proof} If $i$ is an element of $I$, we denote by $W_i$ the image of $E_i$ through the map $p$ and by $x_i\in F$ the generic point of $W_i$.

Notice first that if $x\in F$, we have $$\delta(x)\geqslant \dim(F)-\codim_F(x)\geqslant \dim(\overline{\{ x\}}).$$

Applying this to the point of $x_i$ for $i\in I_0$ proves (2)(a). We now prove the remaining statements. 

Letting $y$ be a closed point of $E_J$ of maximal codimension, We have $$\delta(E_J)-\dim(E_J)=\delta(E_J)-\codim_{E_J}(y)=\delta(y)\geqslant \delta(p(y))\geqslant 0.$$

Since the relative dimension of $p_J$ is at most $\dim(E_J)$, using \cite[4.2.4]{bbd} and the fact that $\Z_{\ell,E_J}$ lies in degree $\delta(E_J)$ with respect to the perverse t-structure, we get that $\Hlp^{q}\left((p_J)_*\Z_{\ell,E_J}\right)$ vanishes if $q$ is a negative integer. 

Therefore, the only non-vanishing term of the spectral sequence of \Cref{spectral sequence} such that $p+q\leqslant 2$ is the term $$\mathrm{E}^{0,2}_1=\bigoplus_{i \in I}\Hlp^0((p_i)_*\Z_{\ell,E_i})(-1).$$

Now, if $i \notin I_0$, we claim that $\Hlp^0((p_i)_*\Z_{\ell,E_i})$ vanishes. 

Assume first that $W_i$ is of positive dimension. Then, no fibers of $p_i$ can be of dimension $\dim(E_i)$ because the map $E_i\rar W_i$ is surjective and because $E_i$ is irreducible.  Therefore, the relative dimension of $p_i$ is at most $\dim(E_i)-1$. Since the sheaf $\Z_{\ell,E_i}$ lies in degree $\delta(E_i)\geqslant \dim(E_i)$, \cite[4.2.4]{bbd} ensures that the sheaf $\Hlp^0((p_i)_*\Z_{\ell,E_i})$ vanishes.

In the case where $W_i=\{ x_i\}$, we get as before that $$\delta(E_i)-\dim(E_i)\geqslant \delta(x_i)=\delta(W_i)> 0$$ and therefore, the sheaf $\Hlp^0((p_i)_*\Z_{\ell,E_i})$ vanishes by \cite[4.2.4]{bbd} as before.

The corollary then follows from \Cref{easy}.
\end{proof}

The following statement is the main result of this paper.
\begin{theorem}\label{main theorem} Assume that $S$ is an excellent scheme of dimension $2$ or less and $\ell$ be a prime number which is invertible on $S$, then, the perverse t-structure induces a t-structure on $\mc{D}^A(S,\Z_\ell)$.
\end{theorem}
\begin{proof} We can assume that $S$ is reduced and connected. We proceed by noetherian induction on $S$. If $\dim(S)=0$, the result follows from \Cref{t-structure perverse smooth fields}.

Assume now that $\dim(S)>0$. If $j\colon U\rar S$ is an open immersion, denote by $\mc{D}^A_U(S,\Z_\ell)$ the subcategory of $\mc{D}^A(S,\Z_\ell)$ made of those complexes $M$ such that the $\ell$-adic complex $j^*M$ is smooth Artin. Using \Cref{ouvert dense lisse}, every object of $\mc{D}^A(S,\Z_\ell)$ lies in some $\mc{D}^A_U(S,\Z_\ell)$  for some  dense open immersion $j\colon U\rar S$. Shrinking $U$ if needed, we can assume that $U$ is regular since $S$ is excellent. Using \Cref{reduction affine} below, we may also assume that the morphism $j$ is affine. 

Therefore, it suffices to show that for any dense affine open immersion $j\colon U\rar S$ with $U$ regular, the perverse t-structure induces a t-structure on the subcategory $\mc{D}^A_U(S,\Z_\ell)$. Let now $j\colon U\rar S$ be such an immersion, let $i\colon F\rar S$ be the reduced closed complementary immersion. We let \begin{itemize}
\item $\mc{S}=\mc{D}^b_c(S,\Z_\ell)$,
\item $\mc{U}=\mc{D}^b_c(U,\Z_\ell)$ and $\mc{U}_0=\mc{D}^{smA}(U,\Z_\ell)$,
\item $\mc{F}=\mc{D}^b_c(F,\Z_\ell)$ and $\mc{F}_0=\mc{D}^A(F,\Z_\ell)$.
\end{itemize}

The category $\mc{S}$ is a gluing of the pair $(\mc{U},\mc{F})$ along the fully faithful functors $j_*$ and $i_*$ in the sense of \cite[A.8.1]{ha} and the perverse t-structure on $\mc{S}$ is obtained by gluing the perverse t-structures of $\mc{U}$ and $\mc{F}$ in the sense of \cite[1.4.10]{bbd}. In addition, the perverse t-structure induces a t-structure on $\mc{U}_0$ by \Cref{t-structure perverse smooth non connexe} and on $\mc{F}_0$ by induction. Furthermore, since the morphism $j$ is affine, the functor $j_!\colon \mc{U}\rar \mc{S}$ is t-exact by the Affine Lefschetz Theorem. 

Therefore, using \Cref{outil 1} below, it suffices to show that for any object $M$ of $\mathrm{Perv}^A(F,\Z_\ell)$, any object $N$ of $\mathrm{Perv}^{smA}(U,\Z_\ell)$ and any map $$f\colon M\rar \Hlp^{-1}(i^*j_*N),$$ 
the kernel $K$ of $f$ is an Artin $\ell$-adic complex on $F$.

If $\dim(S)=1$, the scheme $F$ is $0$-dimensional. 
Thus, \Cref{t-structure perverse smooth fields} implies that the subcategory $\mathrm{Perv}^A(F,\Z_\ell)$ of $\mathrm{Perv}(F,\Z_\ell)$ is Serre. 
Since $K$ is a subobject of $M$, it is an object of $\mathrm{Perv}^A(F,\Z_\ell)$ and thus it is an Artin $\ell$-adic complex.

If $\dim(S)=2$, the scheme $F$ is of dimension at most $1$. 

Let $Z\rar F$ be a closed immersion such that the open complementary immersion is dense and affine. 
Then, the scheme $Z$ is of dimension at most $0$ and therefore, by \Cref{t-structure perverse smooth fields}, the subcategory $\mathrm{Perv}^A(Z,\Z_\ell)$ of $\mathrm{Perv}(Z,\Z_\ell)$ is Serre. Furthermore, \Cref{outil 2} below implies that the perverse sheaf $\Hlp^{-1}(\iota^*\Hlp^{-1}(i^*j_*N))$ is Artin.

Finally, using \Cref{outil 3} below, those two facts imply together that $K$ is an Artin perverse sheaf over $F$ which finishes the proof.
\end{proof}
%\begin{lemma}\label{outil 1} Let $S$ be a reduced excellent scheme and $\ell$ be a prime number which is invertible on $S$. Assume that:
%\begin{enumerate}[label=(\roman*)]\item The perverse t-structure induces a t-structure on $\mc{D}^A(F,\Z_\ell)$ where $F$ is any closed subscheme of $S$ which is of positive codimension.
%\item If $j\colon U\rar S$ is a dense affine open immersion with $U$ regular, letting $i\colon F\rar S$ be the reduced closed complementary immersion, for any object $M$ of $\mathrm{Perv}^A(F,\Z_\ell)$, any object $N$ of $\mathrm{Perv}^{smA}(U,\Z_\ell)$ and any map $$f\colon M\rar \Hlp^{-1}(i^*j_*N),$$ the kernel of $f$ is an Artin $\ell$-adic complex over $F$.
%\end{enumerate}
%Then, the perverse t-structure induces a t-structure on $\mc{D}^A(S,\Z_\ell)$.
%\begin{remark} The notation $\mathrm{Perv}^A(F,\Z_\ell)$ in (ii) makes sense because of (i) and the notation $\mathrm{Perv}^{smA}(U,\Z_\ell)$ makes sense because of \Cref{t-structure perverse smooth non connexe}.
%\end{remark}
%\end{lemma}

\begin{lemma}\label{reduction affine} Let $S$ be a (noetherian) scheme. Let $U\subseteq S$ be an open subscheme. Then, there is a dense open subscheme $V$ of $U$ such that the map $V\rar S$ is affine.
\end{lemma}
\begin{proof}
The proof of the existence of affine stratifications on noetherian schemes of \cite[0F2X]{stacks} provides a dense open affine subscheme $W$ of $S$ such that the immersion $W\rar S$ is affine. 
Since a composition of affine morphisms is affine, we may replace $S$ with $W$. 
Therefore we can assume that $S$ is an affine scheme. By the prime avoidance lemma \cite[00DS]{stacks}, there is an affine open $V$ contained in $U$ which contains the generic points of $S$. 
The map $V\rar S$ is affine and thus, the lemma is proved.
\end{proof}

\begin{lemma}\label{outil 1} Let $\mc{U}$ and $\mc{F}$ be t-categories, let $j_*\colon \mc{U}\rar \mc{S}$ and $i_*\colon \mc{F}\rar \mc{S}$ be fully faithful functors such that $\mc{S}$ is a gluing of the pair $(\mc{U},\mc{F})$ in the sense of \cite[A.8.1]{ha} (and therefore the axioms of \cite[1.4.3]{bbd} are satisfied) and let $\mc{U}_0$ (resp. $\mc{F}_0$) be a sub-t-category of $\mc{U}$ (resp. $\mc{F}$). Endow $\mc{S}$ with the glued t-structure of \cite[1.4.10]{bbd} and denote by $i^*$ (resp. $j^*$) the left adjoint functor to $i_*$ (resp. $j_*$).  Assume that the following properties hold. 
\begin{enumerate}[label=(\roman*)]\item The stable subcategory $\mc{S}_0$ of $\mc{S}$ made of those objects $M$ such that $i^*M$ belongs to $\mc{F}_0$ and $j^*M$ belongs to $\mc{U}_0$ is bounded. 
\item The left adjoint functor $j_!$ to $j^*$ is t-exact. 
\item For any object $M$ of $\mc{F}_0^\heart$, any object $N$ of $\mc{U}_0^\heart$ and any map $$f\colon M\rar \Hl^{-1}(i^*j_*N),$$ the kernel of $f$ lies in $\mc{F}_0$.
\end{enumerate}
Then, the t-structure of $\mc{S}$ induces a t-structure on the stable subcategory $\mc{S}_0$. 
\end{lemma}
\begin{proof} \textbf{Step 1:} 
Let $M$ be an object of $\mc{F}_0^\heart$, let $N$ be an object of $\mc{U}_0^\heart$. By \cite[1.4.16]{bbd}, the object $i_*M$ belongs to $\mc{S}^\heart$. Furthermore, using our second hypothesis, the object $j_!N$ also belongs to $\mc{S}^\heart$. Let $f\colon i_*M\rar j_!N$ be an arbitrary map. In this step we prove that the objects $\ker(f)$ and $\coker(f)$ belong to $\mc{S}_0$. 

By \cite[A.8.13]{ha}, we have an exact triangle:
$$j_!N\rar j_*N\rar i_* i^* j_* N$$
which by \cite[1.4.16]{bbd} yields an exact sequence:
$$0\rar i_* \Hlp^{-1}(i^* j_* N) \rar j_! N \rar \Hl^0(j_*N).$$

Notice that by \cite[1.4.16]{bbd}, we have $$\Hl^0(j_*M)=\tau_{\leqslant 0}(j_*M).$$

Since $j^*i_*=0$, we get $$\Hom(i_*M,\tau_{\leqslant 0}(j_*M))=\Hom(i_*M,j_*N)=0.$$ Thus, the map $f$ factors through the object $i_* \Hl^{-1}(i^* j_* N)$. Since the functor $i_*$ is t-exact, we get $$\ker(f)=i_* \ker\left(M\rar  \Hl^{-1}(i^* j_* N)\right).$$

Therefore, our last assumption implies that object $\ker(f)$ belongs to $\mc{S}_0$. 

To finish the first step, let $C$ be the kernel of the map $j_!N\rar \coker(f)$. We have exact triangles:
$$\ker(f)\rar i_* M\rar C$$ 
$$C\rar j_! N\rar \coker(f).$$

The first exact triangle shows that $C$ is belongs to $\mc{S}_0$. Thus, using the second exact triangle, the object $\coker(f)$ also lies in $\mc{S}_0$. 

\textbf{Step 2:} We now prove the lemma. First notice that by dévissage, it suffices to show that if $P$ belongs to $\mc{S}_0$ and $n$ is an integer, then, the object $\Hl^n(P)$ of $\mc{S}$ lies in $\mc{S}_0$.

The localization triangle \cite[1.4.3.4]{bbd} yields an exact sequence:
$$i_* \Hl^{n-1}(i^*P)\overset{\alpha}{\rar} j_! \Hl^n(j^*P)\rar \Hl^n(P)\rar i_* \Hl^n(i^*P)\overset{\beta}{\rar} j_! \Hl^{n+1}(j^*P)$$
which yields an exact triangle:
$$\coker(\alpha)\rar \Hl^n(P)\rar \ker(\beta).$$

For any integer $m$, the object $\Hl^{m}(i^*P)$ belongs to $\mc{F}_0^\heart$ and the object $\Hl^m(j^*P)$ lies in $\mc{U}_0^\heart$.
Thus, using our first step, the objects $\coker(\alpha)$ and $\ker(\beta)$ belong to $\mc{S}_0$ and so does the object $\Hl^n(P)$ which finishes the proof.
\end{proof}

\begin{lemma}\label{outil 3} Let $F$ be a reduced excellent scheme and $\ell$ be a prime number which is invertible on $F$. Assume that the perverse t-structure on $\mc{D}^b_c(F,\Z_\ell)$ induces a t-structure on $\mc{D}^A(F,\Z_\ell)$. 
Let $M$ be an Artin perverse sheaf over $F$, let $P$ be a perverse sheaf over $F$ and let $f\colon M\rar P$ be an arbitrary map of perverse sheaves. 

Assume that for any reduced closed immersion $\iota\colon Z\rar F$ such that the open complementary immersion is dense and affine,
\begin{enumerate}[label=(\roman*)]
    \item The subcategory $\mathrm{Perv}^A(Z,\Z_\ell)$ of $\mathrm{Perv}(Z,\Z_\ell)$ is Serre.
    \item The perverse sheaf $\Hlp^{-1}(\iota^*P)$ over $Z$ is an Artin perverse sheaf.
\end{enumerate}
Then, the kernel of $f$ is an Artin perverse sheaf over $F$.
\end{lemma}
\begin{remark} The notation $\mathrm{Perv}^A(Z,\Z_\ell)$ in (i) makes sense by \Cref{gluing 0}.
\end{remark}
\begin{proof} Using \Cref{ouvert dense lisse} we have $\gamma\colon V\rar F$ a dense open immersion such that $\gamma^*M$ is a smooth Artin $\ell$-adic complex. 
Shrinking $V$ if needed, we can assume that $V$ is regular and that $\gamma^*P$ is a lisse perverse sheaf. 
Using \Cref{reduction affine}, we may assume that $\gamma$ is affine so that the functor $\gamma_!$ is t-exact by the Affine Lefschetz Theorem. Let $\iota\colon Z\rar F$ be the complementary closed immersion of $\gamma$. We let 
\begin{itemize}
\item $\mc{F}=\mc{D}^b_c(F,\Z_\ell)$,
\item $\mc{V}=\mc{D}_{\mathrm{lisse}}(V,\Z_\ell)$ and $\mc{V}_0=\mc{D}^{smA}(U,\Z_\ell)$,
\item $\mc{Z}=\mc{D}^b_c(Z,\Z_\ell)$ and $\mc{Z}_0=\mc{D}^A(Z,\Z_\ell)$.
\end{itemize}
By \Cref{t-structure perverse smooth non connexe}, the subcategory $\mc{V}_0^\heart=\mathrm{Perv}^{smA}(V,\Z_\ell)$ of $\mc{V}^\heart$ is Serre while the subcategory $\mc{Z}_0^\heart$ of $\mc{Z}^\heart$ is Serre by assumption.

Furthermore, using \cite[1.4.17.1]{bbd}, the exact and fully faithful functor $$\iota_*\colon \mathrm{Perv}(Z,\Z_\ell)\rar \mathrm{Perv}(F,\Z_\ell)$$ maps $\mathrm{Perv}(Z,\Z_\ell)$ into a Serre subcategory. 

Finally, by \Cref{t-structure perverse smooth non connexe}, the perverse t-structure on $\mc{D}^A(F,\Z_\ell)$ induces a t-structure on the subcategory $\mc{D}^A_V(F,\Z_\ell)$ made of those complexes $M$ such that the $\ell$-adic complex $\gamma^*M$ is smooth Artin. Therefore, the result follows from \Cref{outil 3bis} below.
\end{proof}

\begin{lemma}\label{outil 3bis} Let $\mc{V}$ and $\mc{Z}$ be t-categories, let $\gamma_*\colon \mc{V}\rar \mc{F}$ and $f\iota_*\colon \mc{Z}\rar \mc{F}$ be fully faithful functors such that $\mc{F}$ is a gluing of the pair $(\mc{U},\mc{F})$ in the sense of \cite[A.8.1]{ha} (and therefore the axioms of \cite[1.4.3]{bbd} are satisfied) and let $\mc{V}_0$ (resp. $\mc{Z}_0$) be a sub-t-category of $\mc{V}$ (resp. $\mc{Z}$). Endow $\mc{F}$ with the glued t-structure of \cite[1.4.10]{bbd} and denote by $\iota^*$ (resp. $\gamma^*$) the left adjoint functor to $\iota_*$ (resp. $\gamma_*$). Assume that the following properties hold. 
\begin{enumerate}[label=(\roman*)]\item The t-structure on $\mc{F}$ induces a bounded t-structure on the stable subcategory $\mc{F}_0$ of $\mc{F}$ made of those objects $M$ such that $\iota^*M$ belongs to $\mc{F}_0$ and $\gamma^*M$ belongs to $\mc{U}_0$.
\item The left adjoint functor $\gamma_!$ to $\gamma^*$ is t-exact. 
\item The exact and fully faithful functor $\iota_*\colon \mc{Z}^\heart\rar \mc{F}^\heart$ maps $\mc{Z}^\heart$ into a Serre subcategory.
\item The subcategory $\mc{V}_0^\heart$ (resp. $\mc{Z}_0^\heart$) of $\mc{V}^\heart$ (resp. $\mc{Z}^\heart$) is Serre.
\end{enumerate}
Then, if $M$ is an object of $\mc{F}_0^\heart$, if $P$ is an object of $\mc{F}^\heart$ such that $\Hl^{-1}(\iota^*P)$ belongs to $\mc{Z}_0^\heart$ and if $f\colon M\rar N$ is a map in $\mc{F}^\heart$, the kernel of $f$ belongs to $\mc{F}_0$.
\end{lemma}
\begin{proof} Let $K$ be the kernel of $f$. Using the t-exactness assumption on $\gamma_!$ and \cite[1.4.16]{bbd}, we have commutative diagrams with exact rows in $\mc{F}^\heart$.

$$\begin{tikzcd}
    0 \ar[r] &
    \iota_*\Hl^{-1}(\iota^*M)\ar[r]\ar[d,"\iota_*u"]&
    \gamma_!\gamma^*M \ar[d,"\gamma_!v"]\ar[r]&
    Q\ar[r]\ar[d,"g"]&
    0 \\
    0 \ar[r] &
    \iota_*\Hl^{-1}(\iota^*P)\ar[r]&
    \gamma_!\gamma^*P \ar[r]&
    R\ar[r]&
    0
\end{tikzcd}$$
$$\begin{tikzcd}
    0 \ar[r] &
    Q\ar[r]\ar[d,"g"]&
    M \ar[d,"f"]\ar[r]&
    \iota_*\Hl^{0}(\iota^*M)\ar[r]\ar[d,"\iota_*w"]&
    0 \\
    0 \ar[r] &
    R\ar[r]&
    P \ar[r]&
    \iota_*\Hl^{0}(\iota^*P)\ar[r]&
    0
\end{tikzcd}$$
which, by the snake lemma, yield exact sequences: 
$$0\rar \iota_*\ker(u)\overset{a}{\rar} \gamma_!\ker(v)\rar \ker(g)\overset{b}{\rar} \iota_*\coker(u)$$
$$0\rar \ker(g) \rar K\overset{c}{\rar} \iota_*\ker(w).$$

Since by assumption, the subcategory $\mc{V}_0^\heart$ of $\mc{V}^\heart$ is Serre, the object $\ker(v)$ belongs to $\mc{V}_0^\heart$. Moreover, the object $\Hl^{0}(\iota^*M)$, $\Hl^{-1}(\iota^*M)$ and $\Hl^{-1}(\iota^*P)$ belong to $\mc{Z}_0^\heart$. Since the subcategory $\mc{Z}_0^\heart$ of $\mc{Z}^\heart$ is Serre, the objects $\ker(w)$, $\ker(u)$ and $\coker(u)$ then belong to $\mc{Z}_0^\heart$. 

Since we have an exact triangle $$\iota_*\ker(u)\overset{a}{\rar} \gamma_!\ker(v)\rar \coker(a),$$
the object $\coker(a)$ therefore belongs to $\mc{F}_0^\heart$.

As the functor $\iota_*\colon \mc{Z}^\heart\rar \mc{F}^\heart$ maps $\mc{Z}^\heart$ into a Serre subcategory, the image of $\ker(g)$ through $b$ is of the form $\iota_*C$, where $C$  is a subobject of $\coker(u)$ and therefore belongs to $\mc{Z}_0^\heart$.

Likewise, the image of $K$ through $c$ is of the form $\iota_*K'$ where $K'$ is a subobject of $\ker(w)$ and is therefore belongs to $\mc{Z}_0^\heart$.

Thus, we get exact triangles:
$$\coker(a)\rar \ker(g)\rar \iota_*C $$
$$\ker(g)\rar K\rar \iota_*K'$$

Hence, the object $\ker(g)$ belongs to $\mc{F}_0^\heart$ and so does $K$ which finishes the proof.
\end{proof}

\begin{lemma}\label{outil 2} Let $S$ be a reduced excellent connected scheme of dimension $2$ and $\ell$ be a prime number which is invertible on $S$.
Let $j\colon U\rar S$ be a dense affine open immersion with $U$ regular. 
Let $i\colon F\rar S$ be the reduced closed complementary immersion. 
%Assume that the irreducible components of $F$ are of dimension $1$ (and therefore that $F$ is purely of codimension $1$).
Let $N$ be an object of $\mathrm{Perv}^{smA}(U,\Z_\ell)$. 
Let $\iota\colon Z\rar F$ be a reduced closed immersion of positive codimension.
%and let $i_x\colon \{ x\}\rar F$ be the canonical reduced closed immersion. 
Then, the perverse sheaf $\Hlp^{-1}(\iota^*\Hlp^{-1}(i^*j_* N))$ is Artin.
\end{lemma}
\begin{proof} We may assume that $\delta(S)=2$.
Using \Cref{passage coeffs Ql}, it suffices to show the same result with coefficients in $\Q_\ell$. 
Furthermore, replacing $S$ with the closure of a connected component of $U$, we may assume that $U$ is connected. 
Thus, since $U$ is regular, \Cref{t-structure perverse smooth} ensures that $$\mathrm{Perv}^{smA}(U,\Z_\ell)=\Sh^{smA}(U,\Z_\ell)[2].$$

\textbf{Step 1:} We show that we can assume that the image of $N$ through the equivalence of \Cref{description des faisceaux l-adiques lisses} is an Artin representation of $\pi_1^{\mathrm{pro\acute{e}t}}(U)$ (throughout this proof, we omit the choice of a base point to lighten the notations). 

Assume the result for the objects whose image through the equivalence of \Cref{description des faisceaux l-adiques lisses} is an Artin representation. Let $\mathrm{A}_0=\Rep^A(\pi_1^{\mathrm{pro\acute{e}t}}(U),\Q_\ell)$ and if $n\geqslant 0$ is an integer, let $\mathrm{A}_{n+1}$ be the full subcategory of $\Rep(\pi_1^{\mathrm{pro\acute{e}t}}(U),\Q_\ell)$ whose objects are the extensions of objects of $\mathrm{A}_n$ in $\Rep(\pi_1^{\mathrm{pro\acute{e}t}}(U),\Q_\ell)$. 
There is a non-negative integer $n$ such that the image of $N$ through the equivalence of \Cref{description des faisceaux l-adiques lisses} is an object of $\mathrm{A}_n$. 

We now prove the result by induction on $n$. If $n=0$, this is our assumption. 
Suppose that the result holds for an integer $n$. 
Let $N$ be an object of $\mathrm{A}_{n+1}$.
We have an exact sequence $$0\rar N'\rar N \rar N''\rar 0$$
where $N'$ and $N''$ are objects of $\mathrm{A}_n$.

This exact sequence yields an exact triangle:
$$i^*j_ *N'\rar i^*j_*N \rar i^*j_*N''.$$
Since $i^*j_*$ is of cohomological amplitude $[-1,0]$, this yields an exact sequence:
$$0\rar \Hlp^{-1}(i^*j_*N')\rar \Hlp^{-1}(i^*j_*N)\rar \Hlp^{-1}(i^*j_*N'').$$

Letting $C$ be the image of the map $\Hlp^{-1}(i^*j_*N)\rar \Hlp^{-1}(i^*j_*N'')$ and $C'$ be the cokernel of the map $C\rar \Hlp^{-1}(i^*j_*N'')$, we get two exact triangles:
$$\Hlp^{-1}(i^*j_*N')\rar \Hlp^{-1}(i^*j_*N)\rar C$$
$$C\rar \Hlp^{-1}(i^*j_*N'')\rar C'.$$

Applying $\iota^*$ to the first exact triangle yields an exact sequence:
$$\Hlp^{-1}(\iota^*\Hlp^{-1}(i^*j_*N'))\rar \Hlp^{-1}(\iota^*\Hlp^{-1}(i^*j_*N))\rar \Hlp^{-1}(\iota^*C)$$

Furthermore, since $\iota^*$ has cohomological amplitude $[-1,0]$, the second exact triangle yields an exact sequence:
$$0\rar \Hlp^{-1}(\iota^*C)\rar \Hlp^{-1}(\iota^*\Hlp^{-1}(i^*j_*N''))$$

Since $Z$ is at most $0$-dimensional, \Cref{t-structure perverse smooth fields} implies that $\mathrm{Perv}^{A}(Z,\Q_\ell)$ is a Serre subcategory of $\mathrm{Perv}(Z,\Q_\ell).$ Furthermore, by our induction hypothesis, the perverse sheaves $\Hlp^{-1}(\iota^*\Hlp^{-1}(i^*j_*N''))$ and $\Hlp^{-1}(\iota^*\Hlp^{-1}(i^*j_*N'))$ are objects of $\mathrm{Perv}^{A}(Z,\Q_\ell)$. Thus, so is $\Hlp^{-1}(\iota^*C)$ and therefore, so is $\Hlp^{-1}(\iota^*\Hlp^{-1}(i^*j_*N))$.

\textbf{Step 2:} In this step we prove that we can assume $S$ to be normal and that we can assume that $N=\Q_{\ell,U}[2]$. 

Since the category $\Rep^A(\pi_1^{\mathrm{pro\acute{e}t}}(U),\Q_\ell)$ is semisimple by Maschke's lemma, the the image of $N$ through the equivalence of \Cref{description des faisceaux l-adiques lisses} is a direct factor of an object whose image is of the form $$\Q_\ell[G_1]\bigoplus \cdots \bigoplus \Q_\ell[G_n]$$ where $G_1,\ldots,G_n$ are finite quotients of $\pi_1^{\mathrm{pro\acute{e}t}}(U)$. 

Thus, $N$ is a direct factor of an object of the form $f_*\Q_{\ell,V}$ where $f\colon V\rar U$ is a finite étale map. Therefore, we can assume that $N=f_*\Q_{\ell,V}$.

Let $\overline{V}$ be the relative normalization of $S$ in $V$. Since $S$ is excellent and noetherian, it is Nagata by \cite[033Z]{stacks}. Thus, \cite[0AVK]{stacks} ensures that the structural map $\nu\colon \overline{V}\rar S$ is finite. Furthermore, the scheme $\overline{V}$ is normal by \cite[035L]{stacks}.

%Letting $\nu_F\colon F'\rar F$ be the pullback of $\nu$ along the map $i$, letting $\psi\colon F'\rar F$ be the pullback of $i$ along $\nu$ and letting $\gamma\colon V\rar \overline{V}$ be the structural map, we get a commutative diagram
Consider the following commutative diagram
$$\begin{tikzcd} V \ar[d,"f"]\ar[r,"\gamma"] & \overline{V} \ar[d,"\nu"]& F' \ar[l,sloped,"\psi"] \ar[d,"\nu_F"] \\
U \ar[r,"j"]& S & F \ar[l,sloped,"i"]
\end{tikzcd}$$
made of cartesian squares (the left-hand square is cartesian by construction of the relative normalization). 
In particular, the map $\psi$ is a closed immersion and the map $\gamma$ is the open complementary immersion.

Now, we have $$i^*j_*N=i^*j_*f_*\Q_{\ell,V}[2]=i^*\nu_* \gamma_*\Q_{\ell,V}[2]=(\nu_F)_* \psi^*\gamma_*\Q_{\ell,V}[2].$$

Since $\nu_F$ is finite, it is perverse t-exact. Thus,  $$\Hlp^{-1}(\iota^*\Hlp^{-1}(i^*j_* N))= \Hlp^{-1}(\iota^*(\nu_F)_*\Hlp^{-1}(\psi^*\gamma_* \Q_{\ell,V}[2])).$$

Let $\nu_Z\colon Z'\rar Z$ be the pullback of $\nu_F$ along the map $\iota$ and $\varepsilon\colon Z'\rar F'$ be the obvious map. 
Since $\nu_Z$ is finite, we get
$$\Hlp^{-1}(\iota^*(\nu_F)_*\Hlp^{-1}(\psi^*\gamma_* \Q_{\ell,V}[2]))=(\nu_Z)_*\Hlp^{-1}(\varepsilon^*\Hlp^{-1}(\psi^*\gamma_* \Q_{\ell,V}[2])).$$

%Since $Z$ is $0$-dimensional, letting $\iota_y\colon\{ y\} \rar Z$ be the closed immersion when $\{ y\}$ is endowed with its reduced scheme structure, we have that for any object $M$ of $\mc{D}^b_c(Z,\Z_\ell)$, $$M=\bigoplus_{y \in Z} (\iota_y)_*\iota_y^*M.$$

%Thus, letting $i_y=\varepsilon \circ \iota_y:\{ y\} \rar F'$, we get:
%$$\Hlp^{-1}(i_x^*\Hlp^{-1}(i^*j_* N))=(\nu_x)_*\bigoplus_{y \in Z} (\iota_y)_*\Hlp^{-1}(i_y^*\Hlp^{-1}(\iota^*\gamma_* \Q_{\ell,V}[2])).$$

Therefore, we can replace $S$ with $\overline{V}$, $U$ with $V$, $Z$ with $Z'$ in order to assume that $S$ is normal and that $N=\Q_{\ell,U}[2]$.

\textbf{Step 3:} In this step, we study the singularities of $S$. We want to prove that the perverse sheaf $\Hlp^{-1}(\iota^*\Hlp^{1}(i^*j_* \Q_{\ell,U}))$ is Artin.

First, notice that if $Y\rar S$ is a closed immersion and $Z \cap Y=\varnothing$, letting $\xi\colon F\setminus (Y\cap F)\rar F$ be the canonical open immersion, and letting $u\colon Z\rar F\setminus (Y\cap F)$ be the canonical closed immersion, we have $\iota=\xi \circ u$. Since $\xi$ is an open immersion, we get $$\Hlp^{-1}(\iota^*\Hlp^{1}(i^*j_* \Q_{\ell,U}))=\Hlp^{-1}(u^*\Hlp^{1}(\xi^*i^*j_* \Q_{\ell,U})).$$

Notice that the immersion $i\circ \xi$ induces a closed immersion $F \setminus (Y\cap F) \rar S\setminus Y$. We can therefore remove any closed subset which does not intersect $Z$ from $S$ without changing $\Hlp^{-1}(\iota^*\Hlp^{1}(i^*j_* \Q_{\ell,U})).$

Since $S$ is normal, we can therefore assume that its singular locus is contained in $Z$. 
%Furthermore, we can assume $Z$ to be connected and therefore that $Z=\{ x\}$.
Moreover, we can assume that the irreducible components of $F$ only cross at points of $Z$. 

Since $S$ is excellent, Lipman's Theorem on embedded resolution of singularities applies (see \cite[0BGP,0BIC,0ADX]{stacks} for precise statements). We get a cdh-distinguished square:

$$\begin{tikzcd}
    E\ar[r,"i_E"]\ar[d,"p"] &\widetilde{S}\ar[d,"f"]\\
    F \ar[r,"i"] & S
\end{tikzcd}$$
such that the map $f$ induces an isomorphism over $S\setminus Z$, the scheme $\widetilde{S}$ is regular and the subscheme $E$ of $\widetilde{S}$ is a simple normal crossing divisor. By removing a closed subset from $S$ which does not intersect $Z$, we can assume that crossings only occur over $Z$. 

Let $\gamma\colon U\rar \widetilde{S}$ be the complementary open immersion of $i_E$. We have $j=f\circ \gamma$. 
Therefore, we get $$i^*j_* \Q_{\ell,U}=i^*f_*\gamma_*\Q_{\ell,U}=p_* i_E^*\gamma_*\Q_{\ell,U}.$$

We have a localization exact triangle:
$$(i_E)_*i_E^!\Q_{\ell,\widetilde{S}}\rar \Q_{\ell,\widetilde{S}}\rar \gamma_* \Q_{\ell,U}.$$

Applying $p_*i_E^*$, we get an exact triangle:
$$p_*i_E^!\Q_{\ell,\widetilde{S}}\rar p_*\Q_{\ell,E}\rar i^*j_* \Q_{\ell,U}.$$

Since $E$ is a simple normal crossing divisor, \Cref{calcul penible} ensures that the perverse sheaf $\Hlp^1(p_* i_E^!\Q_{\ell,\widetilde{S}})$ is trivial. Furthermore, since $p$ induces an isomorphism over $F\setminus Z$, the only irreducible components of $E$ whose image through $p$ is $0$-dimensional have their image contained in $Z$. Therefore, \Cref{calcul penible} also ensures that $$\Hlp^2(p_* i_E^!\Q_{\ell,\widetilde{S}})=\iota_*M$$ where $M$ is a perverse sheaf over $Z$.

Hence, we have an exact sequence:
$$0\rar \Hlp^1(p_*\Q_{\ell,E})\rar \Hlp^1(i^*j_*\Q_{\ell,U})\overset{\alpha}{\rar} \iota_*M.$$

Using \cite[1.4.17.1]{bbd}, the exact and fully faithful functor $$\iota_*\colon \mathrm{Perv}(Z,\Q_\ell)\rar \mathrm{Perv}(F,\Q_\ell)$$ maps $\mathrm{Perv}(Z,\Q_\ell)$ to a Serre subcategory. Thus, the image of $\Hlp^1(i^*j_*\Q_{\ell,U})$ through $\alpha$ is of the form $\iota_*M'$ where $M'$ is an object of $\mathrm{Perv}(Z,\Q_\ell)$.

Thus, we have an exact triangle:
$$\Hlp^1(p_*\Q_{\ell,E})\rar \Hlp^1(i^*j_*\Q_{\ell,U})\rar \iota_*M'.$$
Applying $\iota^*$, we get an exact triangle: 
$$\iota^*\Hlp^1(p_*\Q_{\ell,E})\rar \iota^*\Hlp^1(i^*j_*\Q_{\ell,U})\rar M'.$$
Since $M'$ is placed in degree $0$, we get:
$$\Hlp^{-1}(\iota^*\Hlp^1(i^*j_*\Q_{\ell,U}))=\Hlp^{-1}(\iota^*\Hlp^1(p_*\Q_{\ell,E})).$$

\textbf{Step 4:} In this step we compute $\Hlp^1(p_*\Q_{\ell,E})$ and we finish the proof. 

There is a finite set $I$ and regular 1-dimensional closed subschemes $E_i$ of $\widetilde{X}$ for $i\in I$, such that $E=\bigcup\limits_{i \in I} E_i$. Choose a total order on the finite set $I$. If $i<j$, write $E_{ij}=E_i\cap E_j$. Consider for $J\subseteq I$ of cardinality at most $2$ the obvious diagram: 
$$\begin{tikzcd}
    E_J \ar[r,"u_{E_J}"] \ar[rd,swap,"p_J"] & E \ar[d,"p"] \\
    & F
\end{tikzcd}$$

\Cref{Mayer-Vietoris2} then yields an exact triangle:

$$\Q_{\ell,E} \rar\bigoplus_{i \in I}(u_{E_i})_*(u_{E_i})^* \Q_{\ell,E} \rar  \bigoplus\limits_{i< j} (u_{E_{ij}})_*(u_{E_{ij}})^* \Q_{\ell,E}.$$

Applying $p_*$, we get an exact triangle:
$$p_* \Q_{\ell,E} \rar\bigoplus_{i \in I}(p_i)_* \Q_{\ell,E_i} \rar  \bigoplus\limits_{i< j} (p_{ij})_* \Q_{\ell,E_{ij}}.$$

Since crossings only occur over $Z$, the map $p_{ij}$ factors through $Z$ for any $i<j$. Write $p_{ij}=\iota \circ \pi_{ij}$. The morphism $\pi_{ij}$ is then finite. 

Let $$I_0=\{ i \in I \mid \delta(p(E_i))=0\}.$$ 
If $i \in I\setminus I_0$, we have $\delta(E_i)=1$ and $p_i$ is finite. If $i\in I_0$, the map $p_i$ factors through $Z$ and we can write $p_i=\iota \circ \pi_i$. Finally, notice that $\delta(Z)=0$ and thus, for any $i<j$, we have $\delta(E_{ij})=0$. Hence, we get an exact sequence:

$$0 \rar \iota_* P\rar \Hlp^1(p_*\Q_{\ell,E})\rar R \oplus \iota_*Q\rar 0$$
where $$P=\coker\left(\bigoplus\limits_{i\in I_0} \Hlp^0((\pi_i)_*\Q_{\ell,E_i}) \rar \bigoplus\limits_{i< j} (\pi_{ij})_* \Q_{\ell,E_{ij}}\right),$$ $$Q=\bigoplus_{i \in I} \Hlp^1((\pi_i)_*\Q_{\ell,E_i})\text{ and }R=\bigoplus_{i \in I\setminus I_0} (p_i)_* \Q_{\ell,E_i}[1].$$

Therefore, we get an exact triangle:
$$\iota_* P\rar \Hlp^1(p_*\Q_{\ell,E})\rar R \oplus \iota_*Q$$ which yields an exact triangle:
$$P\rar \iota^*\Hlp^1(p_*\Q_{\ell,E})\rar \iota^*R \oplus Q.$$

Thus, the perverse sheaf $\Hlp^{-1}(\iota^*\Hlp^1(p_*\Q_{\ell,E}))$ over $Z$ is a subobject of the perverse sheaf $\Hlp^{-1}(\iota^*R)$. Now, we have $$\Hlp^{-1}(\iota^*R)=\bigoplus\limits_{i \in I\setminus I_0} \Hlp^0((q_i)_*\Q_{\ell,E_i^Z})$$ where for all $i\in I \setminus J$, the map $q_i\colon E_i^Z\rar Z$ is the pullback of $p_i$ along $\iota$. Since $p_i$ is finite, so is $q_i$ and thus, we have $$\Hlp^0((q_i)_*\Q_{\ell,E_i^Z})=(q_i)_*\Q_{\ell,E_i^Z}.$$ 

Thus, the perverse sheaf $\Hlp^{-1}(\iota^*R)$ is Artin. 

Hence, the perverse sheaf $\Hlp^{-1}(\iota^*\Hl^1(i^*j_*\Q_{\ell,U}))$ is also Artin.
\end{proof}

\subsection{The case of 3-folds with real closed or separably closed closed points.}
In the case of 3-folds, we will show that the perverse t-structure can be restricted to Artin $\ell$-adic complexes in some cases, and that it is impossible in other cases.

\begin{proposition}\label{main theorem 2}Let $S$ be an excellent $3$-fold such that all the closed points of $S$ have a separably closed or real closed residue field and $\ell$ be a prime number which is invertible on $S$. Then, the perverse t-structure induces a t-structure on $\mc{D}^A(S,\Z_\ell)$.
\end{proposition}
\begin{remark} In the proof, we only use that the residue fields of the closed points of $S$ have finite absolute Galois groups. Recall that a field has a finite absolute Galois group if and only if it is separably closed or real closed.
\end{remark}
\begin{proof} The proof is almost the same as the proof of \Cref{main theorem} but we need to use \Cref{outil 4} below.

Using \Cref{outil 1}, it suffices to show that if $j\colon U\rar S$ is a dense affine open immersion with $U$ regular, letting $i\colon F\rar S$ be the reduced closed complementary immersion, for any object $M$ of $\mathrm{Perv}^A(F,\Z_\ell)$, any object $N$ of $\mathrm{Perv}^{smA}(U,\Z_\ell)$ and any map $$f\colon M\rar \Hlp^{-1}(i^*j_*N),$$ 
the kernel $K$ of $f$ is an Artin $\ell$-adic complex on $F$.

The scheme $F$ is of dimension at most $2$. Let $\iota \colon Z\rar F$ be a closed immersion such that the open complementary immersion is dense and affine. 
Then, $Z$ is of dimension at most $1$ and therefore, \Cref{outil 4}(2) below ensures that the subcategory $\mathrm{Perv}^A(Z,\Z_\ell)$ of $\mathrm{Perv}(Z,\Z_\ell)$ is Serre. 

\Cref{outil 3} shows that to finish the proof, it suffices to prove that the perverse sheaf $\Hlp^{-1}(\iota^*\Hlp^{-1}(i^*j_* N))$ is Artin.

Using \Cref{outil 4}(1) below, we only need to prove that if $\eta$ is a generic point of $Z$, and $i_\eta\colon\{ \eta \} \rar Z$ is the canonical immersion, the perverse sheaf $i_\eta^*\Hlp^{-1}(\iota^*\Hlp^{-1}(i^*j_* N))$ over $k(\eta)$ is Artin.

We have a commutative diagram:

$$\begin{tikzcd}
    \{ \eta \} \ar[d,"i_\eta"] \ar[r,"u_\eta"] & \Spec(\mc{O}_{F,\eta}) \ar[r,"\psi"]\ar[d,"\xi"] & \Spec(\mc{O}_{S,\eta})\ar[d,"\alpha"]& U \times_S \Spec(\mc{O}_{S,\eta}) \ar[d,"\beta"]\ar[l,sloped,"\gamma"] \\
    Z \ar[r,"\iota"] & F \ar[r,"i"] & S & U \ar[l,sloped,"j"]
\end{tikzcd}$$
made of cartesian squares and such that the vertical maps are filtering limits of open immersions with affine transition maps. 
If $f$ is such a filtering limit, the functor $f^*$ is perverse t-exact by definition of the perverse t-structure. Thus, we have $$\begin{aligned}i_\eta^*\Hlp^{-1}(\iota^*\Hlp^{-1}(i^*j_* N))&=\Hlp^{-1}(u_\eta^*\xi^*\Hlp^{-1}(i^*j_* N))\\
&=\Hlp^{-1}(u_\eta^*\Hlp^{-1}(\psi^*\alpha^*j_* N)).\end{aligned}$$

Now, since $j$ is of finite type and since $\alpha$ is filtering limits of open immersions with affine transition maps, \Cref{exchange proetale} below yields $$i_\eta^*\Hlp^{-1}(\iota^*\Hlp^{-1}(i^*j_* N))=\Hlp^{-1}(u_\eta^*\Hlp^{-1}(\psi^*\gamma_* (\beta^*N))).$$

We can now apply \Cref{outil 2}: the scheme $\Spec(\mc{O}_{S,\eta})$ is of dimension $2$, the perverse sheaf $\beta^*N$ is a smooth Artin over $U \times_S \Spec(\mc{O}_{S,\eta})$, the open immersion $\gamma$ is affine and dense and the complementary open immersion of $u_\eta$ is also affine and dense. 

Thus, the perverse sheaf $i_\eta^*\Hlp^{-1}(\iota^*\Hlp^{-1}(i^*j_* N))$ over $k(\eta)$ is Artin. This finishes the proof.
 \end{proof}

\begin{lemma}\label{outil 4} Let $C$ be an excellent $1$-dimensional connected scheme and $\ell$ be a prime number which is invertible on $C$. Assume that the closed points of $C$ have finite absolute Galois groups. Then,
\begin{enumerate}
    \item Let $M$ be a constructible $\ell$-adic complex over $C$. Then, the complex $M$ is Artin if and only if for any generic point $\eta$ of $C$, the complex $\eta^*M$ is over $k(\eta)$ is Artin.
    \item The subcategory $\mathrm{Perv}^A(C,\Z_\ell)$ of $\mathrm{Perv}(C,\Z_\ell)$ is Serre.
\end{enumerate}
\end{lemma}
\begin{proof} We first claim that if $F$ is a $0$-dimensional scheme such that the residue fields of $F$ have finite absolute Galois groups, then $$\mc{D}^b_c(F,\Z_\ell)=\mc{D}^A(F,\Z_\ell).$$

To prove this claim, we can assume that $F$ is reduced and connected and therefore that $F=\Spec(k)$ where $k$ is a field with finite Galois group. Since every continuous representation of $G_k$ is in fact an Artin representation, \Cref{t-structure ordinaire smooth fields} then implies our claim.

We now prove the first assertion. The "only if" part follows from the usual stability properties of Artin complexes of \Cref{6 foncteurs}. We now prove the "if" part. 

Let $M$ be a constructible $\ell$-adic complex over $C$. Assume that for any generic point $\eta$ of $C$, the $\ell$-adic complex $\eta^*M$ is Artin.
The above claim implies that the $\ell$-adic complex $M|_F$ is Artin for any closed subset $F$ of $C$ of positive codimension. Thus, using \Cref{stratification}, it suffices to find a dense open subset $U$ such that the $\ell$-adic complex $M|_U$ is Artin.
By \Cref{t-structure ordinaire}, we can assume that $M$ is concentrated in degree $0$ with respect to the ordinary t-structure of $\mc{D}^b_c(C,\Z_\ell)$ \textit{i.e.} that it lies in $\Sh_c(C,\Z_\ell)$.

Let $U$ be a dense open subset such that $M|_U$ is lisse \textit{i.e.} lies in $\mathrm{Loc}_U(\Z_\ell)$. Since $C$ is excellent, we can assume that $U$ is regular. We now claim that $M|_U$ lies in $\Sh^{smA}(U,\Z_\ell)$. To show this, we can assume that $U$ is connected and therefore only has one generic point $\eta$. Since $U$ is regular, by \cite[7.4.10]{bhatt-scholze}, we have $$\pi_1^{\mathrm{pro\acute{e}t}}(U,\overline{\eta})=\pi_1^{\et}(U,\overline{\eta}).$$ 

Let $\phi$ be the canonical group homomorphism $G_{k(\eta)}\rar \pi_1^{\et}(U,\overline{\eta})$. 
The fiber functor associated to the separable closure $\overline{\eta}$ of $\eta$ induces by \Cref{description des faisceaux l-adiques lisses} a commutative square 
$$\begin{tikzcd}
\mathrm{Loc}_U(\Z_\ell)\ar[r]\ar[d,"\eta^*"]& \Rep(\pi_1^{\et}(U,\overline{\eta}),\Z_\ell)\ar[d,"\phi^*"] \\
\mathrm{Loc}_{\eta}(\Z_\ell)\ar[r]& \Rep(G_{k(\eta)},\Z_\ell)
\end{tikzcd}$$
where the horizontal arrows are equivalences. By \cite[V.8.2]{sga1}, the homomorphism $\phi$ is surjective. Thus, a continuous representation $N$ of $\pi_1^{\et}(U,\overline{\eta})$ is of Artin origin if and only if $\phi^*N$ is of Artin origin as a representation of $G_{k(\eta)}$. Hence, since $\eta^*M$ lies in $\Sh^{smA}(\eta,\Z_\ell)$, the lisse sheaf $M|_U$ is smooth Artin which proves the first assertion.

Finally, the first assertion implies the second one using that $$\mathrm{Perv}^A(C,\Z_\ell)=\mc{D}^A(C,\Z_\ell)\cap \mathrm{Perv}(C,\Z_\ell),$$ that the functor $\eta^*$ is perverse t-exact and that $\mathrm{Perv}^A(k(\eta),\Z_\ell)$ is a Serre subcategory of $\mathrm{Perv}(k(\eta),\Z_\ell)$.
\end{proof}

\begin{lemma}\label{exchange proetale} Let $\ell$ be a prime number. Let 
$$\begin{tikzcd}X'\ar[r,"g"]\ar[d,"p"]& Y' \ar[d,"q"]\\
X \ar[r,"f"]& Y
\end{tikzcd}$$ be a cartesian square of $\Z[1/\ell]$-schemes such that $f$ is a filtering limit of étale morphisms with affine transition maps and $q$ is of finite type. Then, the exchange transformation $$f^*q_*\rar p_*p^*f^*q_* \simeq p_* g^* q^*q_* \rar p_* g^*$$ is an equivalence of functors from $\mc{D}^b_c(Y',\Z_\ell)$ to $\mc{D}^b_c(X,\Z_\ell)$.
\end{lemma}
\begin{proof} If $S$ is a scheme and $\ell$ is invertible on $S$, the functor $$-\otimes_{\Z_\ell}\Z/\ell\Z\colon  \mc{D}^b_c(S,\Z_\ell)\rar \mc{D}(S_{\et},\Z/\ell\Z)$$ is conservative according to the Nakayama-Ekedahl Lemma \cite[3.6(ii)]{torsten}. Thus, the result follows from \cite[1.1.14]{em}.
\end{proof}

\subsection{A counter-example}
\begin{proposition} Let $X$ be a normal scheme of dimension $3$ with singular locus a closed point $x$ of codimension $3$. Let $\ell$ be a prime number which is invertible on $X$. Assume that the field $k(x)$ is finite. Let $f\colon \widetilde{X}\rar X$ be a resolution of singularities of $X$ and assume that the exceptional divisor $E$ is smooth over $k(x)$ and has a non-zero first $\ell$-adic betti number.

Then, taking the convention that $\delta(X)=3$, the $\ell$-adic complex $\Hlp^2(\Z_{\ell,X})$ is not an Artin $\ell$-adic complex over $X$.
\end{proposition}
\begin{proof} We have a cdh-distinguished square:
$$\begin{tikzcd}
    E \ar[d,"p"] \ar[r,"i_E"] & \widetilde{X} \ar[d,"f"]\\
    \{ x\} \ar[r,"i"] & X
\end{tikzcd}
$$
which by \Cref{cdh-descent} yields an exact triangle:
$$\Z_{\ell,X} \rar f_* \Z_{\ell,\widetilde{X}} \oplus i_*\Z_{\ell,x} \rar i_* p_* \Z_{\ell,E}.$$

Letting $j\colon U\rar X$ be the complementary open immersion of $i$, the localization triangle \eqref{colocalization} associated to the immersions $U\rar \widetilde{X}$ and $i_E$ yields after applying $f_*$ an exact triangle:
$$i_* p_* i_E^! \Z_{\ell,\widetilde{X}} \rar f_* \Z_{\ell,\widetilde{X}} \rar j_*\Z_{\ell,U}.$$

Now, since $j$ is of relative dimension $0$, the functor $j_*$ is perverse left t-exact. Since $U$ is regular and $\delta(U)=3$, the complex $j_*\Z_{\ell,U}$ is in perverse degree at least $3$. Hence, applying \Cref{calcul penible}, the perverse sheaf $\Hlp^1(f_*\Z_{\ell,\widetilde{X}})$ is trivial and $$\Hlp^2(f_*\Z_{\ell,\widetilde{X}})=i_* \Hlp^2(p_* i_E^! \Z_{\ell,\widetilde{X}}).$$

Using \cite[1.4.17.1]{bbd}, the fully faithful exact functor $i_*$ maps $\mathrm{Perv}(k(x),\Z_\ell)$ to a Serre subcategory of $\mathrm{Perv}(X,\Z_\ell)$. Therefore, we get an exact sequence:

$$0\rar i_*\Hlp^1(p_*\Z_{\ell,X}))\rar \Hlp^2(\Z_{\ell,X})\rar i_*M\rar 0$$ where $M$ is a perverse sheaf over $k(x)$. 

Hence, once again using \cite[1.4.17.1]{bbd}, there is a perverse sheaf $P$ over $k(x)$ such that $\Hlp^2(\Z_{\ell,X})=i_* P$. The perverse sheaf $\Hlp^1(p_*\Z_{\ell,X}))$ is a subobject of $P$.

It now suffices to show that $\Hlp^1(p_*\Z_{\ell,X})$ does not correspond to representation of Artin origin of $G_{k(x)}$. By \Cref{strongly of weight 0} and \Cref{weight 0}, it suffices to show that it is not of weight $0$ as a perverse sheaf over $k(x)$.

But by \cite[5.1.14, 5.4.4]{bbd}, the perverse sheaf $\Hlp^1(p_*\Z_{\ell,X})$ is pure of weight $1$. Since the first $\ell$-adic Betti number $b_1(E)$ is non-zero, the perverse sheaf $\Hlp^1(p_*\Z_{\ell,X})$ is non-zero and thus, it is not of Artin origin.
\end{proof}

\begin{example}\label{exemple dim 3} Let $E$ be an elliptic curve over a finite field $k$ and $X$ be the affine cone over $E \times \mb{P}^1_k$. Let $\ell$ be a prime number distinct from the characteristic of $k$. The scheme $X$ has a single singular point $x$ and the blow-up $\widetilde{X}$ of $X$ at $x$ is a resolution of singularities of $X$ with exceptional divisor $E \times \mb{P}^1_k$.

As $b_1(E\times \mb{P}^1_k)=1$, the proposition applies and therefore, the perverse t-structure does not induce a t-structure on $\mc{D}^A(X,\Z_\ell)$.

Finally, notice that by Noether's normalization lemma, there is a finite map $X\rar \mb{A}^3_k$. Thus, the perverse t-structure does not induce a t-structure on $\mc{D}^A(\mb{A}^3_k,\Z_\ell)$. As $\mb{A}^3_k$ is a closed subscheme of $\mb{A}^n_k$ for $n\geqslant 3$, the perverse t-structure does not induce a t-structure on $\mc{D}^A(\mb{A}^n_k,\Z_\ell)$ if $n \geqslant 3$.
\end{example}

\begin{remark}The condition that $\dim(S)\leqslant 2$ or $\dim(S)\leqslant 3$ and the residue field of closed points of $S$ are separably closed or real closed seems rather optimal: loosely speaking, if $X$ becomes more singular, the perverse cohomology sheaves of $\Z_{\ell,X}$ should become more complicated. Here, the simplest possible singularity on a scheme of dimension $3$ over a finite field already renders the cohomology sheaf not Artin.
\end{remark}

\section{The perverse homotopy t-structure}
Recall (see \Cref{t-structure generated} below) that we can formally define t-structures on a presentable stable category by setting a certain set of objects to be the t-negative objects of our t-structure. Thus, an other attempt to define a t-structure on $\mc{D}^A(S,\Z_\ell)$ is to define it by generators on $\mc{D}^A_{\In}(S,\Z_\ell)$ and then to see when this t-structure can be restricted to $\mc{D}^A(S,\Z_\ell)$. 

%We will show that in our case those two approach give the same result on $\mc{D}^A(S,\Z_\ell)$ in the sens of \Cref{homotopy=perverse} below.

\subsection{Cohomological \texorpdfstring{$\ell$}{\textell}-adic complexes and the six functors for them}\label{Dbbcoh}
\begin{definition}\label{Dbcoh} Let $S$ be a scheme and $\ell$ be a prime number invertible on $S$. The category of \emph{cohomological} (resp. \emph{ind-cohomological}) \emph{$\ell$-adic complexes over $S$} is the thick (resp. localizing) subcategory $\mc{D}^{\coh}(S,\Z_\ell)$ (resp. $\mc{D}^{\coh}_{\In}(S,\Z_\ell)$) of $\mc{D}(S,\Z_\ell)$ generated by the objects $f_*\Z_{\ell,X}$ for $f\colon X\rar S$ proper.
\end{definition}

In \cite[1.12]{plh}, Pepin Lehalleur proved that (constructible) cohomological motives are stable under some of the six operations when the base schemes have resolution of singularities by alterations. In \cite[6.2]{em}, using Gabber's method, Cisinski and Déglise showed that constructible étale motives are endowed with the six functors formalism when the base schemes are quasi-excellent. One can mimic their proof in the case of cohomological $\ell$-adic complexes. We outline how to do this:

\begin{enumerate} \item The fibred subcategory $\mc{D}^{\coh}(-,\Z_\ell)$ of $\mc{D}(-,\Z_\ell)$ is closed under the tensor product operation, negative Tate twists, the functors $f^*$ for any morphism $f$ and the functors $f_!$ for any separated morphism of finite type $f$.
\item From the absolute purity property and the stability of $\mc{D}^{\coh}$ under negative Tate twists and \Cref{Mayer-Vietoris}, we deduce that if $i$ is the immersion of a simple normal crossing divisor with regular target $X$, the $\ell$-adic complex $i^!\Z_{\ell,X}$ is cohomological.
\item We want to prove Gabber's lemma for cohomological complexes: if $X$ is quasi-excellent, then for any dense open immersion $U\rar X$, the $\ell$-adic complex $j_*\Z_{\ell,U}$ is cohomological. Notice that the fibred stable category $\mc{D}^{\coh}(-,\Z_\ell)$ has the properties \cite[6.2.9(b)(c)]{em} replacing $\mc{DM}_{h}(-,R)$ with $\mc{D}(-,\Z_\ell)$. It does not necessarily have property \cite[6.2.9(a)]{em}; however, in the proof of Gabber's lemma of \cite[6.2.7]{em}, we only need the following weaker version: for any scheme $X$, the subcategory $\mc{D}^{\coh}(X,\Z_\ell)$ of $\mc{D}(X,\Z_\ell)$ is thick and contains the object $\Z_{\ell,X}$. 
\item Follow the proofs of \cite[6.2.13, 6.2.14]{em} to prove the following result.
\end{enumerate}
\begin{proposition}
The fibered category $\mc{D}^{\coh}(-,\Z_\ell)$ over the category of quasi-excellent (noetherian) $\Z[1/\ell]$-schemes is stable under the following operations:
\begin{itemize}
    \item tensor product and negative Tate twists.
    \item $f^*$ for any morphism $f$.
    \item $f_*$ and $f_!$ for any separated and of finite type morphism $f$.
    \item $f^!$ for any quasi-finite morphism $f$.
\end{itemize}
\end{proposition}

\subsection{Definition for Ind-Artin \texorpdfstring{$\ell$}{\textell}-adic complexes}
\begin{proposition}\label{t-structure generated} Let $\mc{C}$ be a presentable stable category which admits small coproducts and push-outs. Given a family $\mc{E}$ of objects, the smallest subcategory $\mc{E}_-$ closed under extensions, small coproducts and positive shifts is the set of non-positive objects of a t-structure.
\end{proposition}
\begin{proof}\cite[1.2.1.16]{ha}.
\end{proof}

In the setting of \Cref{t-structure generated} we will call this t-structure the \emph{t-structure generated by} $\mc{E}$. In the motivic setting, t-structures given by a set of generators have been used to define the ordinary homotopy t-structure in \cite[2.2.3]{ayo07}, the perverse homotopy t-structure in \cite[2.2.3]{ayo07} or \cite{bondarko-deglise} and the ordinary motivic t-structure on $0$-motives and $1$-motives in \cite{plh}. Our perverse homotopy t-structure is similar to a "shifted" version of Pepin Lehalleur's ordinary motivic t-structure.

\begin{definition} Let $S$ be a scheme and $\ell$ be a prime number that is invertible on $S$. Then, the \emph{perverse homotopy t-structure} on $\mc{D}^A_{\In}(S,\Z_\ell)$ is the t-structure generated by the complexes $f_*\Z_{\ell,X}[\delta(X)]$ for $f\colon X\rar S$ finite.
\end{definition}
 We use the notation $\mathrm{hp}$ to denote any object related to the perverse homotopy t-structure. For instance, letting $n$ be an integer, we denote by $\Hlh^n$ the $n$-th cohomology functor with respect to this t-structure.

The same method as in \Cref{generators} shows that the perverse homotopy t-structure is also generated by the $f_!\Z_{\ell,X}[\delta(X)]$ for $f\colon X\rar S$ quasi-finite. By definition, an object $M$ of $\mc{D}^A_{\In}(S,\Z_\ell)$ lies in ${}^{\mathrm{hp}}\mc{D}^A_{\In}(S,\Z_\ell)^{\geqslant n}$ if and only if for any finite map $f\colon  X\rar S$ and any integer $p>-n$, we have $$\Hom_{\mathrm{D}(S,\Z_\ell)}(f_* \Z_{\ell,X}[\delta(X)+p],M)=0.$$

In other words, the object $M$ lies in ${}^{\mathrm{hp}}\mc{D}^A_{\In}(S,\Z_\ell)^{\geqslant n}$ if and only if the chain complex $\Map(f_*\Z_{\ell,X},M)$ is $(n-\delta(X)-1)$-connected.

To study this t-structure, we need to introduce the Artin truncation functor:

\begin{definition}Let $S$ be a scheme and $\ell$ be a prime number that is invertible on $S$. Since the category $\mc{D}^A_{\In}(S,\Z_\ell)$ is presentable by \Cref{presentability l-adique}, the adjunction theorem \cite[5.5.1.9]{htt} gives rise to a right adjoint functor to the inclusion of Ind-Artin $\ell$-adic complexes into Ind-cohomological complexes of $\ell$-adic sheaves:

$$\omega^0\colon\mc{D}^{\coh}_{\In}(S,\Z_\ell)\rar \mc{D}^A_{\In}(S,\Z_\ell).$$

We call this functor the \emph{Artin truncation functor}.
\end{definition}
\begin{remark} In the motivic setting, this functor was first introduced in \cite[2.2]{az}. 
Ayoub and Zucker predicted that the $\ell$-adic realization of their functor was a punctual weightless truncation of the $\ell$-adic realization functor.

Such a punctual weightless truncation functor $w_{\leqslant \id}$ was then introduced when $S$ is of finite type over a finite field in \cite{nv} by Vaish and Nair over the category $\mc{D}^b_m(S,\Q_\ell)$ of mixed $\ell$-adic complexes of \cite[5]{bbd}.
Finally, by giving an explicit description of Ayoub and Zucker's functor which mimics the construction of \cite{nv}, Vaish showed in \cite{vaish2} that the $\ell$-adic realization of Ayoub and Zucker's functor is the weightless truncation of its $\ell$-adic realization.

We will show that when restricted to cohomological $\ell$-adic complexes, our functor coincides with Vaish and Nair's (see \Cref{je suis pareil que Vaish et Nair} below).
\end{remark}

\begin{proposition}\label{t-adj} Let $\ell$ be a prime number. Let $f$ be a quasi-finite morphism of quasi-excellent $\Z[1/\ell]$-schemes and $g$ be morphism of $\Z[1/\ell]$-schemes which is essentially of finite type. Then, with respect to the perverse homotopy t-structure,
\begin{enumerate} 
\item The adjunction $(f_!,\omega^0f^!)$ is a t-adjunction.
\item If $M$ is t-non-positive, the functor $-\otimes_S M$ is right t-exact.
\item If $\dim(g)\leq d$, the adjunction $(g^*[d],\omega^0g_*[-d])$ is a t-adjunction.
\end{enumerate}
\end{proposition}
\begin{proof} The full subcategory of those Ind-Artin $\ell$-adic complexes $N$ such that $f_!N$, $g^*N[d]$ and $N \otimes M$ are perverse homotopy t-non-positive is closed under extensions, small coproducts and negative shifts and contains the set of generators of the t-structure. Thus, by definition of the perverse homotopy t-structure, this subcategory contains all perverse homotopy t-non-positive objects. Hence, the functors $f_!$, $g^*[d]$ and $- \otimes M$ are right t-exact with respect to the perverse homotopy t-structure and the result follows.
\end{proof}
\begin{corollary}\label{t-adj2} Let $\ell$ be a prime number. Let $f$ be a morphism of quasi-excellent $\Z[1/\ell]$-schemes. Then, with respect to the perverse homotopy t-structure,
\begin{enumerate} 
\item If $f$ is étale, the functor $f^*=f^!=\omega^0f^!$ is t-exact.
\item If $f$ is finite, the functor $f_!=f_*=\omega^0f_*$ is t-exact.
\end{enumerate}
\end{corollary}

\subsection{The smooth Artin case}
Compare the following with \Cref{t-structure perverse smooth}.
\begin{proposition}\label{t-structure h-perverse smooth} Let $S$ be a regular scheme and $\ell$ be a prime number which is invertible on $S$. Then, the perverse homotopy t-structure induces a t-structure on $\mc{D}^{smA}(S,\Z_\ell)$ which coincides with the t-structure induced by the perverse t-structure. When $S$ is connected, it coincides with the ordinary t-structure shifted by $\delta(S)$.
\end{proposition}
\begin{proof} It suffices to show that the inclusion functor $$\mc{D}^{smA}(S,\Z_\ell)\rar \mc{D}^A_{\In}(S,\Z_\ell)$$ is t-exact when the left hand side is endowed with the perverse t-structure given by \Cref{t-structure perverse smooth} and the right hand side is endowed with the perverse homotopy t-structure.

We can assume that $S$ is connected. By dévissage, it suffices to show that $$\Sh^{smA}(S,\Z_\ell)[\delta(S)]\subseteq \mc{D}^{smA}(S,\Z_\ell) \cap {}^\mathrm{hp}\mc{D}^A_{\In}(S,\Z_\ell)^\heart$$

Let $M$ be an object of $\Sh^{smA}(S,\Z_\ell)[\delta(S)]$. The object $M$ is in perverse degree $0$ by \Cref{t-structure perverse smooth}. 
If $f\colon X\rar S$ is finite, the complex $f_*\mb{Z}_{\ell,X}$ is in perverse degree at most $\delta(X)$. 
Thus, the complex $\Map(f_* \Z_{\ell,X},M)$ is $(-\delta(X)-1)$-connected.
Hence $M$ is non-negative with respect to the perverse homotopy t-structure.

We now show that $M$ is non-negative with respect to the perverse homotopy t-structure. It suffices to show that $M$ can be obtained from the set of generators of the t-structure by successively taking negative shifts, coproducts and extensions.

Let $\xi$ be a geometric point of $S$. The sheaf $M[-\delta(S)]$ corresponds to a representation of Artin origin of $\pi_1^{\mathrm{pro\acute{e}t}}(S,\xi)$ that we denote by $N$. We can assume that $N$ is an Artin representation, as representations of Artin origin are obtained as a successive extensions of Artin representations.

 Let $G$ be a finite quotient of $\pi_1^{\mathrm{pro\acute{e}t}}(S,\xi)$ through which $N$ factors. Then, by \Cref{resolution de permutation}, the complex $N\oplus N[1]$ of $\Z_\ell[G]$-modules is quasi-isomorphic to a complex $$P=\cdots \rar 0 \rar P_n \rar P_{n-1} \rar \cdots \rar P_0 \rar 0 \rar \cdots$$ where each $P_i$ is placed in degree $-i$ and is isomorphic to $\Z_\ell[G/H_1]\oplus \cdots \oplus \Z_\ell[G/H_m]$ for some subgroups $H_1,\ldots,H_m$ of $G$.
	
If $H$ is a subgroup of $G$, there is a finite étale $S$-scheme $f\colon X\rar S$ such that the image of $\Z_\ell[G/H]$ through the exact functor $$\Psi\colon\mc{D}^b(\Rep(G,\Z_\ell))\rar \mc{D}(S,\Z_\ell)$$ is $f_*\Z_{\ell,X}$. Hence, the image $Q_i$ of $P_i[\delta(S)]$ is t-non-positive with respect to the perverse homotopy t-structure. 

Furthermore, we have exact triangles $$Q_n [n]\rar M \oplus M[1] \rar Q_{<n}$$
 $$Q_{k}[k] \rar Q_{<k+1}\rar Q_{<k} \text{ if }0<k<n$$
 where $Q_{<k}[-\delta(S)]$ is the image of the truncated complex $$\cdots \rar 0 \rar P_{k-1} \rar P_{k-2} \rar \cdots \rar P_0 \rar 0 \rar \cdots$$
 through $\Psi$.

Since $Q_{<1}=Q_0$, a quick induction shows that $Q_{<k}$ is t-non-positive with respect to the perverse homotopy t-structure. Thus, the complex $M\oplus M[1]$ is t-non-positive and therefore, the complex $M$ itself is.

%the object $\mb{Z}_\ell(S)$ is in degree $\delta(S)$ for the perverse homotopy t-structure. Hence, it suffices to prove that if $f\colon X\rar S$ is finite, $\Map(f_*\mb{Z}_\ell(X),\Z_{\ell,S})$ is $(\delta(S)-\delta(X)-1)$-connected.

%But as $\Z_{\ell,S}$ is in degree $\delta(S)$ and $f_*\mb{Z}_\ell(X)$ is in degree $\leqslant \delta(X)$ for the (classical) perverse t-structure, this is true.
\end{proof}

\subsection{Gluing formalism}
\begin{proposition}\label{glueing} Let $S$ be a quasi-excellent scheme. Let $\ell$ be a prime number that is invertible on $S$. Let $i\colon F\rar S$ be a closed immersion and $j\colon U\rar S$ be the open complement. Then, the category $\mc{D}^A_{\In}(S,\Z_\ell)$ is a gluing of the pair $(\mc{D}^A_{\In}(U,\Z_\ell),\mc{D}^A_{\In}(Z,\Z_\ell))$ along the fully faithful functors $i_*$ and $\omega^0j_*$ in the sense of \cite[A.8.1]{ha}, \textit{i.e.} the functors $i_*$ and $\omega^0j_*$ have left adjoint functors $i^*$ and $j^*$ such that
\begin{enumerate}
    \item We have $j^*i_*=0$.
    \item The family $(i^*,j^*)$ is conservative.
\end{enumerate}

In particular, by \cite[A.8.5, A.8.13]{ha}, the sequence $$\mc{D}^A_{\In}(F,\Z_\ell)\overset{i_*}{\rar}\mc{D}^A_{\In}(S,\Z_\ell)\overset{j^*}{\rar}\mc{D}^A_{\In}(U,\Z_\ell)$$ satisfies the gluing formalism of \cite[1.4.3]{bbd}.
%\textit{i.e}
%\begin{enumerate}\item The functors $i_*$ and $j^*$ both have exact left and right adjoints $i^*$, $\omega^0 i^!$, $j_!$ and $\omega^0 j_*$.
%\item 
%\item The functors $i_*$, $j_*$ and $j_!$ are fully faithful.
%\item We have exact triangles:
%$$j_! j^* \rar 1 \rar i_*i^*$$ and $$i_*\omega^0 i^! \rar 1 \rar \omega^0j_* j^*.$$
%\end{enumerate}

Moreover, the perverse homotopy t-structure on $\mc{D}^A_{\In}(S,\Z_\ell)$ is obtained by gluing the t-structures of $\mc{D}^A_{\In}(U,\Z_\ell)$ and $\mc{D}^A_{\In}(F,\Z_\ell)$, (see \cite[1.4.9]{bbd}) \textit{i.e.} for all object $M$ of $\mc{D}^A(S,\Z_\ell)$, we have
\begin{enumerate}\item $M\geqslant_{\mathrm{hp}} 0$ if and only if $j^*M \geqslant_{\mathrm{hp}} 0$ and $\omega^0 i^! M \geqslant_{\mathrm{hp}} 0$.
\item $M \leqslant_{\mathrm{hp}} 0$ if and only if $j^*M\leqslant_{\mathrm{hp}} 0$ and $i^*M \leqslant_{\mathrm{hp}} 0$.
\end{enumerate}
\end{proposition}
\begin{proof} This follows from \Cref{t-adj}, \Cref{t-adj2} and the usual properties of the six functors.
\end{proof}

\begin{corollary}\label{locality} Let $S$ be a quasi-excellent scheme. Let $\ell$ be a prime number that is invertible on $S$. Let $M$ be an object of $\mc{D}^A_{\In}(S,\Z_\ell)$. Then, we have
\begin{enumerate}\item $M\geqslant_{\mathrm{hp}} 0$ if and only if there is a stratification of $S$ such that for any stratum $i\colon T\hookrightarrow S$, we have $\omega^0 i^! M \geqslant_{\mathrm{hp}} 0$.
\item $M \leqslant_{\mathrm{hp}} 0$ if and only if there is a stratification of $S$ such that for any stratum $i\colon T\hookrightarrow S$, we have $i^*M \leqslant_{\mathrm{hp}} 0$.
\end{enumerate}

In particular if $S$ is excellent and $M$ is an object of $\mc{D}^A(S,\Z_\ell)$, then, $M \leqslant_{\mathrm{hp}} 0$ if and only if there is a stratification of $S$ with nil-regular strata, such that for any stratum $i\colon T\hookrightarrow S$, the $\ell$-adic complex $i^*M$ is smooth Artin and $i^*M \leqslant_{\mathrm{p}} 0$.
\end{corollary}
\begin{proof} Recall that an excellent scheme admits a stratification with nil-regular strata. Therefore, the last point follows from \Cref{stratification} and \Cref{t-structure h-perverse smooth}.
\end{proof}

\subsection{The perverse homotopy and the perverse t-structures}

\begin{proposition}\label{perv vs hperv}Let $S$ be an excellent scheme and $\ell$ be a prime number that is invertible on $S$. Then, 
\begin{enumerate}
\item Let $M$ be an object of $\mc{D}^A(S,\Z_\ell)$, then $M$ is perverse homotopy t-negative if and only if it is perverse t-negative.
\item Let $M$ be an object of $\mc{D}^A(S,\Z_\ell)$, then if $M$ is perverse t-positive, then it is perverse homotopy t-positive.
\item If the perverse t-structure on $\mc{D}^b_c(S,\Z_\ell)$ induces a t-structure on $\mc{D}^A(S,\Z_\ell)$, then the perverse homotopy t-structure on $\mc{D}^A_{\In}(S,\Z_\ell)$ induces a t-structure on $\mc{D}^A(S,\Z_\ell)$. When this is the case, both induced t-structures coincide.
\end{enumerate}
\end{proposition}
\begin{remark}We can reformulate the first point: when the perverse homotopy t-structure induces a t-structure on $\mc{D}^A(S,\Z_\ell)$, it is final among those t-structures such that the inclusion $\mc{D}^A(S,\Z_\ell)\rar \mc{D}^b_c(S,\Z_\ell)$ is right t-exact when the right hand side is endowed with the perverse t-structure. 
\end{remark}
\begin{proof} The first point follows from \Cref{locality}(2). To prove the second assertion, let be $M$ a perverse t-positive constructible $\ell$-adic complex. Then, it is in the right orthogonal of the set of non-positive objects of the perverse t-structure. Thus, it is in the right orthogonal of the set of non-positive objects of the perverse homotopy t-structure that are Artin $\ell$-adic complexes; in particular for any finite map $f\colon X\rar S$ and any $p>-n$, the complex $M$ is right orthogonal to $f_* \Z_{\ell,X}[\delta(X)+p]$ and hence $M$ is perverse homotopy t-positive.

If the perverse t-structure induces a t-structure on $\mc{D}^A(S,\Z_\ell)$, the objects of the subcategory $$\mc{D}=\mc{D}^A_{\In}(S,\Z_\ell)^{\leqslant 0} \cap \mc{D}^A(S,\Z_\ell)$$ of $\mc{D}^A(S,\Z_\ell)$ are the t-non-positive objects of a t-structure. Furthermore, the subcategory $\mc{D}^+$ of positive objects of this t-structure is contained in the subcategory of perverse homotopy t-positive objects. But $\mc{D}^+$ contains the perverse homotopy t-positive objects since they are orthogonal to the elements of $\mc{D}$. This finishes the proof.
\end{proof}

\subsection{Comparison of the Artin truncation functor with the weightless truncation functor for schemes of finite type over a finite field and with coefficients \texorpdfstring{$\Q_\ell$}{Q\textell}}\label{je suis pareil que Vaish et Nair}

Let $S$ be a scheme of finite type over a finite field. In this paragraph and in the following ones, we set $\delta(S)=\dim(S)$ which is a consistent choice for any such $S$. Recall that Vaish and Nair introduced in \cite[3.1.1,3.1.5]{nv} an exact functor $$w_{\leqslant \id}\colon \mc{D}^b_m(S,\Q_\ell)\rar \mc{D}^b_m(S,\Q_\ell)$$ where $\mc{D}^b_m(S,\Q_\ell)$ is the category of mixed $\ell$-adic complexes of \cite[5]{bbd}. We now recall its construction and some of its fundamental properties.

First recall Morel's weight truncation of \cite{thesemorel}. If $T$ is a scheme of finite type over a finite field and if $a$ is an integer, Morel defined
\begin{enumerate}
    \item The subcategory ${}^w\mc{D}^b_m(T,\Q_\ell)^{\leqslant a}$ to be the full subcategory of $\mc{D}^b_m(S,\Q_\ell)$ made of those objects whose perverse cohomology sheaves are of weight $a$ or less.
    \item The subcategory ${}^w\mc{D}^b_m(T,\Q_\ell)^{>a}$ to be the full subcategory of $\mc{D}^b_m(S,\Q_\ell)$ made of those objects whose perverse cohomology sheaves are of weight $a$ or less.
\end{enumerate}
According to \cite[3.1.1]{thesemorel}, the pair $\left({}^w\mc{D}^b_m(T,\Q_\ell)^{\leqslant a},{}^w\mc{D}^b_m(T,\Q_\ell)^{> a}\right)$ forms a t-structure on $\mc{D}^b_m(T,\Q_\ell)$ and both categories ${}^w\mc{D}^b_m(T,\Q_\ell)^{\leqslant a}$ and ${}^w\mc{D}^b_m(T,\Q_\ell)^{> a}$ are stable. By definition, they are also closed under retract and therefore thick.

Given a stratification $\mc{S}$ of $S$, every stratum $T$ is therefore endowed with a t-structure $$\left({}^w\mc{D}^b_m(T,\Q_\ell)^{\leqslant \dim(T)},{}^w\mc{D}^b_m(T,\Q_\ell)^{>\dim(T)}\right).$$ Nair and Vaish then define a t-structure $\left({}^{\mc{S}}\mc{D}^b_m(S,\Q_\ell)^{\leqslant \id}, {}^\mc{S} \mc{D}^b_m(S,\Q_\ell)^{> \id}\right)$ on the stable category $ \mc{D}^b_m(S,\Q_\ell)$ by gluing those t-structures. Letting ${}^\mc{S} w_{\leqslant \id}$ be the non-negative truncation functor, by \cite[3.1.3]{nv}, the following properties hold. 
\begin{enumerate}
    \item The functor $$w_{\leqslant \id}=\lim\limits_{\mc{S}} {}^\mc{S} w_{\leqslant \id}$$ is well defined and we have $w_{\leqslant \id}={}^\mc{S} w_{\leqslant \id}$ for any sufficiently fine stratification $\mc{S}$.
    \item The functor $$w_{>\id}=\colim\limits_{\mc{S}} {}^\mc{S} w_{> \id}$$ is well defined and we have $w_{> \id}={}^\mc{S} w_{> \id}$ for any sufficiently fine stratification $\mc{S}$.
\end{enumerate}

Finally, by \cite[3.1.7]{nv}, the functors $(w_{\leqslant \id},w_{>\id})$ are the truncation functors with respect to t-structure $\left(\mc{D}^b_m(S,\Q_\ell)^{\leqslant \id}, \mc{D}^b_m(S,\Q_\ell)^{> \id}\right)$ and the subcategories $\mc{D}^b_m(S,\Q_\ell)^{\leqslant \id}$ and $ \mc{D}^b_m(S,\Q_\ell)^{> \id}$ are thick stable subcategories of $\mc{D}^b_m(S,\Q_\ell)$. Therefore, the induced functor $$w_{\leqslant \id}\colon \mc{D}^b_m(S,\Q_\ell)\rar \mc{D}^b_m(S,\Q_\ell)^{\leqslant \id}$$ is exact and is right adjoint to the inclusion functor $\mc{D}^b_m(S,\Q_\ell)^{\leqslant \id}\rar \mc{D}^b_m(S,\Q_\ell)$.

%We will still denote by $w_{\leqslant \id}$ the restriction of $w_{\leqslant \id}$ to cohomological $\ell$-adic complexes (see \Cref{Dbcoh}). This functor is the truncation functor with respect to the induced t-structure on $\mc{D}^{\coh}(S,\Q_\ell)$.

\begin{example}(Compare with \cite[3.11]{az}). 
    Let $k$ be a finite field. Then, the functor $w_{\leqslant \id}$ on $\mc{D}^b_m(S,\Q_\ell)$ is Morel's functor $w_{\leqslant 0}$. Therefore, if $f\colon X\rar \Spec(k)$ is a proper smooth morphism, letting $$X\rar Z\xrightarrow{g}\Spec(k)$$ be the Stein factorization of $f$, we have $$w_{\leqslant \id}f_*\Q_{\ell,X}=g_*\Q_{\ell,Z}.$$
\end{example}

\begin{proposition}\label{w<id vs omega0} Let $S$ be a scheme of finite type over $\mb{F}_p$. Let $\ell\neq p$ be a prime number. Then, the functors $\omega^0$ and $w_{\leqslant \id}$ induce equivalent functors $$\mc{D}^{\coh}(S,\Q_\ell)\rar \mc{D}^A(S,\Q_\ell)$$ on the subcategory $\mc{D}^{\coh}(S,\Q_\ell)$ of $\mc{D}^b_m(S,\Q_\ell)$.

\end{proposition}
\begin{remark}\label{corollaire de w<id vs omega0} \emph{A priori}, the functor $\omega^0$ sends cohomological $\ell$-adic complexes to Ind-Artin $\ell$-adic complexes. This proposition proves that in fact, it sends cohomological $\ell$-adic complexes to Artin $\ell$-adic complexes.
\end{remark}

\begin{proof} Recall that in \cite[5.1, 7.2.10]{em}, Cisinski and Déglise constructed a triangulated category $\mathrm{DM}(S,\Q)$ endowed with a realization functor $$\rho_\ell\colon\mathrm{DM}(S,\Q)\rar \mathrm{D}(S,\Q_\ell)$$ that commutes with the six operations by \cite[7.2.24]{em}. We can promote their construction to a stable $\infty$-categorical one following \cite{khan} and \cite{rob}.

Now, recall that Vaish constructed a functor (see \cite[3.1]{vaish2}) $w_{\leqslant \id}^{mot}$ from the category of cohomological motives to the category of Artin motives over $S$.

\begin{lemma}\label{omega0mot} Keep the same notations, then $\rho_\ell w_{\leqslant \id}^{mot}=w_{\leqslant \id} \rho_\ell$ \end{lemma}
\begin{proof}We follow the ideas of \cite[4.2.4]{vaish2}. The proof of \cite[4.2.4]{vaish2} tells us that $$\rho_\ell \mc{DM}^{\coh}(S,\Q)^{\leqslant \id} \subseteq \mc{D}^b_m(S,\Q_\ell)^{\leqslant \id}$$ and that $\rho_\ell \mc{DM}^{\coh}(S,\Q)^{> \id} \subseteq \mc{D}^b_m(S,\Q_\ell)^{> \id}$.

Applying $\rho_\ell$ to the canonical $2$-morphism $\varepsilon^{mot}\colon w_{\leqslant \id}^{mot}\rar \id$ yields a $2$-morphism $$\rho_\ell\varepsilon^{mot}\colon \rho_\ell w_{\leqslant \id}^{mot}\rar \rho_\ell.$$
Since $\rho_\ell w_{\leqslant \id}^{mot}$ factors through $\mc{D}^b_m(S,\Q_\ell)^{\leqslant \id}$, 
the functor $\rho_\ell\varepsilon^{mot}$ factors as $$\rho_\ell w_{\leqslant \id}^{mot} \overset{\alpha}{\longrightarrow} w_{\leqslant \id}\rho_\ell \overset{\varepsilon \rho_\ell}{\longrightarrow} \rho_{\ell}$$
where $\varepsilon\colon w_{\leqslant \id}\rar \id$ is the canonical $2$-morphism. 
Therefore, we have a commutative square:
$$\begin{tikzcd}\rho_\ell w_{\leqslant \id}^{mot} \ar[r,"\rho_\ell\varepsilon^{mot}"] \ar[d,"\alpha"] & \rho_\ell \ar[d,equals] \\
w_{\leqslant \id}\rho_\ell \ar[r,"\varepsilon \rho_\ell"] & \rho_\ell    
\end{tikzcd}.$$

This yields a morphism of exact triangles:
$$\begin{tikzcd}\rho_\ell w_{\leqslant \id}^{mot} \ar[r] \ar[d,"\alpha"] & \rho_\ell \ar[d,equals] \ar[r] & \rho_\ell w^{mot}_{>\id} \ar[d]\\
w_{\leqslant \id}\rho_\ell \ar[r] & \rho_\ell\ar[r] & w_{>\id} \rho_\ell    
\end{tikzcd}.$$

%We want to show that $\alpha$ is an equivalence. 
If $M$ is a cohomological morive, as $\rho_\ell w_{\leqslant \id}^{mot} M $ lies in $\mc{D}^b_m(S,\Q_\ell)^{\leqslant \id}$ and $\rho_\ell w_{> \id}^{mot} M$ lies in $\mc{D}^b_m(S,\Q_\ell)^{> \id}$, by uniqueness of the truncation, %the complex $\rho_\ell w_{\leqslant \id}^{mot} f_*\un_X $ is the truncation of $f_*\Q_{\ell,X}$ with respect to the t-structure $\left(\mc{D}^b_c(S,\Q_\ell)^{\leqslant \id}, \mc{D}^b_c(S,\Q_\ell)^{\geqslant \id}\right)$ and 
the map $\alpha(M)$ is an isomorphism. Therefore, the map $\alpha$ is an isomorphism.

\iffalse It suffices to show that $\alpha$ is an equivalence when applied to a set of generators of cohomological motives.

Now, if $X$ is a scheme, let $\un_X$ be the constant object of $\mc{DM}(X,\Q)$. The motives $f_*\un_X$ when $f\colon X\rar S$ runs through the set of proper maps generates the category of cohomological motives as a thick subcategory of $\mc{DM}(S,\Q)$.

%Applying $\rho_\ell$ to the exact triangle:
%$$w_{\leqslant \id}^{mot} f_* \un_X \rar f_* \un_X \rar w_{> \id}^{mot} f_* \un_X,$$
We have a morphism of exact triangles: 
$$\begin{tikzcd}\rho_\ell w_{\leqslant \id}^{mot} f_* \un_X \ar[r] \ar[d,"\alpha(f_*\un_X)"] & f_*\Q_{\ell,X} \ar[d,equals] \ar[r] & \rho_\ell w^{mot}_{>\id}f_* \un_X \ar[d]\\
w_{\leqslant \id}f_*\Q_{\ell,X} \ar[r] &f_*\Q_{\ell,X} \ar[r] & w_{>\id} f_*\Q_{\ell,X}   
\end{tikzcd}.$$
%$$\rho_\ell w_{\leqslant \id}^{mot} f_*\un_X \rar f_*\Q_{\ell,X} \rar \rho_\ell w_{> \id}^{mot} f_* \un_X.$$

As $\rho_\ell w_{\leqslant \id}^{mot} f_*\un_X $ lies in $\mc{D}(S,\Q_\ell)^{\leqslant \id}$ and $\rho_\ell w_{> \id}^{mot} f_* \un_X$ lies in $\mc{D}(S,\Q_\ell)^{> \id}$, by uniqueness of the truncation, %the complex $\rho_\ell w_{\leqslant \id}^{mot} f_*\un_X $ is the truncation of $f_*\Q_{\ell,X}$ with respect to the t-structure $\left(\mc{D}^b_c(S,\Q_\ell)^{\leqslant \id}, \mc{D}^b_c(S,\Q_\ell)^{\geqslant \id}\right)$ and 
the map $\alpha(f_*\un_X)$ is an isomorphism. 
\fi
\end{proof}

Using the lemma, we see that if $f\colon X\rar S$ is proper, the $\ell$-adic complex $w_{\leqslant \id} f_*\Q_{\ell,X}$ is Artin. Hence, the functor $w_{\leqslant \id}$ has its essential image included in $\mc{D}^A(S,\Q_\ell)$.

Recall that the functor $w_{\leqslant \id}\colon\mc{D}^b_m(S,\Q_\ell) \rar \mc{D}^b_m(S,\Q_\ell)^{\leqslant \id}$ is exact and is right adjoint to the inclusion functor. Now, if $g\colon Y\rar S$ is finite, the motive $g_*\Q_{\ell,Y}$ lies in $\mc{D}^{\coh}(S,\Q_\ell)^{\leqslant \id}$ by \cite[3.1.8]{nv}, hence, if $M$ is a cohomological $\ell$-adic complex, the adjunction map $$\Map(g_*\Q_{\ell,Y}, w_{\leqslant \id} M)\rar \Map(g_*\Q_{\ell,Y},M)$$
is an equivalence. Therefore, if $N$ is an ind-Artin $\ell$-adic complex and $M$ is a cohomological $\ell$-adic complex, the adjunction map $$\Map(N, w_{\leqslant \id} M)\rar \Map(N,M)$$ is an equivalence. \Cref{w<id vs omega0} follows.
\end{proof}
Using the Artin truncation functor $\omega^0$, we prove that there is a trace of the six functors formalism on Artin $\ell$-adic sheaves with coefficients in $\Q_\ell$ over a finite field.

\begin{corollary}\label{six foncteurs constructibles} Let $\ell\neq p$ be prime numbers. Then, on schemes of finite type over $\mb{F}_p$, the six functors from $\mc{D}^{\coh}(-,\Q_\ell)$ (see \Cref{Dbbcoh}) induce the following adjunctions on $\mc{D}^A(-,\Q_\ell)$:
\begin{enumerate}\item $(f^*,\omega^0 f_*)$ for any separated morphism of finite type $f$,
\item $(f_!,\omega^0 f^!)$ for any quasi-finite morphism $f$,
\item $(\otimes,\omega^0\underline{\Hom})$.
\end{enumerate}
\end{corollary}

\begin{remark}\label{EC_X 0}These functors have quite remarkable features. Let $S$ be a scheme of finite type over $\mb{F}_p$. Let $\ell\neq p$ be a prime number. Let $X$ be a scheme of finite type over $k$.

First, Vaish and Nair used them to define the weightless intersection complex $EC_X=\omega^0 j_* \Q_{\ell,U}$ with $j\colon U\rar X$ the open immersion of a regular subscheme of $X$ (see \cite[3.1.15]{nv}). This is an analog of Ayoub and Zucker's motive $\mb{E}_X$ (see \cite[2.20]{az}). 

Furthermore, if $f\colon \widetilde{X}\rar X$ is a resolution of singularities, then we have $EC_X=\omega^0 f_* \Q_{\ell,\widetilde{X}}$ (see \cite[4.1.2]{nv}). The proof amounts to the following fact: if $i\colon F\rar S$ is a closed immersion of dense complement and $S$ is regular, the functor $\omega^0 i^!$ vanishes on $\mc{D}^{smA}(S,\Q_\ell)$ (see the proof of \cite[4.1.1]{nv}).
\end{remark}

\subsection{The perverse homotopy t-structure for schemes of finite type over a finite field and with coefficients \texorpdfstring{$\Q_\ell$}{Q\textell}}

We will now show that the perverse homotopy t-structure induces a t-structure on $\mc{D}^A(S,\Q_\ell)$ when $S$ is of finite type over a finite field. To show this, we will construct a t-structure on $\mc{D}^A(S,\Q_\ell)$ by gluing the perverse t-structure on smooth Artin objects along stratifications and show that it coincides with the perverse homotopy t-structure.

\begin{proposition}\label{perverse homotopy t-structure induces a t-structure} Let $S$ be a scheme of finite type over $\mb{F}_p$. Let $\ell\neq p$ be a prime number. Then, 
\begin{enumerate} \item The perverse homotopy t-structure induces a t-structure on $\mc{D}^A(S,\Q_\ell)$. 

\item Let $i\colon F\rar S$ be a closed immersion and $j\colon U\rar S$ be the open complement. The sequence $$\mc{D}^A(F,\Q_\ell)\overset{i_*}{\rar}\mc{D}^A(S,\Q_\ell)\overset{j^*}{\rar}\mc{D}^A(U,\Q_\ell)$$ satisfies the axioms of the gluing formalism of \cite[1.4.3]{bbd} and the perverse homotopy t-structure on $\mc{D}^A(S,\Q_\ell)$ is obtained by gluing the t-structures of $\mc{D}^A(U,\Q_\ell)$ and $\mc{D}^A(F,\Q_\ell)$, (see \cite[1.4.9]{bbd}).

%\item Let $\iota:\mc{D}^A(S,\Q_\ell)\rar \mc{D}^{\coh}(S,\Q_\ell)$ be the inclusion. Then, with the perverse homotopy t-structure on the left hand side and the perverse t-structure on the right hand side, the adjunction $(\iota,\omega^0)$ is a t-adjunction and $\omega^0$ is t-exact.
\end{enumerate}
\end{proposition}
\begin{proof} We can assume that $S$ is reduced. The first part of (2) is a direct consequence of \Cref{glueing} and \Cref{six foncteurs constructibles}. To prove the rest of the proposition, we need to adapt the ideas of \cite[2.2.10]{bbd} to our setting.

We first need a few definitions:
\begin{itemize}
    \item If $\mc{S}$ is a stratification of $S$, we say that a locally closed subscheme $X$ of $S$ is an $\mc{S}$-subscheme of $S$ if $X$ is a union of strata of $\mc{S}$.
    \item A pair $(\mc{S},\mc{L})$ is \emph{admissible} if 
        \begin{itemize}
            \item $\mc{S}$  is a stratification of $S$ with smooth strata everywhere of the same dimension. 
            \item for every stratum $T$ of $\mc{S}$, $\mc{L}(T)$ is a finite set of isomorphism classes of smooth Artin perverse sheaves (see \Cref{smooth Artin perverse sheaves}).
        \end{itemize}
    \item If $(\mc{S},\mc{L})$ is an admissible pair and $X$ is an $\mc{S}$-subscheme of $S$, the category $\mc{D}_{\mc{S},\mc{L}}(X,\Q_\ell)$ of \emph{$(\mc{S},\mc{L})$-constructible Artin $\ell$-adic complexes} over $X$ is the full subcategory of $\mc{D}^A(X,\Q_\ell)$ made of those objects $M$ such that for any stratum $T$ of $\mc{S}$ contained in $X$, the restriction of $M$ to $T$ is a smooth Artin $\ell$-adic complex and its perverse cohomology sheaves are successive extensions of objects whose isomorphism classes lie in $\mc{L}(T)$. 
    \item A pair $(\mc{S}',\mc{L}')$ refines a pair $(\mc{S},\mc{L})$ if every stratum $S$ of $\mc{S}$ is a union of strata of $\mc{S}'$ and any perverse sheaf $M$ over a stratum $T$ of $\mc{S}$ whose isomorphism class lies in $\mc{L}(T)$ is $(\mc{S}',\mc{L}')$-constructible. 
\end{itemize}

If $(\mc{S},\mc{L})$ is an admissible pair and $X$ is an $\mc{S}$-subscheme of $S$, the category $\mc{D}_{\mc{S},\mc{L}}(X,\Q_\ell)$ is a stable subcategory of $\mc{D}^A(S,\Q_\ell)$ (\textit{i.e.} it is closed under finite (co)limits). 

If $(\mc{S},\mc{L})$ is an admissible pair and  $i\colon U \rar V$ is an immersion between $\mc{S}$-subschemes of $S$, the functors $i_!$ and $i^*$ preserve $(\mc{S},\mc{L})$-constructible objects. 

Now, we say that an admissible pair $(\mc{S},\mc{L})$ is \emph{superadmissible} if letting $i\colon T \rar S$ be the immersion of a stratum $T$ of $\mc{S}$ into $S$ and letting $M$ be a perverse sheaf over $T$ whose isomorphism class lies in $\mc{L}(T)$, the complex $\omega^0i_*M$ is $(\mc{S},\mc{L})$-constructible. 

The proof of \cite[3.16]{az} yields equivalences $\omega^0i_*\omega^0\simeq \omega^0 i_*$ and $\omega^0i^!\omega^0\simeq \omega^0 i^!$ for any quasi-finite morphism $i$. 

We now claim that an admissible pair $(\mc{S},\mc{L})$ can always be refined into a superadmissible one. 

Assume that this is true when we replace $S$ with the union of the strata of dimension $n$ or more. Let $i\colon T \rar S$ and $i'\colon T'\rar S$ be immersions of a strata of $\mc{S}$. Let $M$ be a perverse sheaf over $T$ whose isomorphism class lies in $\mc{L}(T)$. If $T$ and $T'$ are of dimension at least $n$, then, the perverse cohomology sheaves of $(i')^*\omega^0i_*M$ are obtained as successive extensions of objects whose isomorphism classes lie in $\mc{L}(T')$ by induction.

If $T=T'$, then $$(i')^*\omega^0i_*M=\omega^0 i^*\omega^0i_*M=\omega^0i^*i_*M=M$$ whose isomorphism class lies in $\mc{L}(T)$. 

Otherwise, if $T$ is of dimension $n-1$, we can always replace $T$ with an open subset, so that the closure of $T$ is disjoint from $T'$. Then, 
$$(i')^*\omega^0i_*M=\omega^0(i')^*i_*M=0.$$ 

Assume now that $T$ is of dimension at least $n$ and that $T'$ is of dimension $n-1$. We can replace $T'$ with an open subset so that for any $M$ in $\mc{L}(T)$, the complex $(i')^*\omega^0i_*M$ is smooth Artin by \Cref{six foncteurs constructibles} and  \Cref{stratification}. We can also add all the isomorphism classes of the $\Hlp^n\left((i')^*\omega^0i_*M\right)$ to $\mc{L}(T')$ because there are finitely many of them. This proves our claim when we replace $S$ with the union of the strata of dimension $n-1$ or more.

Hence, any admissible pair can be refined into a superadmissible one.

If $(\mc{S},\mc{L})$ is a superadmissible pair and  $i\colon U \rar V$ is an immersion between $\mc{S}$-subschemes of $S$, we claim that the functors $\omega^0i_*$ and $\omega^0i^!$ preserve $(\mc{S},\mc{L})$-constructibility. The proof is the same as \cite[2.1.13]{bbd} using as above the commutation of $\omega^0$ with the six functors.

Hence, if $(\mc{S},\mc{L})$ is a superadmissible pair, we can define a t-structure on $(\mc{S},\mc{L})$-constructible objects by gluing the perverse t-structure on the strata. By \Cref{locality}, the positive (resp. negative) objects of this t-structure are the positive (resp. negative) objects of the perverse homotopy t-structure that are $(\mc{S},\mc{L})$-constructible. 

Thus, the perverse homotopy t-structure induces a t-structure over $(\mc{S},\mc{L})$-cons\-tructible objects. Therefore, $(\mc{S},\mc{L})$-constructible objects are stable under the truncation functors of the perverse homotopy t-structure. 

Any object $M$ of $\mc{D}^A(S,\Q_\ell)$ is $(\mc{S},\mc{L})$-constructible for some superadmissible pair $(\mc{S},\mc{L})$. Indeed, by \Cref{stratification}, there is an admissible pair $(\mc{S},\mc{L})$ such that $M$ is $(\mc{S},\mc{L})$-constructible. Such a pair is can always be refined into one that is superadmissible. 

Hence, the subcategory $\mc{D}^A(S,\Q_\ell)$ of the category $\mc{D}^A_{\In}(S,\Q_\ell)$ is stable under the truncation functors of the perverse homotopy t-structure. Therefore, the perverse homotopy t-structure induces a t-structure on $\mc{D}^A(S,\Q_\ell)$.

Now, the second part (2) follows from \Cref{glueing}.
\end{proof}

\begin{definition}\label{Artin homotopy perverseerse sheaves} Let $S$ be a scheme of finite type over $\mb{F}_p$. Let $\ell\neq p$ be a prime number. We denote by $\mathrm{Perv}^A(S,\Q_\ell)^\#$ and call \emph{abelian category of Artin homotopy perverse sheaves over} $S$ the heart of the perverse homotopy t-structure on $\mc{D}^A(S,\Q_\ell)$.
\end{definition}

\begin{remark}\label{perv vs pervsharp} When both are defined, the abelian category $\mathrm{Perv}^A(S,\Q_\ell)^\#$ coincides with the abelian category $\mathrm{Perv}^A(S,\Q_\ell)$ (see \Cref{Artin perverse sheaves}). Recall that this is the case when $S$ is of finite type over the field $\mb{F}_p$ and $\dim(S)\leqslant 2$ according to our main theorem.
\end{remark}

With the assumptions of the definition, there is a right exact functor $$\iota\colon \mathrm{Perv}^A(S,\Q_\ell)^\# \rar \mathrm{Perv}(S,\Q_\ell).$$ 

It is exact if and only if the perverse t-structure induces a t-structure on Artin $\ell$-adic complexes. 

%Now, \Cref{keur} allows us to prove that the perverse t-structure induces a t-structure on $\mc{D}^{\coh}(S,\Q_\ell)$: using de Jong's resolution of singularities we can prove that if $X\rar S$ is proper, the perverse cohomology sheaves of $f_* \Q_{\ell,X}$ are cohomological by noetherian induction on $X$; the case where $X$ is smooth over $k$ is clear because the perverse cohomology sheaves of $f_*\Q_{\ell,X}$ are direct sumands of $f_*\Q_{\ell,X}$. We let $\mathrm{Perv}^{\coh}(S,\Q_\ell)$ be the heart of this t-structure.

%The functor $\iota$ factors through the abelian category $\mathrm{Perv}^{\coh}(S,\Q_\ell)$. We still denote $\iota$ the functor $\mathrm{Perv}^A(S,\Q_\ell)^\# \rar \mathrm{Perv}^{\coh}(S,\Q_\ell)$. The latter functor has an exact right adjoint:

%\begin{definition}Let $S$ be a scheme of finite type over $\mb{F}_p$. Let $\ell\neq p$ be a prime number. We call \emph{weightless truncation} the exact functor induced by $\omega^0$: $$\omega^0:\mathrm{Perv}^{\coh}(S,\Q_\ell) \rar \mathrm{Perv}^A(S,\Q_\ell)^\#.$$
%Artin homotopy perverse sheaves have similar properties as those of Artin perverse sheaves: the analog of \Cref{properties of perverse sheaves} holds.

\subsection{Weightless perverse t-structure}
In the next section, we will study a functor akin to the intermediate extension functor of \cite{bbd}. To show that this functor behaves nicely, one would like to restrict the perverse t-structure over $\mc{D}^{\coh}(S,\Q_\ell)$ and to show that the functor $\omega^0$ is t-exact. 

However, it is not known to the author of this thesis whether the perverse t-structure preserves cohomological motives. To circumvent this issue, notice however that, using \cite[5.1.7]{bbd}, the perverse t-structure of $\mc{D}^b_c(S,\Q_\ell)$ induces a t-structure on $\mc{D}^b_m(S,\Q_\ell)$ and that the heart $\mathrm{Perv}_m(S,\Q_\ell)$ of this t-structure is a Serre subcategory of $\mathrm{Perv}(S,\Q_\ell)$. Furthermore, notice that using \cite[5.1.6]{bbd}, the fibered category $\mc{D}^b_m(-,\Q_\ell)$ is endowed with the six operations. We will now define a t-structure on a certain subcategory of \emph{weightless complexes} such that the inclusion of Artin complexes into this category is t-exact and the right adjoint functor of the inclusion of weightless complexes into mixed complexes is t-exact.

\begin{definition}\label{def weightless} Let $S$ be a scheme of finite type over $\mb{F}_p$. Let $\ell\neq p$ be a prime number. We denote by $\mc{D}^{w\leqslant 0}(S,\Q_\ell)$ and call \emph{category of weightless $\ell$-adic complexes} the subcategory of the non-positive objects of the t-structure $(w_{\leqslant \id},w_{>\id})$ of Vaish and Nair.    
\end{definition}
\begin{remark} We chose the notation $\mc{D}^{w\leqslant 0}(S,\Q_\ell)$ rather than $\mc{D}^{w\leqslant \id}(S,\Q_\ell)$, because as stated in \Cref{weightless lisse} below, the punctual weight of the standard cohomology sheaves of lisse weightless complexes are non-positive.
\end{remark}

Since the functor $w_{\leqslant \id}$ of Vaish and Nair was proved in \Cref{je suis pareil que Vaish et Nair} to coincide with $\omega^0$ over cohomological motives and is right adjoint to the inclusion of weightless complexes into the category of mixed complexes, we will from this point on denote the functor $w_{\leqslant \id}$ by $\omega^0$.

We have an analog of \Cref{six foncteurs constructibles}:

\begin{proposition}\label{six foncteurs weightless} Let $\ell\neq p$ be prime numbers. Then, on schemes of finite type over $\mb{F}_p$, the six functors from $\mc{D}^b_m(-,\Q_\ell)$ induce the following adjunctions on $\mc{D}^{w\leqslant 0}(-,\Q_\ell)$:
\begin{enumerate}\item $(f^*,\omega^0 f_*)$ for any separated morphism of finite type $f$,
\item $(f_!,\omega^0 f^!)$ for any quasi-finite morphism $f$,
\item $(\otimes,\omega^0\underline{\Hom})$.
\end{enumerate}
\end{proposition}
\begin{proof}
    The first stability assertion follows from \cite[3.1.10]{nv} while the third follows from \cite[3.2.3]{nv}.

    In the case when $f$ is finite, the second stability assertion follows from \cite[3.1.8]{nv}. Using Zariski's Main Theorem \cite[18.12.13]{ega4}, it suffices to prove that if $j\colon U\rar S$ is an open immersion and $M$ is weightless over $U$, then $j_!M$ is still weightless over $S$. 

    Recall that a mixed complex $N$ defined over a scheme $X$ is weightless if and only if there is an adapted stratification $\mc{X}$ of $X$ in the sense of \cite[3.1.6]{nv} such that for each stratum $Y$ of $\mc{X}$, the perverse cohomology sheaves of the restriction of $N$ to $Y$ are of weight at most $ \dim(Y)$.

    Take an adapted stratification of $S$ such that $U$ is a union of strata. Then, perverse cohomology sheaves of the restriction of $j_!M$ to any stratum $T$ contained in $U$ is of weight at most $\dim(T)$ since $M$ is weightless, while the restriction of $j_!M$ to any stratum disjoint from $U$ vanishes. Therefore, the complex $j_!M$ is weightless, which yields the result.
\end{proof}

\begin{definition}
    Let $S$ be a scheme of finite type over $\mb{F}_p$. Let $\ell\neq p$ be a prime number. 
    \begin{enumerate}
        \item We denote by $\mathrm{Loc}^{w\leqslant 0}_S(\Q_\ell)$ and call \emph{category of weightless lisse $\ell$-adic sheaves} the full subcategory of lisse mixed $\ell$-adic sheaves over $S$ that are of non-positive weight in the sense of \cite[5]{bbd}.
        \item We denote by $\mc{D}^{w\leqslant 0}_{\mathrm{lisse}}(S,\Q_\ell)$ and call \emph{category of weightless lisse $\ell$-adic complexes} the subcategory of $\mc{D}^{w\leqslant 0}(S,\Q_\ell)$ made of those objects which are lisse as $\ell$-adic complexes.
    \end{enumerate}
\end{definition}
Since representations of weight at most $0$ of $\hat{\Z}$ form a Serre subcategory of $\Rep(\hat{\Z},\Q_\ell)$, the category 
$\mathrm{Loc}^{w\leqslant 0}_S(\Q_\ell)$ is a Serre subcategory of $\mathrm{Loc}_S(\Q_\ell)$. 

\begin{proposition}\label{weightless lisse}
    Let $S$ be a scheme of finite type over $\mb{F}_p$. Let $\ell\neq p$ be a prime number. Then, the ordinary t-structure of $\mc{D}_{\mathrm{lisse}}(S,\Q_\ell)$ induces a t-structure on $\mc{D}^{w \leqslant 0}_{\mathrm{lisse}}(S,\Q_\ell)$. The heart of this t-structure is $\mathrm{Loc}^{w\leqslant 0}_S(\Q_\ell)$. 
    
    In particular, $\mc{D}^{w \leqslant 0}_{\mathrm{lisse}}(S,\Q_\ell)$ is the subcategory of $\mc{D}^b_c(S,\Q_\ell)$ made of those complexes $C$ of $\ell$-adic sheaves such that for all integer $n$, the sheaf $\Hl^n(C)$ is weightless and lisse.
\end{proposition}
\begin{proof}
    Using \Cref{keur}, it suffices to show that $$\mc{D}^{w \leqslant 0}(S,\Q_\ell) \cap \mathrm{Loc}_S(\Q_\ell)=\mathrm{Loc}^{w \leqslant 0}_S(\Q_\ell).$$

    Let $M$ be an object of $\mathrm{Loc}^{w \leqslant 0}_S(\Q_\ell)$. It lies in $\mathrm{Loc}_S(\Q_\ell)$ by definition.    
    Furthermore, if a stratification $\mc{S}$ is adapted to $M$ in the sense of \cite[3.1.6]{nv} and $T$ is a stratum of $\mc{S}$, since $M|_T$ is lisse and $T$ is smooth, for any integer $n$, we have $$\Hlp^{n}(M|_T)=\begin{cases} M|_T[\dim(T)] &\text{if }n=\dim(T)\\
    0 & \text{otherwise}
    \end{cases}.$$ 
    
    Since the complex $M|_T$ is of weight at most $0$, the complex $M|_T[\dim(T)]$ is of weight at most $\dim(T)$. Therefore, for any integer $n$, the complex $\Hlp^{n}(M|_T)$ is of weight at most $\dim(T)$. Hence, the complex $M$ lies in $\mc{D}^{w \leqslant 0}(S,\Q_\ell)$

    Conversely, if $M$ is lies in $\mc{D}^{w \leqslant 0}(S,\Q_\ell) \cap \mathrm{Loc}_S(\Q_\ell)$, letting $\mc{S}$ be adapted to $M$ and $T$ be a stratum of $M$, the complex $\Hlp^{\dim(T)}(M|_T)=M|_T[\dim(T)]$ is of weight at most $\dim(T)$.    
    Thus, the sheaf $M|_T$ is of weight at most $0$ and therefore, the sheaf $M$ lies in $\mathrm{Loc}^{w \leqslant 0}_S(\Q_\ell)$.
\end{proof}

\begin{corollary}\label{weightless lisse perverse}
    Let $S$ be a smooth scheme over $\mb{F}_p$. Let $\ell\neq p$ be a prime number. Then, the perverse t-structure of $\mc{D}_{\mathrm{lisse}}(S,\Q_\ell)$ induces a t-structure on $\mc{D}^{w \leqslant 0}_{\mathrm{lisse}}(S,\Q_\ell)$. If $S$ is connected, the heart of this t-structure is $\mathrm{Loc}^{w\leqslant 0}_S(\Q_\ell)[\dim(S)]$. 
\end{corollary}

If $S$ is a scheme of finite type over $\mb{F}_p$ and $\ell\neq p$ is prime, we now define a t-structure on $\mc{D}^{w\leqslant 0}(S,\Q_\ell)$ by gluing the perverse t-structures over strata. Proceeding as in \Cref{perverse homotopy t-structure induces a t-structure}, we define the following:

\begin{itemize}
    \item A pair $(\mc{S},\mc{L})$ is \emph{admissible} if 
        \begin{itemize}
            \item $\mc{S}$  is a stratification of $S$ with smooth strata everywhere of the same dimension. 
            \item for every stratum $T$ of $\mc{S}$, $\mc{L}(T)$ is a finite set of isomorphism classes of weightless lisse perverse sheaves (see \Cref{weightless lisse perverse}).
        \end{itemize}
    \item If $(\mc{S},\mc{L})$ is an admissible pair and $X$ is an $\mc{S}$-subscheme of $S$, the category $\mc{D}_{\mc{S},\mc{L}}(X,\Q_\ell)$ of \emph{$(\mc{S},\mc{L})$-constructible weightless $\ell$-adic complexes} over $X$ is the full subcategory of the stable category $\mc{D}^{w \leqslant 0}(X,\Q_\ell)$ made of those objects $M$ such that for any stratum $T$ of $\mc{S}$ contained in $X$, the restriction of $M$ to $T$ is a weightless lisse $\ell$-adic complex and its perverse cohomology sheaves are successive extensions of objects whose isomorphism classes lie in $\mc{L}(T)$.
    \item An admissible pair $(\mc{S},\mc{L})$ is \emph{superadmissible} if letting $i\colon T \rar S$ be the immersion of a stratum $T$ of $\mc{S}$ into $S$ and letting $M$ be a perverse sheaf over $T$ whose isomorphism class lies in $\mc{L}(T)$, the complex $\omega^0i_*M$ is $(\mc{S},\mc{L})$-constructible. 
\end{itemize}

As before, an admissible pair $(\mc{S},\mc{L})$ can always be refined into a superadmissible one, replacing \Cref{stratification} with the fact that weightless complexes admit a stratification over which they are weightless and lisse.

If $(\mc{S},\mc{L})$ is a superadmissible pair and  $i\colon U \rar V$ is an immersion between $\mc{S}$-subschemes of $S$, the functors $i_!$, $i^*$, $\omega^0i_*$ and $\omega^0i^!$ preserve $(\mc{S},\mc{L})$-constructibility.

Hence, if $(\mc{S},\mc{L})$ is a superadmissible pair, we can define a t-structure on $(\mc{S},\mc{L})$-constructible objects by gluing the perverse t-structures on the strata.

Any object $M$ of $\mc{D}^{w \leqslant 0}(S,\Q_\ell)$ is $(\mc{S},\mc{L})$-constructible for some superadmissible pair $(\mc{S},\mc{L})$. If $(\mc{S}',\mc{L}')$ refines $(\mc{S},\mc{L})$, the t-structure defined on $\mc{D}_{\mc{S}',\mc{L}'}(X,\Q_\ell)$ induces the t-structure defined on $\mc{D}_{\mc{S},\mc{L}}(X,\Q_\ell)$. Hence, we can non-ambiguously define a t-structure on $\mc{D}^{w \leqslant 0}(S,\Q_\ell)$.

Notice that if $M$ is Artin, we can chose $(\mc{S},\mc{L})$ above to be such that the $\mc{L}(T)$ are smooth Artin for any stratum $T$ of $\mc{S}$. Therefore, our t-structure induces the perverse homotopy t-structure on $\mc{D}^A(S,\Q_\ell)$.

\begin{definition} Let $S$ be a scheme of finite type over $\mb{F}_p$. Let $\ell\neq p$ be a prime number.
\begin{enumerate}
    \item We call \emph{weightless perverse t-structure} the t-structure defined above on $\mc{D}^{w \leqslant 0}(S,\Q_\ell)$. 
    \item We denote by $\mathrm{Perv}^{w \leqslant 0}(S,\Q_\ell)$ its heart and call it the \emph{abelian category of weightless perverse sheaves}.
\end{enumerate} 
\end{definition}
We use the notation $\mathrm{wp}$ to denote the objects related to the weightless perverse t-structure.

By construction, we have:
\begin{proposition}\label{gluing weightless} Let $S$ be a scheme of finite type over $\mb{F}_p$. Let $\ell\neq p$ be a prime number.
\begin{enumerate}
    \item Let $i\colon F\rar S$ be a closed immersion and $j\colon U\rar S$ be the open complement. The sequence $$\mc{D}^{w \leqslant 0}(F,\Q_\ell)\overset{i_*}{\rar}\mc{D}^{w \leqslant 0}(S,\Q_\ell)\overset{j^*}{\rar}\mc{D}^{w \leqslant 0}(U,\Q_\ell)$$ satisfies the axioms of the gluing formalism (see \cite[1.4.3]{bbd} and \Cref{glueing}) and the weightless perverse t-structure on $\mc{D}^{w \leqslant 0}(S,\Q_\ell)$ is obtained by gluing the weightless perverse t-structures of $\mc{D}^{w \leqslant 0}(U,\Q_\ell)$ and $\mc{D}^{w \leqslant 0}(F,\Q_\ell)$ (see \cite[1.4.9]{bbd}).
\item Let $M$ be an object of $\mc{D}^{w \leqslant 0}(S,\Q_\ell)$. Then,
\begin{enumerate}\item $M\geqslant_{\mathrm{wp}} 0$ if and only if there is a stratification of $S$ such that for any stratum $i\colon T\hookrightarrow S$, we have $\omega^0 i^! M \geqslant_{\mathrm{wp}} 0$.
\item $M \leqslant_{\mathrm{wp}} 0$ if and only if there is a stratification of $S$ such that for any stratum $i\colon T\hookrightarrow S$, we have $i^*M \leqslant_{\mathrm{wp}} 0$.
\end{enumerate}
\item The inclusion functor $$\mc{D}^A(S,\Q_\ell)\rar \mc{D}^{w \leqslant 0}(S,\Q_\ell)$$
is t-exact when the left hand side is endowed with the perverse homotopy t-structure and the right hand side is endowed with the weightless perverse t-structure.
\end{enumerate}
\end{proposition}

Proceeding as in \Cref{perv vs hperv}, this implies:
\begin{corollary}\label{perv vs wperv} Let $S$ be a scheme of finite type over $\mb{F}_p$. Let $\ell\neq p$ be a prime number. Let $M$ be a weightless $\ell$-adic complex. Then, 
\begin{enumerate}
\item  The complex $M$ is weightless perverse t-negative if and only if it is perverse t-negative.
\item The complex $M$ be a weightless $\ell$-adic complex. If the complex $M$ is perverse t-positive, then it is weightless perverse t-positive.
\end{enumerate}
\end{corollary}

\begin{theorem}\label{omega^0 t-exact} Let $S$ be a scheme of finite type over $\mb{F}_p$. Let $\ell\neq p$ be a prime number.

Let $$\iota\colon\mc{D}^{w\leqslant 0}(S,\Q_\ell)\rar \mc{D}^b_m(S,\Q_\ell)$$ be the inclusion. When the left hand side is endowed with the weightless perverse t-structure and the right hand side is endowed with the perverse t-structure, the adjunction $(\iota,\omega^0)$ is a t-adjunction and the functor $\omega^0$ is t-exact.  
\end{theorem}
\begin{proof} \Cref{perv vs wperv} implies that the functor $\iota$ is right t-exact. Thus, the functor $\omega^0$ is left t-exact. But by \cite[3.2.6]{nv}, the functor $\omega^0$ is perverse right exact, \textit{i.e.} it sends perverse t-negative objects to perverse t-negative objects. But any such object is t-negative with respect to the weightless perverse t-structure using \Cref{perv vs wperv}. Thus, the functor $\omega^0$ is right t-exact and therefore, it is t-exact.
\end{proof}
We can now define the weightless truncation functor over mixed perverse sheaves. 

\begin{definition} Let $S$ be a scheme of finite type over $\mb{F}_p$. Let $\ell\neq p$ be a prime number. We call \emph{weightless truncation} the exact functor induced by $\omega^0$:  $$\omega^0\colon\mathrm{Perv}_m(S,\Q_\ell) \rar \mathrm{Perv}^{w \leqslant 0}(S,\Q_\ell).$$
\end{definition}

Finally, \Cref{omega^0 t-exact} implies an analog of the affine Lefschetz theorem for the perverse homotopy t-structure.
\begin{corollary}\label{lef aff l-adique} Let $S$ be a scheme of finite type over $\mb{F}_p$. Let $\ell\neq p$ be a prime number. Let $j\colon U\rar S$ be an affine open immersion. Then, the functors $$j_!, \omega^0j_*\colon \mc{D}^{smA}(U,\Q_\ell)\rar \mc{D}^A(S,\Q_\ell)$$
are t-exact with respect to the perverse homotopy t-structure on both sides.    
\end{corollary}
\begin{proof}
    The functor $$\omega^0j_*\colon \mc{D}^{smA}(U,\Q_\ell)\rar \mc{D}^{w  \leqslant 0}(S,\Q_\ell)$$ identifies with the composition of t-exact functors:
    $$\begin{tikzcd} \mc{D}^{smA}(U,\Q_\ell)\ar[r,"\iota"] & \mc{D}^b_m(U,\Q_\ell) \ar[r,"j_*"] & \mc{D}^b_m(S,\Q_\ell) \ar[r,"\omega^0"] & \mc{D}^{w  \leqslant 0}(S,\Q_\ell) \\
    \end{tikzcd}$$
    and is therefore t-exact using \Cref{t-structure h-perverse smooth} and \Cref{omega^0 t-exact}. Since the inclusion functor $$\mc{D}^A(S,\Q_\ell)\rar \mc{D}^{w \leqslant 0}(S,\Q_\ell)$$ is t-exact by \Cref{gluing weightless}, the result follows for $\omega^0j_*$. The case of the functor $j_!$ is similar.
\end{proof}

\subsection{Artin intermediate extension and simple Artin homotopy perverse sheaves}
We now show that the abelian category of Artin homotopy perverse sheaves shares some features with the category of perverse sheaves. 

Let $j\colon U\rar S$ be an open immersion of schemes of finite type over $\mb{F}_p$. Recall (see \cite[1.4.22]{bbd}) that the intermediate extension functor $j_{!*}$ is defined as the image of the map $\Hlp^0j_!\rar \Hlp^0 j_*$. In fact, whenever the gluing formalism of \cite[1.4.3]{bbd} is satisfied, \cite[1.4.22]{bbd} provides an intermediate extension functor. Hence, \Cref{perverse homotopy t-structure induces a t-structure} allows us to define an \emph{Artin intermediate extension functor}:
\begin{definition} Let $S$ be a scheme of finite type over $\mb{F}_p$. Let $\ell\neq p$ be a prime number. Let $j\colon U\rar S$ be an open immersion. The \emph{Artin intermediate extension functor} $j_{!*}^A$ is the functor which associates to an Artin homotopy perverse sheaf $M$ over $U$ the image of the map $$\Hlh^0 j_! M\rar\Hlh^0 \omega^0 j_* M.$$ 
\end{definition}

Since the weightless truncation functor and the inclusion of Artin homotopy perverse sheaves into weightless perverse sheaves are exact, the following proposition holds.

\begin{proposition}\label{j!*^A=omega^0j!*} Let $S$ be a scheme of finite type over $\mb{F}_p$. Let $\ell\neq p$ be a prime number. Let $j\colon U\rar S$ be an open immersion. Then, we have $j_{!*}^A=\omega^0j_{!*}$.
\end{proposition}

Using \cite[1.4.23]{bbd}, we have the following description of the intermediate extension functor (compare with \cite[2.1.9]{bbd}):
\begin{proposition}\label{carext} Let $S$ be a scheme of finite type over $\mb{F}_p$. Let $\ell\neq p$ be a prime number. Let $i\colon F\rar S$ be a closed immersion and $j\colon U\rar S$ be the open complementary immersion. Let $M$ be an Artin homotopy perverse sheaf over $U$. Then, the complex $j_{!*}^A (M)$ over $S$ is the only extension $P$ of $M$ to $S$ such that $i^* P<0$ and $\omega^0 i^! P>0$.
\end{proposition}

The following result is similar to \cite[4.3.1]{bbd}.

\begin{proposition}\label{simple} Let $S$ be a scheme of finite type over $\mb{F}_p$. Let $\ell\neq p$ be a prime number. 
\begin{enumerate}
\item The abelian category of Artin homotopy perverse sheaves over $S$ with coefficients in $\Q_\ell$ is artinian and noetherian: every object is of finite length.
\item If $j\colon V\hookrightarrow S$ is the inclusion of a smooth subscheme and if $L$ is an irreducible object of $\Sh^{smA}(V,\Q_\ell)$, then the Artin homotopy perverse sheaf $j_{!*}^A(L[\dim(V)])$ is simple. Every simple Artin perverse sheaf over $S$ is obtained in this way.
\end{enumerate}
\end{proposition}
\begin{proof} The proof is the same as in \cite[4.3.1]{bbd} replacing \cite[4.3.2, 4.3.3]{bbd} with the following lemmas.
\end{proof}
\begin{lemma} Let $V$ be a smooth connected $\mb{F}_p$-scheme. Let $\ell\neq p$ be a prime number. Let $L$ be an irreducible object of $\Sh^{smA}(V,\Q_\ell)$. Then, if $F=L[\dim(V)]$, for any open immersion $j\colon U\hookrightarrow V$, we have $F=j_{!*}^A j^*F$.
\end{lemma}
\begin{proof} We use \Cref{carext}. Let $i\colon F\rar S$ be the reduced complementary closed immersion of $j$. Then, the $\ell$-adic sheaf $i^*L$ over $F$ is smooth Artin. It is therefore in degree at most $\dim(F)$ with respect to the perverse homotopy t-structure. Thus, the $\ell$-adic complex $i^*F$ is in degree at most $\dim(F)-\dim(V)<0$.

Finally, the functor $\omega^0 i^!$ vanishes over $\mc{D}^{smA}(V,\Q_\ell)$ since $V$ is smooth (see the proof of \cite[4.1.1]{nv}).
\end{proof}

\begin{lemma} Let $V$ be a smooth connected $\mb{F}_p$-scheme. Let $\ell\neq p$ be a prime number. Let $L$ be a simple object of $\Sh^{smA}(V,\Q_\ell)$. Then, the Artin homotopy perverse sheaf $L[\dim(V)]$ is simple.
\end{lemma}
\begin{proof}The proof is exactly the same as \cite[4.3.3]{bbd}.
\end{proof}

\begin{proposition}\label{ECX} Let $X$ be a scheme of finite type over $\mb{F}_p$. Let $\ell\neq p$ be a prime number. Recall the construction of the weightless intersection complex $EC_X$ of Vaish and Nair (see \Cref{EC_X 0}). Let $j\colon U\rar X$ be an open immersion with $U$ regular. Then, we have 
\begin{enumerate} \item $EC_X=j_{!*}^A \Q_{\ell,U}[\dim(U)]$.
\item $EC_X$ is a simple Artin homotopy perverse sheaf.
\item If $X$ is a surface, $EC_X$ is an Artin perverse sheaf; in particular, it is a perverse sheaf.
\end{enumerate}

\end{proposition}
\begin{proof} Let $IC_X=j_{!*}\Q_{\ell,U}[\dim(U)]$ be the intersection complex, then $$EC_X=\omega^0 IC_X$$ by \cite[3.1.13]{nv}. This proves (1) by \Cref{j!*^A=omega^0j!*}; (2) follows by the description of simple objects given in \Cref{simple} and (3) follows from the comparison between Artin perverse sheaves and Artin homotopy perverse sheaves of \Cref{perv vs pervsharp}. 
\end{proof}

However, when $X$ is a surface, the weightless intersection complex $EC_X$ need not be simple in the category of perverse sheaves:
\begin{example}\label{exemple ECX} Let $X$ be a normal surface with a single singular point $i\colon\{ x\} \rar X$. Let $f\colon Y\rar X$ be a resolution of singularities such that the fiber of $f$ above $x$ is a simple normal crossing divisor with components of positive genus. Then, $EC_X$ is not simple as a perverse sheaf.

Indeed, let $i_E\colon E\rar X$ be the pull-back of $i$ along $f$ and $p\colon E\rar \{ x\}$ be the pull-back of $f$ along $i$. 

We have exact triangles:

$$EC_X\rar f_*\Q_{\ell,Y}[2]\rar w_{>\id}f_*\Q_{\ell,Y}[2]$$
and $$w_{>\id}j_!\Q_{\ell,U}\rar w_{>\id}f_*\Q_{\ell,Y}\rar w_{>\id}i_* p_*\Q_{\ell,E}.$$

But as $j_!\Q_{\ell,U}$ is Artin, we have $w_{>id}j_!\Q_{\ell,U}=0$. 

In addition, by \cite[3.1.8, 3.1.10]{nv}, the functor $i_*$ is t-exact with respect to the t-structure $\left(\mc{D}^{\leqslant \id}(-,\Q_\ell), \mc{D}^{> \id}(-,\Q_\ell)\right)$. Thus, $i_*$ commutes with $w_{>\id}$. But over $\mc{D}^b_c(k(x),\Q_\ell)$, the functor $w_{>\id}$ is defined as Morel's weight truncation functor $w_{\geqslant 1}$ (see \cite[4.1]{thesemorel}).

Now, write $E=\bigcup\limits_{i\in I} E_i$ where $E_i$ is regular. If $J\subseteq I$, write $E_J=\bigcap_{i\in J} E_i$ and $p_J\colon E_J\rar \{ x\}$ the structural morphism. \Cref{Mayer-Vietoris2} gives an exact triangle $$p_*\Q_{\ell,E}\rar \bigoplus\limits_{i \in I} (p_i)_*\Q_{\ell,E_i} \rar \bigoplus\limits_{J\subseteq I, |J|=2} (p_J)_*\Q_{\ell,E_J}.$$

If $|J|=2$, the $\ell$-adic complex $(p_J)_*\Q_{\ell,E_J}$ is pure of weight $0$. Therefore, applying Morel's truncation functor $w_{\geqslant 1}$, we get an isomorphism:

$$\begin{aligned} w_{\geqslant 1} p_*\Q_{\ell,E}&=\bigoplus\limits_{i \in I} w_{\geqslant 1}(p_i)_*\Q_{\ell,E_i}\\
&=\bigoplus\limits_{i\in I} R^1(p_i)_*\Q_{\ell,E_i}[-1]\oplus\bigoplus\limits_{i\in I} R^2(p_i)_*\Q_{\ell,E_i}[-2].
\end{aligned}$$

 Hence, we have $$w_{>id}i_* p_*\Q_{\ell,E}=i_*\left(\bigoplus\limits_{i\in I} R^1(p_i)_*\Q_{\ell,E_i}[-1]\oplus\bigoplus\limits_{i\in I} R^2(p_i)_*\Q_{\ell,E_i}[-2]\right).$$

Now, using \cite[1.8.1]{surveyperversesheaves}, we have $$f_*\Q_{\ell,Y}=IC_X\oplus i_*\left(\bigoplus\limits_{i\in I} R^2(p_i)_*\Q_{\ell,E_i}\right)[-2].$$

Hence, we have an exact sequence of perverse sheaves: $$0\rar i_*\left(\bigoplus\limits_{i\in I} R^1(p_i)_*\Q_{\ell,E_i}\right)\rar EC_X\rar IC_X\rar 0.$$

Therefore, since one of the $E_i$ has positive genus, $i_*\left(\bigoplus\limits_{i\in I} R^1(p_i)_*\Q_{\ell,E_i}\right)$ is non-zero and thus $EC_X$ is not simple.
\end{example}

\section{Explicit description of Artin perverse sheaves over schemes of dimension \texorpdfstring{$1$}{1}}
\sectionmark{Explicit desctiption in dimension \texorpdfstring{$1$}{1}}
\subsection{Perverse sheaves as gluing of representations over schemes of dimension \texorpdfstring{$1$}{1}}
The material in this section is a variation on Beilinson's method described in \cite{howtoglue}. Let $M$ be a perverse sheaf over an excellent base scheme $S$. Recall that there is a stratification $\mc{S}$ of $S$ such that for all $T\in \mc{S}$, $M|_T$ is lisse. We want to recover $M$ from the $M|_T$ and additional data. This is possible using the nearby and vanishing cycle functors in general. However, in the case when $S$ is $1$-dimensional, we have another way of doing this.

We take the convention that $\delta(S)=1$. Assume for simplicity that $S$ is reduced. The scheme $S$ is then the disjoint union of two regular strata $U$ and $F$ over which $M$ is lisse and such that $U$ is open and dense. Shrinking $U$, we can also assume that the open immersion $j\colon U\rar S$ is affine by \Cref{reduction affine}. Write $i\colon F\rar S$ the closed immersion. The $\ell$-adic complex $M|_U$ is a lisse perverse sheaf and the $\ell$-adic complex $M|_F$ lies in perverse degrees $[-1,0]$ by \cite[4.1.10]{bbd}.

By localization, we have an exact triangle:

$$j_!M|_U\rar M\rar i_*M|_F$$ which gives rise to an exact sequence: $$0\rar i_*\Hlp^{-1}M|_F\rar j_!M|_U\rar M \rar i_* \Hlp^0 M|_F\rar 0.$$

Hence, we can recover $M$ as the cone of a map $i_*M|_F[-1]\rar j_!M|_U$ such that the induced map $i_*\Hlp^{-1}M|_F\rar j_!M|_U$ is a monomorphism. 

Applying the localization triangle \eqref{localization} to the functor $j_*$, we have an exact triangle:

$$j_!\rar j_* \rar i_* i^* j_*.$$ Moreover, we have $j^*i_*=0$. Therefore, we have $$\Hom(i_*M|_F[-1],j_!M|_U)=\Hom(M|_F, i^*j_*M|_U)$$ and we have an exact sequence:

$$0\rar i_*\Hlp^{-1}(i^*j_* M|_U) \rar j_! M|_U \rar j_* M|_U \rar i_* \Hlp^0(i^*j_* M|_U)\rar 0.$$

Therefore, a map $i_*\Hlp^{-1}M|_F\rar j_!M|_U$ is a monomorphism if and only if the $\Hlp^{-1}$ of the corresponding map $M|_F\rar i^*j_* M|_U$ is a monomorphism.

Conversely, a lisse perverse sheaf $M_U$ over $U$, a (lisse) $\ell$-adic complex $M_F$ over $F$ which lies in perverse degrees $[-1,0]$ and a map $M_F\rar i^*j_*M_U$ such that the induced map $\Hlp^{-1}(M_F)\rar  \Hlp^{-1}(i^*j_* M|_U)$ is a monomorphism, give rise to a unique perverse sheaf $M$ over $S$.

Since $\delta(S)=1$, an $\ell$-adic complex $N$ is a lisse perverse sheaf over $U$ if and only if $N[-1]$ is a lisse sheaf. Moreover, since $\delta(F)=0$, the perverse and ordinary t-structures coincide over $\mc{D}^b_c(F,\Z_\ell)$. Hence, the following data determine a unique perverse sheaf over $S$:
\begin{enumerate}
\item A dense affine open immersion $j\colon U\rar S$ with $U$ regular. Denote by $i\colon F\rar S$  the complementary reduced closed immersion. 
\item A locally constant $\ell$-adic sheaf $M_U$ over $U$.
\item A (lisse) $\ell$-adic complex $M_F$ over $F$ placed in (ordinary) degrees $[-1,0]$.
\item A map $\phi\colon M_F\rar i^*j_*M_U$ such that $\Hlp^{-1}(\phi)$ is a monomorphism.
\end{enumerate}

We can compute $i^*j_* M_U$ more explicitly. The first case to consider is the case of a regular $1$-dimensional scheme. We first need some notations:

\begin{notations}\label{notations regular}Let $C$ be a regular connected $1$-dimensional scheme and $K$ be its field of functions. Then, if $U$ is an open subset of $C$ and $x$ is a closed point of $C$.
\begin{itemize}
%\item $K^U$ be the maximal extension of $K$ which is unramified over $U$.
\item We let $K_x$ be the completion of $K$ with respect to the valuation defined by $x$ on $K$.
\item We let $\phi_x\colon\Spec(K_x)\rar \Spec(K)$ be the induced map.
%\item $\overline{K_x}$ be a separable closure of $K_x$ and $K_x^{\mathrm{nr}}$ be the maximal extension of $K_x$ which is unramified at $x$.
\item We let $G_0(x)$ be inertia subgroup of $G_{K_x}$.
%\item $K_x^U$ be the maximal extension of $K_x$ which is unramified on $U$.
\end{itemize}
\end{notations}
The following lemma is classical (use \cite[V.8.2,X.2.1]{sga1}):

\begin{lemma}\label{lemma regular} Let $C$ be a regular connected $1$-dimensional scheme of generic point $\eta$, $U$ an open subset of $C$ and $x$ a closed point. Let $K=k(\eta)$ be the field of regular functions on $C$. Then,
\begin{enumerate}%\item $K_x^U$ is a completion of $K^U$ with respect to the valuation defined by $x$.
%\item If $x\in U$, $K_x^U=K_x^{\mathrm{nr}}$.
%\item If $x\notin U$, $K_x^U=\overline{K_x}$.
%\item $\pi_1^{\et}(U)$ can be identified with $\mathrm{Gal}(K^U/K)$.
\item We have an exact sequence of groups: 
$$1\rar G_0(x)\rar G_{K_x}\rar G_{k(x)}\rar 1$$
%can be identified with $\mathrm{Gal}(K_x^{\mathrm{nr}}/K_x)=G_{K_x}/G_0(x)$.
\item If $x\notin U$, the map 
$$G_{K_x}\rar G_K \rar \pi_1^{\et}(U,\overline{\eta})$$
is injective.
\end{enumerate}

In this setting, we have a morphism of sites $(BG_{K_x})_{\mathrm{pro\acute{e}t}}\rar (BG_{k(x)})_{\mathrm{pro\acute{e}t}}$ which sends the continuous $G_{K_x}$-set $M$ to the continuous $G_{k(x)}$-set $ M^{G_0(x)}$. This morphism then induces a functor $$\mc{D}^b_{\Z_\ell-\mathrm{perf}}(\Sh((BG_{K_x})_{\mathrm{pro\acute{e}t}},\Z_\ell))\rar \mc{D}^b_{\Z_\ell-\mathrm{perf}}(\Sh((BG_{k(x)})_{\mathrm{pro\acute{e}t}},\Z_\ell))$$

that we will denote by $\partial_x$.
\end{lemma}

\begin{proposition}\label{i^*j_* regular} Let $C$ be a regular connected $1$-dimensional scheme of generic point $\eta$ and field of regular functions $K$. Let  $i\colon F\rar C$ be a reduced closed immersion and let $j\colon U\rar C$ be the open complement. If $x \in F$, let $i_x\colon \{ x\}\rar F$ be the closed immersion. Finally, let $\ell$ be a prime number which is invertible on $C$. 

Then, the diagram 
% https://tikzcd.yichuanshen.de/#N4Igdg9gJgpgziAXAbVABwnAlgFyxMJZABgBpiBdUkANwEMAbAVxiRAB12BbOnACwBOXYABkIAYwC+AfQCqACk4AtaZxgMGAShCTS6TLnyEUZAIxVajFm049+Q4ACUYaSYvZos00wD1ganDdZUk4IGhgBBiwwGH92GBw6SU0Q9hU1DW1dfWw8AiIyACYLemZWRA5uXkFhZ1d5AHFpAGlU9PjMnT0QDFyjIlNSYupS6wrONAFoAAJOKK5cOGlgAA9Z9mjpgDFJW2qHOrcm4GbpFd1lVQ6tLpzDfJRCoZKrcsq7PnFGYAARSR8AEbuD4HFxHZYAa3kK00FzSV3UN2yPQMeWMyCelBGrxsVX4XwYv3+AOk4nkWzaCM6kgsMCgAHN4ERQAAzKZcJBkEA4CBIQYgBh0AHqAAKqP6FQEWHpfBwIGxZVxCToPgAVLcQGyIBzEPyeUgnty6FgGGw+BAIBCNVqdYb9YgAMwKsaVSYzOZYBY4JardabHbrNB8LwrNXW9lIJ3c3mIAAszreEymUHW80WyxWnH9kkDdAEeEYZ3D2s51HtAFYE2wsGqAFbSdXIm1IePRpCVgVC0Xih4gKUyuVV8bsAHSiBoZhwVOe9O+rNgbY5+QhzQNnQUSRAA

\adjustbox{scale=0.9,center}{$$\begin{tikzcd}
\mathrm{Loc}_U(\Z_\ell) \arrow[d, "\overline{\eta}^*"'] \arrow[rr, "i^*j_*"]         &                                                                                                       & {\mathcal{D}^b_c(F,\Z_\ell)}                                                                      \\
{\mathrm{Rep}(\pi_1^{\et}(U,\overline{\eta}),\Z_\ell)} \arrow[d, hook]    &                                                                                                       &                                                                                                   \\
{\mathrm{Rep}(G_K,\Z_\ell)} \arrow[r, "\prod \limits_{x \in F} \phi_x^*"] & {\prod \limits_{x \in F}\mathrm{Rep}(G_{K_x},\Z_\ell)} \arrow[r, "\prod \limits_{x\in F} \partial_x"] & {\prod \limits_{x\in F}\mc{D}^b_{\Z_\ell-\mathrm{perf}}(\Sh((BG_{k(x)})_{\mathrm{pro\acute{e}t}},\Z_\ell))} \arrow[uu, "\bigoplus \limits_{x \in F} (i_x)_*"']
\end{tikzcd}$$}
is commutative.
%the functor $$i^*j_*\colon \Rep^A(\pi_1^{\et}(U),\Z_\ell) \rar \mc{D}^b_c(F,\Z_\ell)$$ coincides with the functor $\bigoplus \limits_{x\in F} (i_x)_* (R\Psi_x) \phi_x^*$, where $i_x\colon \{ x\} \rar F$ is the inclusion.
\end{proposition}
\begin{proof}First, note that $i^*j_*=\bigoplus\limits_{x\in F} (i_x)_* i_x^* i^* j_*$. We can therefore assume that $F$ has a single point $x$.

Then, by \cite[VIII.5]{sga4} and \cite[5.5.4]{bhatt-scholze}, if $n$ is an integer and $M$ lies in $\mathrm{Loc}_U(\Z_\ell)$, the pro-étale sheaf $\Hl^n(i^*j_*M)$ over $\{ x\}$ identifies with the $G_{k(x)}$-module: 

$$\Hl^n_{\et}(\Spec(L),\phi_x^*\overline{\eta}^* M)=\Hl^n(G_0(x),\phi_x^*\overline{\eta}^*M)$$ where $L$ is the maximal unramified extension of $K_x$. This yields the result.
\end{proof}

We will now tackle the case of a general $1$-dimensional scheme.

\begin{notations}\label{notations general} Let $C$ be an excellent $1$-dimensional scheme. Let $\nu\colon\widetilde{C}\rar C$ be the normalization of $C$. Let $\Gamma$ be the set of generic points of $\widetilde{C}$. Let $U$ is an open subset of $C$, let $y$ be a closed point of $\widetilde{C}$ and let $\eta\in \Gamma$.

\begin{itemize}\item We let $U_\eta=\overline{\{\eta\}}\cap \nu^{-1}(U)\subseteq \widetilde{C}$.
\item We let $K_\eta$ be the field of regular functions on $C_\eta$.
%\item $K^U_\eta$ be the maximal extension of $K_\eta$ which is unramified on $U_{\eta}$.
\item We let $\eta(y)$ be the unique element of $\Gamma$ such that $y\in C_{\eta(y)}$.
\item We let $K_y$ be the completion of $K_{\eta(y)}$ with respect to the valuation defined by $y$ on $K_{\eta(y)}$.
\item We let $\phi_y\colon \Spec(K_y)\rar \Spec(K_{\eta(y)})$ be the induced map.
%\item $\overline{K_y}$ be a separable closure of $K_y$ and $K_y^{\mathrm{nr}}$ be the maximal extension of $K_y$ which is unramified at $y$.
\item We let $G_0(y)$ be the inertia subgroup of  $G_{K_y}$.
%\item $K_y^U$ be the maximal extension of $K_y$ which is unramified on $U_{\eta(y)}$.
\end{itemize}
\end{notations}

The following lemma follows from \Cref{lemma regular}.
\begin{lemma}\label{lemma not regular} Let $C$ be a $1$-dimensional excellent scheme. Let $\nu\colon\widetilde{C}\rar C$ be the normalization of $C$. Let $\Gamma$ be the set of generic points of $\widetilde{C}$. Let $U$ be a nil-regular open subset of $C$, let $y$ be a closed point of $\widetilde{C}$ and let $\eta\in \Gamma$. Then,
\begin{enumerate}
%\item $K_y^U$ is a completion of $K^U_{\eta(y)}$ with respect to the valuation defined by $y$.
%\item If $y\in U_{\eta(y)}$, $K_y^U=K_y^{\mathrm{nr}}$.
%\item If $y\notin U_{\eta(y)}$, $K_y^U=\overline{K_y}$.
%\item $\pi_1^{\et}(U_\eta)$ can be identified with $\mathrm{Gal}(K^U_\eta/K_\eta)$.
%\item $G_{k(y)}$ can be identified with $\mathrm{Gal}(K_y^{\mathrm{nr}}/K_y)=G_{K_y}/G_0(y)$.
\item We have an exact sequence of groups: 
$$1\rar G_0(y)\rar G_{K_y}\rar G_{k(y)}\rar 1$$

\item If $y\notin U_{\eta(y)}$, the map
$$ G_{K_y}\rar G_{K_{\eta(y)}}\rar \pi_1^{\et}(U_{\eta(y)},\overline{\eta(y)})$$
is injective.

%$$\begin{tikzcd}G_{K_y} \ar[r,hook]\ar[d,twoheadrightarrow]& \mathrm{Gal}(K^U_{\eta(y)}/K_{\eta(y)})=\pi_1^{\et}(U_{\eta(y)})\\
%\mathrm{Gal}(K^C_y/K_y)=G_{k(y)}&
%\end{tikzcd}$$
%\item A smooth Artin perverse sheaf on $U$ is the same as a direct sum $\bigoplus\limits_{\eta\in \Gamma} M_\eta$ where each $M_\eta$ is an Artin representation of $\pi_1(U_\eta)=\mathrm{Gal}(K^U_\eta/K_\eta)$.

\end{enumerate}

In this setting, we have a morphism of sites $(BG_{K_y})_{\mathrm{pro\acute{e}t}}\rar (BG_{k(y)})_{\mathrm{pro\acute{e}t}}$ given by $M\mapsto M^{G_0(y)}$. This morphism then induces a functor $$\mc{D}^b_{\Z_\ell-\mathrm{perf}}(\Sh((BG_{K_y})_{\mathrm{pro\acute{e}t}},\Z_\ell))\rar \mc{D}^b_{\Z_\ell-\mathrm{perf}}(\Sh((BG_{k(y)})_{\mathrm{pro\acute{e}t}},\Z_\ell))$$

that we will denote by $\partial_y$.
\end{lemma}

\begin{corollary}\label{i^*j_*} Let $C$ be a $1$-dimensional excellent scheme and let $\ell$ be a prime number which is invertible on $C$. Let $\nu\colon\widetilde{C}\rar C$ be the normalization of $C$. Let $\Gamma$ be the set of generic points of $C$. Let $i\colon F\rar C$ be a closed immersion and $j\colon U\rar C$ be the open complement. Assume that $U$ is nil-regular. 

Then, the diagram
% https://tikzcd.yichuanshen.de/#N4Igdg9gJgpgziAXAbVABwnAlgFyxMJZABgBpiBdUkANwEMAbAVxiRAB12BbOnACwBOXYABkIAYwC+AfQCqACk4AtaZxgMGAShCTS6TLnyEUZAIxVajFm05oB0AAScGWLrjjTganHSfssYH4A4nRcPJKcPPxCwABKMGiSiuxoWNKmAHpe7DA4SbKe3nS6nBA0MAIuYDDZucWapMqqORrauvrYeAREZABMFvTMrIgcKfZQfi5uOB61Pn4BwaHhkbyCwvGJ8kGeANLNdSXsKmqtOnogGJ1GPaQAzANWw6N2js6u7oU585yLnCFhYq2cbNKbuYAAT1+gVK5UqARqnGAfjqnEkEXY4joaD8YCYWQAtKYkgAxTQYqLrOIJJI7YD7CFHE4tLTnDqGbomUgAFkeQxsYze7DBM08EIWMPYeMJxPkZMkfkpMU2tL20gAHkyDmd2pcDF1jMhTDy+dYRsChSLZuLobj8cAiaTyYq1liGMAACKSDIAI2SSo2NO2auAAGt5BDyVrTlo2hcrhzDb0TdRBmaXuNJh9RcB1baSRjXhN3tNZpw8RHNABeTUu-huz3ev2raKBrZ0-ZhyvoxrHbWxtl666c5DJ-qpp4Coslz65-MU12MRu+-1rZVBunh9Xk3vM9Ss3UJg1EZOUCf883cRfur2+6TiOW7-ttCwwKAAc3gRFAADN7FwkDIEAcAgJBjRABg6B9dQAAV9RuEYBCwd8+BwEBz3TC1i2FbMy2+XxbX+ZYgXYMoKiqRF8O9AAqQc-wgADEHAkCkGTYC6CwBg2D4CAIFDOj-1Y6gWMQO4MOeLCs1LL46glJZAULEEZxmSFCNIuEKOyZEijRDEsRxct7UdOVnVsPg0ghDJaOoSDoIYODh2MEAkJQtDdXoxixOA0DEG5dzBN84SfIAVnEqdM2U2Y838SVpQdWV5T8NA6AEPBGGkCEBIYpBQu8pAADZ-OyxB8qCpAAHYwsvIspNndU5ILPwfWQiA0GYOBIq+CtIxrBdW2AABJMAoG9YBN27GQxs8LcowcdCIKg2D4M5ZzkNQrLGMqvLEAADiq0ZmvfVr2tqnN6vnBx5DSbdpGsha7IcxM2Bc9aisYoCRL2ywLxALArIAK1unQKEkIA

\adjustbox{scale=0.8,center}{%
$$\begin{tikzcd}
\mathrm{Loc}_U(\Z_\ell) \arrow[d, "\prod \limits_{\eta \in \Gamma}\overline{\eta}^*"'] \arrow[rr, "i^*j_*"] & & {\mathcal{D}^b_c(F,\Z_\ell)}\\
{\prod \limits_{\eta \in \Gamma}\mathrm{Rep}(\pi_1^{\et}(U_{\eta},\overline{\eta}),\Z_\ell)} \arrow[d, hook]&&\\
{\prod \limits_{\eta \in \Gamma}\mathrm{Rep}(G_{K_\eta},\Z_\ell)} \arrow[d, "\prod \limits_{\eta \in \Gamma}\prod\limits_{y\in \overline{\{ \eta\}}\cap \nu^{-1}(F)} \phi_y^*"'] && {\prod\limits_{x\in F}\mc{D}^b_{\Z_\ell-\mathrm{perf}}(\Sh((BG_{k(x)})_{\mathrm{pro\acute{e}t}},\Z_\ell))} \arrow[uu, "\bigoplus \limits_{x \in F} (i_x)_*"'] \\
{\prod \limits_{\eta \in \Gamma}\prod\limits_{y\in \overline{\{ \eta\}}\cap \nu^{-1}(F)}\mathrm{Rep}(G_{K_y},\Z_\ell)} \arrow[d,equals]&&\\
{\prod \limits_{y \in \nu^{-1}(F)} \mathrm{Rep}(G_{K_y},\Z_\ell)} \arrow[r, "\prod \limits_{x\in \nu^{-1}(F)} \partial_y"]& {\prod \limits_{y \in \nu^{-1}(F)} \mc{D}^b_{\Z_\ell-\mathrm{perf}}(\Sh((BG_{k(y)})_{\mathrm{pro\acute{e}t}},\Z_\ell))} \arrow[r,equals] & {\prod \limits_{x\in F}\prod \limits_{\nu(y)=x} \mc{D}^b_{\Z_\ell-\mathrm{perf}}(\Sh((BG_{k(y)})_{\mathrm{pro\acute{e}t}},\Z_\ell))} \arrow[uu, "\prod \limits_{x \in F} \bigoplus\limits_{\nu(y)=x}\mathrm{Ind}^{G_{k(y)}}_{G_{k(x)}} "']
\end{tikzcd}$$}
is commutative.
%the functor $$i^*j_*:\Sh^A(U,\Z_\ell) \rar \mc{D}^b_c(F,\Z_\ell)$$ can be identified with the functor that sends $\bigoplus\limits_{\eta\in \Gamma} M_\eta$ where each $M_\eta$ is an Artin representation of $G_{K_\eta}$ to $$\bigoplus\limits_{x\in F} (i_x)_* \bigoplus\limits_{\nu(y)=x}\mathrm{Ind}_{G_{k(y)}}^{G_{k(x)}} \left( R\Psi_y\left[\phi_y^*(M_{\eta(y)})\right]\right) .$$ 
\end{corollary}
\begin{proof}Let $\nu_U\colon U_{\mathrm{red}}\rar U$ (resp. $\nu_F\colon \widetilde{F}\rar F$) be the inverse image of $\nu$ along $j$ (resp. $i$). Let $\widetilde{j}$ (resp. $\widetilde{i}$) be the inclusion of $U_{\mathrm{red}}$ (resp. $\widetilde{F}$) in $\widetilde{C}$.

Recall that $\nu_U^*$ is an equivalence of categories. Thus, $$i^*j_*=i^*j_* (\nu_U)_* \nu_U^*=i^* \nu_* \widetilde{j}_* \nu_U^*=(\nu_F)_* \widetilde{i}^*\widetilde{j}_* \nu_U^*.$$

The result then follows from \Cref{i^*j_* regular} and \Cref{6 foncteurs rpz}.
\end{proof}

This discussion leads us to the following definition:

\begin{definition}\label{P(S)} Let $C$ be a $1$-dimensional excellent scheme and let $\ell$ be a prime number invertible on $C$. Let $\Gamma$ be the set of generic points of $C$. Let $\nu\colon\widetilde{C}\rar C$ be the normalization of $C$. We define an abelian category $\mathrm{P}(C,\Z_\ell)$ as follows:
\begin{itemize}\item An object of $\mathrm{P}(C,\Z_\ell)$ is a quadruple $(U,(M_\eta)_{\eta\in \Gamma},(M_x)_{x\in C \setminus U},(f_x)_{x\in C\setminus U})$ where 
\begin{enumerate}
    
    \item $U$ is a nil-regular open subset of $C$ such that the immersion $U\rar C$ is affine,
    \item for all $\eta\in \Gamma$, $M_\eta$ is an $\ell$-adic representation of $\pi_1^{\et}(U_\eta,\overline{\eta})$,
    \item for all $x\in C\setminus U$, $M_x$ a complex in $\mc{D}^b_{\Z_\ell-\mathrm{perf}}(\Sh((BG_{k(x)})_{\mathrm{pro\acute{e}t}},\Z_\ell))$ placed in degrees $[0,1]$,
    \item for all $x\in C\setminus U$, $f_x$ is an element of $$\Hom_{{\mathrm{D}}  (\Rep(G_{k(x)},\Z_\ell))}\left(M_x,\bigoplus\limits_{\nu(y)=x}\mathrm{Ind}_{G_{k(y)}}^{G_{k(x)}}\left( \partial_y \left[\phi_y^*(M_{\eta(y)})\right]\right)\right)$$

such that $\Hl^0(f_x)$ is injective.
\end{enumerate}

\item An element of $\Hom_{\mathrm{P}(C,\Z_\ell)}\left((U,(M_\eta),(M_x),(f_x)),(V,(N_\eta),(N_x),(g_x))\right)$ is a couple of the form $((\Phi_\eta)_{\eta \in \Gamma},(\Phi_x)_{x \in (C\setminus U)\cap (C\setminus V)})$ where $\Phi_\eta\colon M_\eta\rar N_\eta$ is a map of representations of $G_{K_\eta}$, where $\Phi_x\colon M_x \rar N_x$ is a map in $\mc{D}(\Sh((BG_{k(x)})_{\mathrm{pro\acute{e}t}},\Z_\ell))$ and such that the diagram:

$$\begin{tikzcd} M_x \ar[r,"f_x"]\ar[d,"\Phi_x"]& \bigoplus\limits_{\nu(y)=x}\mathrm{Ind}_{G_{k(y)}}^{G_{k(x)}}\left(  \partial_y\left[\phi_y^*(M_{\eta(y)})\right]\right) \ar[d,"\bigoplus\limits_{\nu(y)=x}\mathrm{Ind}_{G_{k(y)}}^{G_{k(x)}}\left( \partial_y\left(\phi_y^*(\Phi_{\eta(y)})\right)\right)"]\\
N_x \ar[r,"g_x"]& \bigoplus\limits_{\nu(y)=x}\mathrm{Ind}_{G_{k(y)}}^{G_{k(x)}}\left( \partial_y\left[\phi_y^*(N_{\eta(y)})\right]\right)
\end{tikzcd}$$
is commutative.
\end{itemize} 
\end{definition}

Our discussion then yields the following proposition.
\begin{proposition}Let $C$ be an excellent $1$-dimensional scheme and $\ell$ be a prime number invertible on $C$. Then, the category $\mathrm{Perv}(C,\Z_\ell)$ is equivalent to $\mathrm{P}(C,\Z_\ell)$.
\end{proposition}

\subsection{Artin perverse sheaves over schemes of dimension \texorpdfstring{$1$}{1}}
Let $M$ be an Artin perverse sheaf over an excellent $1$-dimensional base scheme $S$. Recall that there is a stratification $\mc{S}$ of $S$ such that for all $T\in \mc{S}$, $M|_T$ is smooth Artin. We want to recover $M$ from the $M|_T$ and additional data. This is possible in the case when $S$ is $1$-dimensional using the method of the previous paragraph.

We again take the convention that $\delta(S)=1$. 

\begin{definition}\label{P^A(S)} Let $S$ be a $1$-dimensional excellent scheme. Let $\ell$ be a prime number invertible on $S$. Let $\Gamma$ be the set of generic points of $S$. We define the full subcategory $\mathrm{P}^A(S,\Z_\ell)$ of $\mathrm{P}(S,\Z_\ell)$ as the subcategory made of those quadruple $(U,(M_\eta)_{\eta\in \Gamma},(M_x)_{x\in C \setminus U},(f_x)_{x\in C\setminus U})$ such that for all $\eta \in \Gamma$, the representation $M_{\eta}$ is of Artin origin and for all $x \in S \setminus U$, the representations $\Hl^0(M_x)$ and $\Hl^1(M_x)$ are of Artin origin.
\end{definition}

\begin{proposition} Let $S$ be an excellent $1$-dimensional scheme and $\ell$ be a prime number invertible on $S$. Then, $\mathrm{Perv}^A(S,\Z_\ell)$ is equivalent to $\mathrm{P}^A(S,\Z_\ell)$.
\end{proposition}

\begin{remark} It is possible that the nearby and vanishing cycle preserve the Artin condition over $1$-dimensional schemes which would give a description of Artin perverse sheaves which is more in the spirit of \cite{howtoglue}. In the case of Artin homotopy perverse sheaves, it may even be possible to have a similar description for every scheme of finite type over a finite field by using the Artin truncation of nearby and vanishing cycles.
\end{remark}

\begin{example}Let $k$ be an algebraically closed field of characteristic $0$. 
A perverse Artin $\ell$-adic sheaf over $\mb{P}^1_k$ is a quadruple $(\mb{P}^1_k\setminus F,M,(M_x)_{x\in F},(f_x)_{x\in F})$ where $F$ is a finite set of points of $\mb{P}^1_k$, where $M$ is a representation of Artin origin of $\pi_1^{\et}(\mb{P}^1_k\setminus F)$, where $M_x$ is a perfect complex of $\Z_\ell$-modules placed in degrees $[0,1]$ and where $$f_x\colon M_x\rar \partial_x[\phi_x^*(M)]$$ is a map in the category $\mc{D}(\Z_\ell)$ such that $\Hl^0(f_x)$ is injective.

We can always assume that the point at infinity is in $F$. Write $F=\{\infty\}\sqcup F'$. Let $m=|F'|$.

Let $x\in F'$, then, the field $k((X-x))$ is the completion of the field $k(X)$ of regular functions on $\mb{P}^1_k$ with respect to the valuation defined by $x$.

On the other hand, the field $k((1/X))$ is the completion of the field $k(X)$ with respect to the valuation of the point at infinity.

Furthermore, recall that the absolute Galois group of the field $k((X))$ is $\hat{\Z}$ since $k$ is algebraically closed and of characteristic $0$.

The étale fundamental group  of the scheme $\mb{P}^1_k\setminus F$ is the profinite completion of the free group over $F'$; we denote by $g_x$ for $x\in F'$ its topological generators. The map $$G_{k((X-x))}\rar \pi_1^{\et}(\mb{P}^1_k\setminus F)$$ is the only continuous map that sends the topological generator of $\hat{\Z}$ to $g_x$. Moreover, the map $$G_{k((1/X))}\rar \pi_1^{\et}(\mb{P}^1_k\setminus F)$$ is the map that sends the topological generator to a certain product $(g_{x_1}\cdot \cdots \cdot g_{x_m})^{-1}$ where $F'=\{x_1,\ldots,x_m\}$.

Now, the action of $\pi_1^{\et}(\mb{P}^1_k\setminus F)$ on $M$ factors through a finite quotient $G$ and is therefore uniquely determined by the representation of a finite group with $m$ marked generators which are the images of the $g_x$ for $x\in F'$.

Finally, let $$\partial\colon\Rep(\hat{\Z},\Z_\ell)\rar \mc{D}^b(\Z_\ell)$$ be the functor that sends a continuous  $\Z_\ell[\hat{\Z}]$ module $M$ to the derived module of fixed points under the action of the topological generator.

A perverse Artin $\ell$-adic sheaf over $\mb{P}^1_k$ is equivalent to the following data:
\begin{itemize}\item A finite number of distinct points $x_1,\ldots,x_m$ of $\mb{A}^1_k$ (\textit{i.e} of elements of $k$).
\item A finite group $G$ generated by $m$ elements $g_1,\ldots,g_m$; let $\phi_i\colon\hat{\Z}\rar G$ be the map given by $g_i$ and $\phi_\infty\colon\hat{\Z}\rar G$ be the map given by $g_\infty=(g_1\cdot \cdots \cdot g_m)^{-1}$.
\item A $\Z_\ell$-linear representation $M$ of $G$ which is finite as a $\Z_\ell$-module.
\item Perfect complexes of $\Z_\ell$-modules $M_1,\ldots,M_m,M_\infty$ placed in degrees $[0,1]$.
\item Maps $f_i\colon M_i\rar \partial(\phi_i^*M)$ for $i=1,\cdots, m, \infty$, such that the maps $$\Hl^0(f_i)\colon \Hl^0(M_i)\rar M^{g_i}$$ are injective.
\end{itemize}
\end{example}
\begin{example} Let $S=\Spec(\Z_p)$. A perverse Artin $\ell$-adic sheaf over $S$ is given by a triple $(M,N,f)$ where $M$ is a representation of Artin origin of $G_{\mb{Q}_p}$, where $N$ is a complex in $\mc{D}^b_{\Z_\ell-\mathrm{perf}}(\Sh((B\hat{\Z})_{\mathrm{pro\acute{e}t}},\Z_\ell))$ placed in degrees $[0,1]$ with cohomology of Artin origin and where $f\colon N \rar \partial(M)$, where $\partial$ is the map $$\mc{D}^b_{\Z_\ell-\mathrm{perf}}(\Sh((BG_{\Q_p)})_{\mathrm{pro\acute{e}t}},\Z_\ell))\rar \mc{D}^b_{\Z_\ell-\mathrm{perf}}(\Sh((B\hat{\Z})_{\mathrm{pro\acute{e}t}},\Z_\ell))$$ induced by the map $M\mapsto M^{G_0}$ and $\Hl^0(f)$ is injective.
\end{example}

\begin{example}Let $k$ be an algebraically closed field of characteristic $0$. Let $A$ be the localization of the ring $k[X,Y]/(XY)$ at the prime ideal $(X,Y)$. It is the local ring at the intersection point of two lines in $\mb{P}^2_k$.

Let $S=\Spec(A)$. The scheme $S$ has two generic points of residue field $k(X)$ and the residue field at the only closed point is $k$. the normalization of $S$ is given by the scheme $$\Spec(k[X]_{(X)}\times k[Y]_{(Y)}).$$

Let $\partial\colon \Rep(\hat{\Z},\Z_\ell)\rar \mc{D}^b(\Z_\ell)$ be the derived functor of fixed points and let $$\phi\colon\hat{\Z}=G_{k((X))}\rar G_{k(X)}$$ be the inclusion.

Now, an Artin perverse sheaf over $S$ is given by a quadruple $(M_1,M_2,N,f)$ where $M_1$ and $M_2$ are representations of Artin origin of $G_{k(X)}$, where $N$ is a complex of $\Z_\ell$-modules placed in degrees $[0,1]$ and where $f\colon N\rar \partial(\phi^*(M_1\oplus M_2))$ is such that $\Hl^0(f)$ is injective.
\end{example}

\begin{example} Let $S=\Spec(\Z[\sqrt{5}]_{(2)})$. This scheme has two points: its generic point has residue field $\Q(\sqrt{5})$ and its closed point has residue field $\mb{F}_2$. The normalization of $S$ is $\Spec\left(\Z\left[\frac{1+\sqrt{5}}{2}\right]_{(2)}\right)$ which is a discrete valuation ring. The residue field of its closed point is the field $\mb{F}_4$ with $4$ elements. The completion of $\Q(\sqrt{5})$ with respect to the valuation defined by $2$ is the field $\Q_2(\sqrt{5})$.

Let $\partial$ be the map $$\mc{D}^b_{\Z_\ell-\mathrm{perf}}(\Sh((BG_{\Q(\sqrt{5})})_{\mathrm{pro\acute{e}t}},\Z_\ell))\rar \mc{D}^b_{\Z_\ell-\mathrm{perf}}(\Sh((B\hat{\Z})_{\mathrm{pro\acute{e}t}},\Z_\ell))$$ induced by the map $M\mapsto M^{G_0}$.

Let $\phi\colon\hat{\Z}\rar \hat{\Z}$ be the multiplication by $2$ map. The map $\phi^*$ has right adjoint functor
$$\In_\phi \colon\Sh((B\hat{\Z})_{\mathrm{pro\acute{e}t}},\Z_\ell)\rar \Sh((B\hat{\Z})_{\mathrm{pro\acute{e}t}},\Z_\ell).$$

Thus an Artin perverse sheaf over $S$ is given by a triple $(M,N,f)$ where $M$ is a representation of Artin origin of $G_{\Q(\sqrt{5})}$, where $N$ is a complex which belongs to $\mc{D}^b_{\Z_\ell-\mathrm{perf}}(\Sh((B\hat{\Z})_{\mathrm{pro\acute{e}t}},\Z_\ell))$ and is placed in degrees $[0,1]$ with cohomology of Artin origin and where $$f\colon N\rar \In(\partial(M))$$ is such that $\Hl^0(f)$ is injective.
\end{example}

\newpage
\bibliographystyle{alpha}
\bibliography{biblio.bib}
\end{document}